\documentclass[aap]{imsart}
\RequirePackage[OT1]{fontenc}
\RequirePackage{amsthm,amsmath}
\RequirePackage[numbers]{natbib}
\RequirePackage[colorlinks,citecolor=blue,urlcolor=blue]{hyperref}

\usepackage{amsfonts,amssymb,amsbsy,amsmath,latexsym,natbib,graphicx}
\usepackage{psfrag,epsfig,epsf} 
\usepackage{comment}
\usepackage{verbatim}
\usepackage{mathrsfs}

\newcommand{\PP}{\ensuremath{\operatorname{P}}}

\newcommand{\EE}{\ensuremath{\operatorname{E}}}

\newcommand{\LL}{\mathscr{L}}

\renewcommand{\phi}{\varphi}
\renewcommand{\epsilon}{\varepsilon}

\startlocaldefs
\numberwithin{equation}{section}
\theoremstyle{plain}
\newtheorem{theo}{Theorem}[section]
\newtheorem{lem}{Lemma}[section]
\newtheorem{defi}{Definition}[section]
\newtheorem{rmk}{Remark}[section]
\newtheorem{prop}{Proposition}[section]
\newtheorem{cor}{Corollary}

\endlocaldefs

\newcommand{\tr}{\mbox{\rm Tr}}

\newcommand{\NN}{\mathbb{N}}
\newcommand{\QQ}{\mathbb{Q}}
\newcommand{\RR}{\mathbb{R}}

\newcommand{\ZZ}{\mathbb{Z}}

\newcommand{\Aa}{ {\cal A }}
\newcommand{\Ba}{ {\cal B }}
\newcommand{\Ca}{ {\cal C }}

\newcommand{\Ka}{ {\cal K }}

\newcommand{\Va}{ {\cal V }}

\newcommand{\Fa}{ {\cal F }}

\newcommand{\Qa}{ {\cal Q }}

\newcommand{\Ma}{ {\cal M }}
\newcommand{\Ta}{ {\cal T}}
\newcommand{\Ha}{ {\cal H }}

\newcommand{\Pa}{ {\cal P }}

\newcommand{\point}{\mbox{\LARGE .}}
\newcommand{\cqfd}{\hfill\blbx \\}
\def\blbx{\hbox{\vrule height 5pt width 5pt depth 0pt}\medskip}

\def \PP{\mathbb{P}}
\def \RR{\mathbb{R}}

\def \EE{\mathbb{E}}

\def \LL{\mathbb{L}}
\def \ZZ{\mathbb{Z}}

\def \BB{\mathbb{B}}

\newcommand{\vertiii}[1]{{\left\vert\kern-0.25ex\left\vert\kern-0.25ex\left\vert #1 
    \right\vert\kern-0.25ex\right\vert\kern-0.25ex\right\vert}}

\begin{document}

\begin{frontmatter}
\title{On the Stability of Positive Semigroups}
\runtitle{Stability of Positive Semigroups}

\begin{aug}
  \author[A]{\fnms{Pierre}  \snm{Del Moral}\corref{}\ead[label=e1]{pierre.del-moral@inria.fr}\ead[label=u3,url]{https://people.bordeaux.inria.fr/pierre.delmoral/research.html}},
 \author[A]{\fnms{Emma}  \snm{Horton}\corref{}\ead[label=e3]{emma.horton@inria.fr}\ead[label=u4,url]{https://sites.google.com/view/emmahorton/home}}
  \and
  \author[B]{\fnms{Ajay} \snm{Jasra}\ead[label=e2]{ajay.jasra@kaust.edu.sa} \ead[label=u1,url]{https://sites.google.com/site/ajayjasra0/home}
 }

  \thankstext{t2}{Corresponding author.}

  \runauthor{Del Moral et al.}


  \address[A]{Centre de Recherche Inria,  
Bordeaux Sud-Ouest
Talence, 33405, FR,
          \printead{e1},\\
          \printead{e3}
}

\address[B]{Applied Mathematics and Computational Science Program,  \\Computer, Electrical and Mathematical Sciences and Engineering Division,\\ King Abdullah University of Science \& Technology, \\
Building 1, Thuwal, 23955-6900, KSA,
          \printead{e2} 
 }

\end{aug}

\medskip

 \begin{abstract}
  The stability and contraction properties of positive integral semigroups on Polish spaces are investigated. 
Our novel analysis is based on the extension of $V$-norm 
contraction methods, associated to functionally weighted Banach spaces for Markov semigroups, to positive semigroups. This methodology is applied to a general class of positive and
possibly time-inhomogeneous bounded integral semigroups and their normalised versions.  The spectral theorems that we develop are an extension of  Perron-Frobenius and Krein-Rutman theorems for positive operators to a class of time-varying positive semigroups.  In the context of time-homogeneous models, the regularity conditions discussed in the present article appear to be necessary and sufficient condition for the existence of leading eigenvalues. We review and illustrate the impact of these results in the context of positive semigroups arising in transport theory, physics, mathematical biology and signal processing.
\end{abstract}

\begin{keyword}[class=MSC]
\kwd[Primary ]{47D08}
\kwd{47D06}
\kwd{47D07}
\kwd{47H07}
\kwd[; secondary ]{47B65, 37A30, 37M25, 60J25}
\end{keyword}

\begin{keyword}
\kwd{Positive semigroups} 
\kwd{Boltzmann-Gibbs transformations} 
\kwd{Contraction inequalities}
\kwd{Dobrushin's ergodic coefficient} 
\kwd{Spectral theorems}
\kwd{Foster-Lyapunov conditions} 
\end{keyword}

\end{frontmatter}

\section{Introduction}\label{sec:intro}

Positive semigroups arise in a variety of areas of applied mathematics, including nonlinear filtering, rare event analysis, branching processes, physics and molecular chemistry. In this article we will study the possibly time-inhomogeneous linear semigroup $Q_{s,t}$ and  its normalisation $\Phi_{s,t}$, where $0\leq s\leq t$ refers to a discrete or a continuous time parameter, which are formally introduced in section \ref{sec:descrip}. Their interpretation depends upon the application model area as we now describe.
\begin{enumerate}
\item In signal processing, the normalised semigroup $\Phi_{s,t}$ depends on a random observation process and describes the solution to the nonlinear filtering equations,  \textcolor{black}{ the semigroup $Q_{s,t}$ represents the evolution of the unnormalised filters. In this context, the stability of the semigroup $\Phi_{s,t}$ ensures that the optimal filter forgets its initial condition.}
\item{In the context of killed absorption processes, $\Phi_{s,t}$ represents the evolution of the distribution of a given process conditioned on non-absorption (a.k.a.~the $Q$-process). In this context,  the fixed point probability measure $\eta_{\infty}=\Phi_{t}(\eta_{\infty})$ of time homogeneous semigroups $\Phi_t:=\Phi_{s,s+t}$ is sometimes called the Yaglom or quasi-invariant measure.}
\item The conditional Markov processes  discussed above also arise  in  risk analysis and rare event simulation. In this context, these semigroups represent the evolution of a conditional Markov process evolving in a rare event regime. 
 \item In quantum physics and molecular chemistry the  top of the spectrum of 
positive integral semigroups is related to ground state and free energy computations of 
Schr\"odinger operators and Feynman-Kac semigroups (see for instance~\cite{dm-sch,dm-horton-21}). 
\item In the dynamic population literature, the semigroups $Q_{s,t}$ represent the evolution of the first moment of a spatial branching process. In this context, theses many-to-one formulae are expressed in terms of Feynman-Kac semigroups connecting the free evolution of a single individual with the killing and the branching rates potential functions.
\end{enumerate}
 The details of these application areas are considered in~\cite{dm-04,dm-2000,dmpenev} and the relevant references therein. The  spectral objects discussed above are naturally related to the analysis of quasi-compact operators and Fredholm integral equations, see for instance~\cite{atkinson,du,sloan,hennion-2007,nussbaum}, as well as in large deviations principles associated with the occupation measures and related additive functional of Markov processes~\cite{dembo,donsker,feng,liming-W}.

\subsection{Description of the models}\label{sec:descrip}
Let $\Ba(E)$ be the algebra  of locally bounded  measurable functions on 
a  Polish space $E$ (that is separable completely metrizable topological space).
We denote by $\Ba_b(E)\subset \Ba(E)$ the sub-algebra  of bounded measurable functions. With a slight abuse of notation, we denote by $0$ and $1$ the null and unit scalars as well as the null and unit function on $E$. 
We denote by $\Ma_b(E)$ the set of bounded signed measures on  $E$. Also let $\Ca_b(E)\subset\Ba_b(E)$ be the sub-algebra of continuous and bounded functions, and by
 $\Pa(E)\subset \Ma_b(E)$ be the subset of probability measures on $E$.

Let $\Ta=\RR_+:=[0,\infty[$ or $\Ta=\NN$ be the continuous or discrete time space, respectively. Consider a collection of positive integral operators $Q_{s,t}:f\mapsto Q_{s,t}(f)$  from $\Ba_{b}(E)$ into $\Ba_b(E)$, indexed by parameters $s,t\in\Ta$ with $s\leq t$, and satisfying for any $s,u,t\in \Ta$, $s\leq  u\leq t$,  the  semigroup property 
\begin{equation}\label{def-Q-s-t-intro}
 Q_{s,u} Q_{u,t}= Q_{s,t}\quad \mbox{\rm with}\quad Q_{s,s}=I.
\end{equation}
The right action  $Q_{s,t}:f\mapsto Q_{s,t}(f)$ and dual left action $\mu \Ma_b(E)\mapsto \mu Q_{s,t}\in \Ma_b(E)$ of and integral operator  $Q_{s,t}$ are  defined in classical measure theoretic notation in (\ref{mu-f-Q-f}). Assume that $Q_{s,t}(1)>0$ for any $s\leq t$.

 For any measure  $\eta_s\in \Pa(E)$  we let   $\Phi_{s,t}(\eta_s)\in  \Pa(E)$ be the normalised distribution defined for any $f\in\Ba_{b}(E)$ by the formula
\begin{equation}\label{def-Phi-s-t-intro}
\Phi_{s,t}(\eta_s)(f):={\eta_sQ_{s,t}(f)}/{\eta_sQ_{s,t}(1)}.
\end{equation}
The mapping $\Phi_{s,t}$ is a well defined nonlinear map from $\Pa(E)$ into $\Pa(E)$ satisfying for any $s,u,t\in\Ta$ with $s\leq u\leq t$ the semigroup property
$$
\Phi_{s,t}=\Phi_{u,t}\circ\Phi_{s,u}\quad \mbox{\rm with}\quad\Phi_{s,s}(\mu)=\mu.
$$
Unless otherwise stated, all the semigroups discussed in this article are  indexed by conformal indices $s\leq t$ in the set $\Ta$. To avoid repetition, we often write $Q_{s,t}$ 
and $\Phi_{s,t}$ without specifying the order $s\leq t$ of the indices $s,t\in\Ta$. 
For time homogeneous models we use the notation
$$
(\Phi_t,Q_t):=(\Phi_{0,t},Q_{0,t}).
$$

In contrast with conventional Hilbert space  approaches to the stability of reversible Markov semigroups (cf. for instance~\cite{brooks,diaconis-saloff,diaconis-saloff-2,konto-Meyn}), the analysis of time varying models of the form (\ref{def-Q-s-t-intro}) and (\ref{def-Phi-s-t-intro}) does not rely on a particular reversible measure. The framework of weighted spaces and $V$-norms considered in the article is a  natural but non-unique framework to analyse  time-varying positive semigroups.

\subsection{Literature review}

In order to guide the reader through the vast array of stability analysis results developed in these different disciplines, we give a brief overview of the literature in these fields. We also provide some precise reference points to aide with the navigation between applications.  

A unifying point of interest in the above applications is the study of the stability of the afore-mentioned semigroups.
In the context of dynamic populations, the long time behavior of branching processes certainly goes back to the end of the 1940s with the pioneering work  of Yaglom~\cite{yaglom} on Galton-Watson processes.
 The stability analysis of time homogenous birth-and-death processes with absorption on finite or countable spaces dates back to the 1950s-1960s  with the pioneering works~\cite{karlin,seneta} and the later developments~\cite{cavender,darroch-1,darroch-2,good,vandooorn}. Sufficient conditions ensuring the existence of a quasi-invariant measure are also developed in~\cite{ferrari,lasserre}. 
 Powerful spectral and $h$-process techniques are also developed in~\cite{cattiaux-1,cattiaux-2,collet-martinez,dm-horton-21,gong,hening,pinsky}. 
 Non-asymptotic spectral techniques that apply to possibly transient and unstable mutation linear diffusions that do not necessarily have a gradient form and with a quadratic absorption rate are discussed in the more recent work~\cite{dm-horton-21}. 

One of the first well-founded results on the long time behavior of the nonlinear filtering equation is the seminal article by Ocone and Pardoux in the mid-1990s~\cite{ocone-pardoux}. In this work the authors show that the optimal filter forgets the initial condition without giving a non-asymptotic rate. This latter analysis is critical in the study of, for instance,
numerical algorithms such as the particle filter. 
The stability of the nonlinear filtering equation is also related to the Lyapunov spectrum and the asymptotic  properties of products of random positive matrices~\cite{bougerol-1,bougerol-2,kesten,guivarch,hennion,kaijser,seneta}. In contrast with positive semigroups arising in physics and biology, the stability analysis of the nonlinear filtering equation involves the study of sophisticated stochastic  semigroups that depend on partial and noisy observations.

The systematic, non-asymptotic stability analysis of non-homogeneous sub-Marko\-vian and Feynman-Kac semigroups on general state spaces has also been considered at the end of the 1990s, mainly in nonlinear filtering theory and rare event analysis. Several techniques have been adopted. 
Hilbert metric and robustification techniques, based on the seminal article by 
Birkhoff~\cite{birkoff} was used in~\cite{atar, amarjit, daprato, klepsyna,legland-oudjane,legland-oudjane-2,legland-mevel-3,mcdonald,oljaca,oudjane-rubenthaler}. On the other hand, Dobrushin's ergodic coefficients, based on the pioneering articles by 
Dobrushin~\cite{dobrushin-1}, were used in~\cite{dm-04,dm-2000,dg-ihp,dg-cras,legland-mevel,legland-mevel-2,legland-mevel-3,mcdonald}. We also mention bounded Lipschitz distance techniques~\cite{legland-mevel,legland-mevel-2} and relative entropy-like criteria~\cite{clark,dm-04,dm-miclo-ledoux,legland-mevel-99}.  {Local Doeblin minorisation conditions applying to non-compact spaces,} including Foster-Lyapunov approaches and coupling, that apply to ergodic signal-obser\-vation filtering problems, including stable Gaussian-linear filtering models, are also discussed in the series of articles~\cite{douc-moulines,douc-moulines-ritov,gerber,ajay,ajay-arnaud-08,whiteley,whiteley-3}. We also refer the reader to~\cite{vanhandel} for related Lipschitz norm techniques  as well as the more recent articles~\cite{bansaye-2, benaim-champagnat-ocafrain,champagnat-4,ferre,guillin} in the context of  positivity preserving operators arising in particle absorption models and quasi-invariant measure literature.
Functional inequalities, including  Poincar\'e inequalities and Bakry-Emery criteria approaches are discussed in~\cite{ocafrain},  quasi-compactness Lyapunov criteria are also discussed in the recent article~\cite{benaim-champagnat-ocafrain}.
Spectral techniques, drift conditions and Wasserstein norm approaches for time-homogenous models are also discussed respectively in~\cite{dm-horton-21,dm-jasra-18,dm-sch,klepsyna-2,whiteley-2}. Further, two-sided minorisation conditions are discussed in \cite{chigansky-1,chigansky-2,dm-04,dm-13,dm-2000}, and truncation techniques are presented in~\cite{crisan-heine,ferre,heine,whiteley}. 

General asymptotic stability results are also provided in~\cite{vanhandel-2}. Stabilising changes of Feynman-Kac measures and related importance-sampling and $h$-process techniques that apply to possibly unstable killed processes on unbounded domains are also discussed in~\cite{bansaye-2,ferre,rousset}, see also~\cite{dm-horton-21} and well as section B.1 in~\cite{bishop-19} and in section 7.1 in~\cite{dm-horton-21-2} in the context of discrete time models. 
The stability of the nonlinear filtering equations with deterministic hyperbolic signals is also developed in the recent article~\cite{oljaca}.

 Despite the numerous references given in the introduction, a complete literature review is not possible. For a more detailed discussion on this subject, we refer the reader to~\cite{amarjit-2,chigansky-3} for a review on the asymptotic stability of nonlinear filters, to the bibliographies and reviews~\cite{collet-martinez-1,meleard-villemonais,pollett,vandooorn-2} on the theory of quasi-stationary distributions, the articles~\cite{dm-2000,ferre} as well as the bibliography in~\cite{dm-horton-21} in quantum physics, and the books~\cite{dm-04,dm-13} for a detailed discussion on the long time stability of Feynman-Kac semigroups. 
 
The stability analysis of positive semigroups is also crucial in the convergence analysis of numerical approximations of Feynman-Kac semigroups, including the computation of the principal eigen-function and the leading eigenvalues in the context of time homogeneous models (cf.~for instance~\cite{dm-sch} as well as Theorem 2.11 and  Theorem 3.27 in~\cite{dm-2000} in the context of time varying semigroups). 
Despite their importance, the numerical implications are not covered in the present article. However, to guide the reader, we end this section with some references to the literature on mean-field particle methodologies currently used in this context. 
Mean field and genetic particle methodologies are discussed in the series of articles on Feynman-Kac semigroups arising in physics and nonlinear filtering~\cite{arnaudon-dm,assaraf,cloez-corujo,dm-04, dm-13, dm-doucet-04, dm-doucet-jasra-06, dm-2000-moran, dg-ihp, rousset, whiteley-4}, as well as in
~\cite{burdzy-1,burdzy-2} in the context of Dirichlet Laplacian and in ~\cite{asselah,asselah-2,assaraf,champagnat-8, cloez-2, cloez-3,ferrari, journel, villemonais-14} in the context of quasistationary measures.

Several pioneering articles from the mid-1950s by Rosenbluth and Rosenbluth~\cite{rosen-55} on sampling self-avoiding walks and another from the mid-1980s by Hetherington~\cite{hetherington-84} on reconfiguration Monte-Carlo methods are relevant to our work. These interacting Monte Carlo methodologies were further extended by Caffarel and his co-authors in the series of articles~\cite{caffarel-86,caffarel-88,caffarel-89}. See also Buonaura-Sorella~\cite{buonaura-98}, as well as the pedagogical introduction to quantum Monte Carlo by Caffarel-Assaraf~\cite{caffarel}.  Similar  heuristic boostrapping metho\-dologies were also used in the mid-1990s in nonlinear filtering~\cite{gordon,kitagawa-96,kitagawa-98} and to simulate long chain molecules~\cite{grassberger-1,grassberger-2}. See also \textcolor{black}{the go-with-the-winner methodology discussed in~\cite{aldous,grassberger-3} and the Fleming-Viot techniques presented in~\cite{burdzy-2}}.

In the context of discrete generation positive semigroups arising in nonlinear filtering and genetic algorithms, we refer the reader to  the first well-founded article~\cite{dm-96}. To the best of our knowledge the first articles discussing time-uniform propagation of chaos estimates seem to be the articles~\cite{dg-ihp,dg-cras}, followed by~\cite{dm-2000,dm-gen-01,dm-sch} and the book~\cite{dm-04}. \textcolor{black}{From a purely mathematical perspective, all of the genetic Monte Carlo methods discussed above can be seen as mean field particle interpretations of Feynman-Kac semigroups. In path space settings, the genealogical tree associated with these branching Monte Carlo methods allows one to compute Feynman-Kac path integrals and provide an unbiased estimate of unnormalised semigroups.}

Whilst all the Monte Carlo methods discussed above belong to the same class of genetic mutation-selection methodology to estimate Feynman-Kac semigroups, they are known under a variety of different names in the applied literature such as Sequential Monte Carlo methods, Feynman-Kac particle interpretations, Particle Filters, Cloning and Pruning as well as Bootstrapping techniques, Diffusion Monte Carlo, Population-Monte Carlo, Reconfiguration Monte Carlo, Moran and Fleming-Viot particle models, to name a few. Related reinforcement and self-interacting Markov chain methodologies and stochastic approximation techniques are presented in~\cite{dm-miclo-self-interacting,dm-miclo-self-interacting-2} and more recently  in~\cite{benaim-panlou-cloez,blanchet,mailler-villemonais}.
   
 \textcolor{black}{Most of the terminology encountered in the literature arises from the application domains as well as with the branching/genetic evolution of the Monte Carlo methodology.
  As underlined in~\cite{dm-2000-moran}, from a mathematical viewpoint, only the terminologies ``Fleming-Viot"  and "Moran" are misleading. The reasons are two-fold: firstly, the Moran particle model is a finite population model that converges as the number of particles tends to infinity towards a stochastic Fleming Viot superprocess~\cite{dawson,fleming} and secondly, the genetic noise arising in the limit requires a Moran finite population process with symmetric-selection jump rates. 
  In our context, the selection/killing-jump rates are far from being symmetric, the empirical measures of the finite population model are biased and the limiting Feynman-Kac semigroup, as the number of particles tends to infinity, is purely deterministic. }

\subsection{Statement of some main results}\label{contribution-statements-sec}

The main objective of this article is to review and further develop 
the stability analysis of positive semigroups  for a general and abstract class of  time inhomogeneous models, which, in general, are much more difficult to handle than their time homogeneous counterparts; this is because the operators may drastically change during the semigroup evolution.
We will also tackle the problem of non-compact state spaces.

We begin with an exposition of discrete time and time homogeneous models.
Consider a positive integral operator $Q$ and let $Q_{n+1}=Q_{n}Q=QQ_{n}$ be the associated discrete time semigroup  indexed by $n\in \NN$.   In this context we have
\begin{equation}\label{product-hom-intro}
\mu(Q_n(1))=\prod_{0\leq k< n}~\Phi_{k}(\mu)(g)\quad\mbox{\rm with}\quad g:=Q(1).
\end{equation}
When $E$ is compact, the  Schauder fixed-point theorem ensures the existence of an invariant measure
$$
\eta_{\infty}=\Phi_{n}(\eta_{\infty})\in \Pa(E)\quad\mbox{\rm and}\quad
\eta_{\infty}Q_n=e^{n\rho}~\eta_{\infty}\quad\mbox{\rm with}\quad\rho:=\log\eta_{\infty}(Q(1)).
$$
Note that the right-hand side assertion in the above display is a direct consequence of the fixed point equation and the product formula (\ref{product-hom-intro}).
For not necessarily compact spaces, we also quote the following abstract 
 theorem for discrete time Feller semigroups, which is of interest in its own right.

  \begin{theo}\label{krylov-bogo-theo}
Consider a positive integral operator $Q$ such that $Q(\Ca_b(E))\subset \Ca_b(E)$ and let $Q_n$ be the associated discrete time Feller semigroup  indexed by $n\in \NN$.  
The normalised  semigroup $\Phi_{n}$  has at least one invariant probability measure $\eta_{\infty}=\Phi_{n}(\eta_{\infty})$ if and only if  there exists some probability measure $\eta$ such that the sequence  of probability measures $\Phi_{n}(\eta)$ indexed by $n\in \NN$  is tight and we have
 \begin{equation}\label{tight-ref-cond}
\beta_{n}(\eta):=\Phi_{n}(\eta)(g)\in ]0,1]\longrightarrow_{n\rightarrow\infty} \beta_{\infty}(\eta)>0,
 \end{equation}
 with the function $g$ defined in (\ref{product-hom-intro}). 
In addition, whenever these conditions are met we have
$
\beta_{\infty}(\eta)=\eta_{\infty}(g)
$.
\end{theo}
The proof of Theorem~\ref{krylov-bogo-theo} is provided in section~\ref{krylov-bogo-theo-proof}.
Related \textcolor{black}{equivalent} conditions for the existence of invariant measures on locally compact spaces $E$ are discussed in~\cite{lasserre}, for models preserving continuous functions that tend to $0$ at $\infty$, see also~\cite{ferrari}.
Under the assumptions of Theorem~\ref{krylov-bogo-theo},  using (\ref{product-hom-intro}) 
 we check the product series formulae
\begin{eqnarray}
h_{n+1}(x)&:=&e^{-(n+1)\rho}~Q_{n+1}(1)(x)\label{def-hn}\\
&=&e^{-\rho}~Q(h_n)(x)=\prod_{0\leq k\leq n}~\left(1+\left(\Phi_{k}(\delta_x)(\overline{g})-\Phi_{k}(\eta_{\infty})(\overline{g})\right)\right)\nonumber
\end{eqnarray}
with the normalisation  $\overline{g}:=e^{-\rho} g=g/\eta_{\infty}(g)$ of the function $g$ defined in (\ref{product-hom-intro}). 
The above rather elementary formulae connect the convergence properties of the functions $h_n$ as $n\rightarrow\infty$ with the stability properties of the normalized measures $\Phi_{n}(\mu)$.

Recall that (see for instance chapter 7 in~\cite{knopp})
 the product in the above display is absolutely convergent if and only if  we have
\begin{equation}\label{ref-product-series-h}
\sum_{k\geq 0}~\left\vert \Phi_{k}(\delta_x)(\overline{g})-\Phi_{k}(\eta_{\infty})(\overline{g})\right\vert<\infty.
\end{equation}
In this situation, the collection of functions $h_n(x)$ converges pointwise as $n\rightarrow\infty$ to the measurable function defined by the product series 
\begin{equation}\label{h-series}
h(x):=\prod_{n\geq 0}~\left(1+\left(\Phi_{n}(\delta_x)(\overline{g})-\Phi_{n}(\eta_{\infty})(\overline{g})\right)\right).
\end{equation}

At this level of generality we cannot ensure that $h$ is a right eigenvalue of $Q$. To apply the dominated convergence theorem, we need to ensure that the sequence of functions $h_n$ is uniformly bounded. This property requires to calibrate with some precision the stability properties of the normalized semigroup $\Phi_n$.

This brief discussion already motivates the importance of the stability analysis of the normalised semigroups. Several classes of semigroups with an increasing level of regularity are presented in section~\ref{all-sg-sec}. 
All of these different models are all expressed in terms of the triangular array of semigroups defined in section~\ref{R-sg-sec}.
The first class of semigroups, termed $R$-semigroups, is discussed in section~\ref{def-R-sg-sec} and it relies on the contraction properties these semigroups with respect to the total variation distance (cf. for instance condition (\ref{ref-H1-2})). Section~\ref{sec:not_res-unif} provides several uniform  
exponential contraction theorems for this class of models with respect to the total variation norm.
The second class of semigroups, termed stable $V$-positive semigroups, is discussed in section~\ref{V-sg-sec}. These models rely on an regularity property with respect to some Lyapunov function $V$ discussed in section~\ref{uf-sg-sec}.

Section~\ref{sec:timeinhom} provides several uniform  
exponential contraction theorems for this class of models with respect to $V$-norms. 
Section~\ref{sec:timehom} illustrates the impact of these results in the context of time homogeneous semigroups. We shall see that $Q_t$ is a stable $V$-positive semigroup if and only if there exists  some leading eigenvalues and the semigroup of the process evolving in the ground state is a stable positive semigroup.

These different classes of positive semigroups are not based on any type of
 absolute continuity condition, but on different types of regularity properties of integral operator regularity and local contraction properties of the pivotal  triangular array of semigroups defined in section~\ref{R-sg-sec}.

In this section, to avoid the details of abstract regularity conditions we have chosen to review and state some of our main results in the context 
absolute continuity with an increasing level of regularity. We also illustrate these regularity conditions with some elementary examples.
We recall that a semigroup of positive integral operators (\ref{def-Q-s-t-intro})  is said to be absolutely continuous as soon as for some $\tau>0$ and any $t\in \Ta$ we have
\begin{equation}\label{ref-Q-chi}
Q_{t,t+\tau}(x,dy)=q_{t,t+\tau}(x,y)~\nu_{\tau}(dy)
\end{equation}
for some density function $q_{t,t+\tau}$ on $E^2$ with respect to a Radon positive measure $\nu_{\tau}$ on $E$. 

 \subsubsection{Total variation stability theorems}

Consider the following condition:

{\em $(\Aa)$ There exists some $\tau>0$ such that density $q_{t,t+\tau}$ is uniformly positive for some parameter $\tau>0$; that is, we have that
\begin{equation}\label{hyp-intro}
0<\iota^-(\tau):=\inf q_{t,t+\tau}(x,y)\leq \iota^-(\tau):=\sup q_{t,t+\tau}(x,y)<\infty
\end{equation}
where the infimum and the supremum are taken over all  space-time indices $((x,y),t)\in (E^2\times \Ta)$. 
For continuous time models, we assume that for any $\mu\in \Pa(E)$ we have
\begin{equation}\label{hyp-d-cont-A}
\inf\mu(Q_{t,t+\epsilon}(1))>0\quad \mbox{and}\quad \pi_{\tau}:=\sup\Vert Q_{t,t+\epsilon}(1)\Vert<\infty
\end{equation}
where the infimum is taken over all continuous time indices $t\geq 0$ and any $\epsilon \in [0,\tau[$.}

In the context of discrete time semigroups, there is no loss of generality to assume that $\tau=1$ (cf.~section~\ref{sec-d-2-continuous}). 
For continuous time sub-Markovian semigroups, the right-hand side ~condition in (\ref{hyp-d-cont-A}) is automatically met with $\pi_r\leq 1$. 
Whenever $Q_{s,t}(1)\geq Q_{s,u}(1)$ for any $u\geq t$, the left-hand side condition in (\ref{hyp-intro}) ensures that $\inf\mu(Q_{t,t+\epsilon}(1))\geq \iota^-(\tau)>0$; and for time homogenous models the right-hand side ~condition in (\ref{hyp-intro}) is met as soon as
$q_{\tau}$ is  bounded. 

For continuous time models, condition (\ref{hyp-intro}) is automatically satisfied for time homogeneous jump
elliptic diffusions on compact manifolds $S$ with a bounded jump rate, see for instance the pioneering work of Aronson~\cite{aronson}, Nash~\cite{nash} and Varopoulos~\cite{varopoulos} on Gaussian estimates for heat kernels on manifolds. It is also met for uniformly elliptic diffusions on the compact closure $E=\overline{D}$ of  some bounded open domain $D$ in $\RR^n$ with an oblique reflection on some smooth boundary $\partial D$, see for instance~\cite{cattiaux-92}. Condition (\ref{hyp-intro}) is also preserved by conditioning {\em any} of the stochastic processes discussed above by the non-absorption event, as soon as the killing rate is bounded on the state space $E$, see for instance~\cite{dm-sch-2} for a more thorough discussion on this class of sub-Markovian semigroups on compact manifolds.
Further references on uniformly positive discrete time semigroups can be found in \cite{dg-ihp,dm-jasra-18,dm-sch-2} and the books~\cite{dm-04,dm-13,dm-2000}.

The next theorem is a direct consequence of the uniform estimates (\ref{def-q-mu}) and a rather well known uniform contraction theorem, Theorem~\ref{theo-1}, which is valid for not necessarily absolutely continuous semigroups under weaker conditions on the triangular array of semigroups discussed above~\cite{dm-2000,dg-ihp,dg-cras}.

\begin{theo}\label{intro-th1}
Let $Q_{s,t}$ be an absolutely continuous  semigroup (\ref{ref-Q-chi}) satisfying condition $(\Aa)$ for some parameter $\tau>0$. In this situation, There exist constants $a<\infty$ and $b>0$ such that for any  $s\leq t$ and any $\mu_1,\mu_2\in\Pa(E)$ we have the uniform stability estimate
\begin{equation}\label{beta-sup-intro}
\Vert \Phi_{s,t}(\mu_1)-\Phi_{s,t}(\mu_2)\Vert_{\tiny tv}\leq 
a~e^{-b(t-s)},
\end{equation}
with the total variation norm $\Vert \point \Vert_{\tiny tv}$ on $\Ma_b(E)$ defined in (\ref{ref-tv-distance}).
In addition, there exists a constant $c(\mu_1)$ and $c(\mu_2)<\infty$ such that for any $s\leq t$ we have the local Lipschitz estimate
$$
\Vert \Phi_{s,t}(\mu_1)-\Phi_{s,t}(\mu_2)\Vert_{\tiny tv}\leq (c(\mu_1)\wedge c(\mu_2))~e^{-b(t-s)}~\Vert \mu_1-\mu_2\Vert_{\tiny tv}.
$$
\end{theo}

Observe that positive semigroups $\Qa_{s,t}$ with continuous time indices $s\leq t\in \RR_+$ can be turned into discrete time models by setting $Q_{p,n}=\Qa_{p\tau,n\tau}$ for any $p\leq n\in \NN$ and some parameter $\tau>0$. The condition (\ref{hyp-d-cont-A}) is a technical condition only made for continuous time semigroups to ensures that the Lipschitz estimates stated in Theorem~\ref{intro-th1} holds for all continuous time indices.

 Returning to the discrete and time homogeneous semigroup $Q=Q_{t,t+1}$ discussed in (\ref{def-hn}) the uniform estimate (\ref{beta-sup-intro}) ensures that
$$
\sum_{n\geq 0}~\sup_{x\in E}\left\vert\Phi_{n}(\delta_x)(\overline{g})-\Phi_{n}(\eta_{\infty})(\overline{g})\right\vert<\infty.
$$ 
In this scenario, the collection of functions $h_n$ defined in (\ref{def-hn}) is uniformly bounded and $h_n$ converges pointwise as $n\rightarrow\infty$ to the function $h\in \Ba_b(E)$ defined in (\ref{h-series}). Applying the dominated convergence theorem, for any $n\in\NN$ we conclude that
$$
Q_n(h)=e^{\rho n}h.
$$
Similar infinite series representations of the ground state function for discrete time models are discussed in Section 3.3 in the article~\cite{berard}.

Following the above comments in the context of discrete or continuous time homogeneous semigroups, we present in   Section~\ref{sec:not_res-unif} 
a variety of results that follow almost immediately from the estimates obtained in Theorem~\ref{intro-th1}. 

These results include the existence of  an unique leading eigen-triple $(\rho,\eta_{\infty},h)\in (\RR\times\Pa(E)\times\Ba_b(E))$ of the positive semigroup; that is, for any $t\geq 0$ we have
\begin{equation}\label{def-eigen-triple-intro}
Q_t(h)=e^{\rho t}~h\quad \mbox{\rm and}\quad\eta_{\infty}Q_t=e^{\rho t}~\eta_{\infty}
\quad \mbox{\rm with}\quad\eta_{\infty}(h)=1.
\end{equation}
The leading eigen-function $h$ is sometimes called the ground state of the semigroup. 
Defining the finite rank (and hence compact) operator
\begin{equation}\label{def-Tt}
T~:~f\in \Ba_b(E)\mapsto T(f):=\frac{h}{\eta_{\infty}(h)}~\eta_{\infty}(f)\in \Ba_b(E),
\end{equation}
we also have the following extended and refined version of the Krein-Rutman theorem.
\begin{cor}\label{cor-intro-1-kr}
For any $t\in\Ta$,  we have the operator norm exponential decays
\begin{equation}\label{add-expo-intro}
\vertiii{e^{-\rho t}Q_{t}-T}\leq 2~a~e^{-b(t-s)}~\left(\Vert h\Vert/\eta_{\infty}(h)+~c(\eta_{\infty})^2\right),
\end{equation}
with the same parameters $(a,b,c(\eta_{\infty}))$ as in (\ref{lipschitz-inq}).
 \end{cor}
The details,  including a proof, of the above results can be found in section~\ref{sec:not_res-unif}. 
As shown at the end of section~\ref{sec:not_res-unif}, the estimate (\ref{add-expo-intro}) ensures the uniqueness of the eigenfunction $h$ (up to some constant) and that the essential spectral radius of $Q_t$ is strictly smaller than its spectral radius $r(Q)=e^{\rho}$. Thus, the operator $Q_t$ is quasi-compact and by a variant of the Krein-Rutman theorem (cf. for instance Theorem 1.1 in~\cite{du,nussbaum}), we recover the fact the existence of non-null eigenfunction $h$.
A brief review of quasi-compact operators is provided in 
section~\ref{quasi-compact-sec}. The quasi-compactness property of $Q_t$ is often based on compact localisation arguments and applying Arzela-Ascoli theorem when the semigroups have a continuous density.
The above corollary does not rely on these arguments and it also offers an exponential convergence rate that only depends on the stability properties of the semigroup $\Phi_t$.
 
Unfortunately,  (\ref{hyp-intro}) is rarely satisfied for non-compact state spaces. Some cases where it does hold are 
for reflected diffusions on smooth boundaries as well as for some particular classes of operators on non-compact spaces (including for ad-hoc truncated drift Gaussian transitions or Laplace transitions on non-compact spaces - see \cite[section 5]{dm-2000} and \cite[Exercise 3.5.2]{dm-04}).
 
  \subsubsection{$V$-norm stability theorems}
 
 Our starting point is to localise  (\ref{hyp-intro}) using a
some locally bounded  $V\geq 1$ with compact level sets $K_r=\{V\leq r\}\subset E$. 
In this notation, the $V$-localised version of   (\ref{hyp-intro}) is defined as follows:

{\em $(\Aa)_V:$ There exists some parameters $0<\tau\in \Ta$ and $r_1>1$ such that for any $r\geq r_1$ we have $\nu_{\tau}(K_r) >0$ as well as
 \begin{equation}\label{min-max-q-vr}
  0<\inf_{t\in \Ta}\inf_{K_r^2}q_{t,t+\tau}\leq
\sup_{t\in \Ta}\sup_{ K_r^2}q_{t,t+\tau}<\infty.
\end{equation}
For continuous time semigroups, we assume that for any $r\geq r_1$ there exists some $\overline{r}\geq r$ such that
\begin{equation}\label{discrete-2-continuous-intro-K}
\pi_\tau^-(K_{r,\overline{r}}):=\inf\inf_{x\in K_r} Q_{t,t+\epsilon}(1_{K_{\overline{r}}})(x)>0\quad\mbox{and}\quad \pi_{\tau}:=\sup\Vert Q_{t,t+\epsilon}(1)\Vert<\infty
\end{equation}
where the infimum is taken over all continuous time indices $t\geq 0$ and any $\epsilon \in [0,\tau]$.
} \\

For any $r\geq r_1$ observe that
\begin{equation}\label{pi-min}
(\ref{discrete-2-continuous-intro-K})\Longrightarrow
\pi_\tau^-(K_{r}):=\inf\inf_{K_r} Q_{t,t+\epsilon}(1)>0\end{equation}
where the infimum is taken over all continuous time indices $t\geq 0$ and any $\epsilon \in [0,\tau]$.

By (\ref{min-max-q-vr}) for any $r\geq r_0$ and $\overline{r}\geq r$ we have the uniform estimate
 {$$
\inf_{t\geq 0}\inf_{K_r}Q_{t,t+\tau}(1)\geq \inf_{t\geq 0}\inf_{K_r}Q_{t,t+\tau}(1_{K_{\overline{r}}})\geq \inf_{t\geq 0}\inf_{K_r}Q_{t,t+\tau}(1_{K_r})>0
$$}
Thus, for discrete time semigroups, condition (\ref{discrete-2-continuous-intro-K}) is automatically met. Recalling that positive semigroups $\Qa_{s,t}$ with continuous time indices $s\leq t\in \RR_+$ can be turned into discrete time models by setting $Q_{p,n}=\Qa_{p\tau,n\tau}$ for any $p\leq n\in \NN$, the condition (\ref{discrete-2-continuous-intro-K}) is a technical condition only made for continuous time semigroups to ensures that the Lipschitz estimates discussed in this section holds for all continuous time indices.

The condition (\ref{min-max-q-vr}) is rather flexible as we will now explain. 
For time homogeneous models on some open connected domain $E\subset\RR^d$ (with respect to the trace $\nu_{\tau}(dy)$ of the Lebesgue measure $dy$ on $E$) condition
(\ref{min-max-q-vr}) is clearly met as soon as 
$q_{\tau}$ is a bounded continuous positive function on $E^2$. To check condition $\nu_{\tau}(\{V\leq r\}) >0$, simply notice that any closed ball in $E$ (which has clearly positive Lebesgue measure) is included in some $r_1$-sub-level set of $V$, thus also included in all upper $r$-sub-level sets, with $r\geq r_1$.  
 
Absolutely continuous integral operators arise in a natural way in discrete time settings~\cite{dm-04,dm-2000,douc-moulines-ritov,whiteley} and in the analysis of continuous time elliptic diffusion absorption models~\cite{aronson,ferre-phd,ferre-ldp,stroock}. In connection to this, two-sided estimates for stable-like processes are provided in~\cite{bogdan,kopnova,song,wang}.
Two sided Gaussian estimates can also be obtained  for some classes of degenerate diffusion processes of rank 2, that is when the Poisson brackets of the first order span the whole space~\cite{konakov}. 
  This class of diffusions includes frictionless Hamiltonian kinetic models. Diffusion density estimates can be extended to sub-Markovian semigroups using the multiplicative functional methodology developed in~\cite{dm-sch-2}. 
  
 Whenever the trajectories of these diffusion flows, say  $t\mapsto X_t(x)$, where $x\in E$ is the initial position, are absorbed on the smooth boundary $\partial E$ of a open connected domain $E$, for any $\tau>0$ the densities $q_{\tau}(x,y)$ of the sub-Markovian semigroup $Q_{\tau}$  (with respect to the trace of the Lebesgue measure on $E$)  associated with the non absorption event are null at the boundary. Nevertheless, whenever these densities are positive and continuous on the open set $E^2$ for some $\tau>0$, they are uniformly positive and bounded on any compact subsets of $E$;  thus condition (\ref{min-max-q-vr}) is satisfied. In addition, whenever $T(x)$ stands for first exit time from $E$ and $T_r(x)$ the first exit time from the compact level set $K_r\subset E$ starting from $x\in K_r$, for any $\epsilon\in [0,\tau]$ and $\overline{r}>r$ we have the estimate
\begin{eqnarray*}
Q_{\epsilon}(1_{K_{\overline{r}}})(x)&:=&\EE\left(1_{X_{\epsilon}(x)\in K_{\overline{r}}}~1_{T(x)>\epsilon}\right)\geq \PP\left(T_{\overline{r}}(x)>\epsilon\right)\geq \PP\left(T_{\overline{r}}(x)>\tau\right)
\end{eqnarray*}

In this context, condition (\ref{discrete-2-continuous-intro-K}) is met as soon as $\inf_{x\in K_r}\PP\left(T_{\overline{r}}(x)>\tau\right)>0$.

To state our next main result, we need to introduce some additional terminology. 
Let $\Ca(E)\subset\Ba(E)$ the sub-algebra of continuous functions.
Let $\Ba_{V}(E)\subset \Ba(E)$ be the Banach space of measurable functions $f$ on $E$ with  $\Vert f/V\Vert<\infty$; and $\Ca_{V}(E)\subset \Ba_V(E)$ be the subspace of continuous functions. 

We also let
$
\Pa_V(E)
$ be the convex set of probability measures $\mu_i\in  \Pa(E)$ such that $\mu_i(V)<\infty$ with $i=1,2$ equipped with the operator $V$-norm
$$
\vertiii{\mu_1-\mu_2}_{V}:=\sup\{|(\mu_1-\mu_2)(f)|~:~\Vert f\Vert_V\leq 1\}.
$$
More generally, we denote by $\vertiii{\cdot}_V$ the operator norm defined in (\ref{def-op-norm}).
When $E$ is a $\sigma$-compact Polish space, we also denote by $\Ba_0(E)\subset \Ba_b(E)$ the sub-algebra of  locally lower bounded that vanish at infinity and by $\Ca_0(E):=\Ba_0(E)\cap\Ca(E)$ the sub-algebra of continuous functions. We also let $\Ba_{\infty}(E)\subset \Ba(E)$ be  the subalgebra of  locally bounded and uniformly positive functions $V$ that grow at infinity. We shall consider the subspaces
$$
\Ba_{0,V}(E):=\left\{f~\in \Ba(E)~:~\vert f\vert /V\in \Ba_0(E)\right\}\quad\mbox{\rm and}\quad
\Ca_{0,V}(E):=\Ba_{0,V}(E)\cap \Ca(E).
$$
We refer to section~\ref{sec-fa-not} for more precise definitions.

The next theorem is a direct consequence of Theorem \ref{theo-intro-2} which is
valid for not necessarily absolutely continuous semigroups under weaker conditions.

\begin{theo}\label{theo-2-intro}
Assume condition $(\Aa)_V$ is met for  some  $\tau>0$ and some $V\in \Ba_{\infty}(E)$. In addition, there exists some function $\Theta_{\tau}\in \Ba_0(E)$ such that for any $s<t$ and any positive function $f\in \Ba_V(E)$ we have
\begin{equation}\label{ref-V-theta}
 Q_{s,t}(f)\in \Ba_{0,V}(E)\quad \mbox{and}\quad Q_{s,s+\tau}(V)/V\leq \Theta_{\tau}. 
\end{equation}
For continuous time semigroups, we also assume that
\begin{equation}\label{discrete-2-continuous}
\pi_{\tau}(V):=\sup\Vert Q_{t,t+\epsilon}(V)/V\Vert<\infty
\end{equation}
where the supremum is taken over all continuous time indices $t\geq 0$ and any $\epsilon \in [0,\tau]$.
In this situation, for any $\mu_1,\mu_2\in \Pa_V(E)$ there exists a parameter $b>0$ and some finite constant $c(\mu_1,\mu_2)<\infty$ such that for any $s\leq t$ we have the local Lipschitz estimate
\begin{equation}\label{lipschitz-inq-V}
\vertiii{ \Phi_{s,t}(\mu_1)-\Phi_{s,t}(\mu_2)}_{\tiny V}\leq c(\mu_1,\mu_2)~e^{-b(t-s)}~\vertiii{ \mu_1-\mu_2}_{\tiny V}.
\end{equation}

\end{theo}

A more precise description of the parameters $(b,c(\mu_1,\mu_2))$ is provided in Theorem~\ref{theo-intro-2}. To the best of our knowledge, Theorem~\ref{theo-2-intro} and its extended version, Theorem~\ref{theo-intro-2}, have not been established elsewhere in the literature. 

It should be noted that the  left-hand side condition in (\ref{ref-V-theta}) ensures that
$Q_{s,t}(V)/V\in \Ba_0(E)$, for any $s<t$. Thus, the condition on the right-hand side in (\ref{ref-V-theta}) provides uniform control of the functions $Q_{s,s+\tau}(V)/V$ with respect to the time parameter $s\geq 0$. We also mention that left-hand side condition in (\ref{ref-V-theta})
is met as soon as $Q_{s,t}(V)/V\in \Ba_0(E)$ and $Q_{s,t}$ is a strong Feller semigroup, in the sense that for any $s<t$ we have $Q_{s,t}(\Ba_V(E))\subset\Ca_{V}(E)$. In this situation, for any positive function $f\in \Ba_V(E)$ and $s<t$ the function $Q_{s,t}(f)$ is positive and  continuous;  and thus locally lower bounded. In this situation, whenever $\Vert f\Vert_V\leq 1$, for any $s<t$ we have the comparison property
$$
Q_{s,t}(f) /V\leq Q_{s,t}(V)/V \in \Ba_0(E)\Longrightarrow Q_{s,t}(f)/V\in \Ba_{0}(E)\Longleftrightarrow Q_{s,t}(f)\in \Ca_{0,V}(E).
$$
 For time homogeneous models $Q_t:=Q_{s,s+t}$, the right-hand side~condition in (\ref{ref-V-theta}) becomes the Lyapunov condition
\begin{equation}\label{intro-QV-Theta}
Q_{\tau}(V)/V\leq \Theta_{\tau}\in\Ba_0(E).
\end{equation}
In the context of time homogeneous models, the strong Feller condition also ensure that the ground state eigen-function is continuous (cf. (\ref{ref-dominated-cv})).

When $Q_{s,t}=P_{s,t}$ is a Markovian semigroup $P_{s,t}$  on $\Ba_V(E)$, the semigroup  $\Phi_{s,t}(\mu)=\mu P_{s,t}$ is linear, and the right-hand side Lyapunov condition in (\ref{ref-V-theta}) ensures that $\mu P_{s,t}(V)$ is uniformly bounded with respect to the time parameters $s\leq t$ (cf.~Lemma~\ref{lemma-Nick} or the Lyapunov condition (\ref{ref-P-h-nh}) applied to the unit function $H=1$). In this context, the constant $c=c(\mu_1,\mu_2)$ in  (\ref{lipschitz-inq-V}) does not depend on the pair of measures $(\mu_1,\mu_2)$, and the estimate (\ref{lipschitz-inq-V}) reduces to the $V$-norm contraction of time varying Markov semigroups
$$
\vertiii{ \mu_1P_{s,t}-\mu_2P_{s,t}}_{\tiny V}\leq c~e^{-b(t-s)}~\vertiii{ \mu_1-\mu_2}_{\tiny V}.
$$
For a direct proof of the above estimate based on $V$-norm contraction properties of Markov semigroups we refer to section~\ref{sec-preliminary-V-dob}.

For some measure $\mu\in \Pa_V(E)$ and some positive function $H\in \Ba_{0,V}(E)$, we define
the finite rank (and hence compact) operator
\begin{equation}\label{T-H-mu}
f\in \Ba_V(E)\mapsto T^{\mu,H}_{s,t}(f):=\frac{Q_{s,t}(H)}{\mu_sQ_{s,t}(1)}~\mu_t(f)\in \Ba_V(E)
\end{equation}
with the flow of measures $\mu_t=\Phi_{s,t}(\mu_s)$ starting at  some $\mu_0=\mu$. In this notation, we have the following time inhomogenous version of the Krein-Rutman theorem.

\begin{cor}\label{cor-TH}
For any $s\leq t $ we have the operator norm exponential decay
$$
\vertiii{\frac{Q_{s,t}}{\mu_sQ_{s,t}(1)}-T_{s,t}^{\mu,H}}_V\leq ~c_H(\mu)~ e^{-b (t-s)}$$ 
for some finite constant $c_H(\mu)$ that depends on $(\mu,H)$ and
with the same constant $b>0$ as in (\ref{lipschitz-inq-V}). 
\end{cor}
A precise description of the parameters $(b,c_H(\mu))$ is provided in Corollary~\ref{cor-projection} (applied to  $\overline{\eta}_0=\mu$, see also Lemma~\ref{lemma-Nick}). To the best of our knowledge, the above corollary is the first result of this type for time varying positive semigroups.

We emphasize that for locally compact Polish spaces, our methodology also applies to 
Feller $V$-positive semigroups, that is  $Q_{s,t}(\Ca_V(E))\subset \Ca_{V}(E)$ and $Q_{s,t}(V)/V\in \Ca_0(E)$ for any $s<t$, as soon as  the Lyapunov function $V$ can be   continuous (also choosing a function $H\in \Ca_{0,V}(E)$ in the definition of the class of non necessarily absolutely continuous $V$-positive semigroups introduced in Definition~\ref{ref-def-V-sg-intro}). As shown in (\ref{ref-compact-cv}), for time homogeneous semigroups, in this context the continuity of the ground function is granted by compact convergence arguments. 
 
Now, choosing $V\in \Ca_{\infty}(E):=\Ba_{\infty}(E)\cap \Ca(E)$, Theorem~\ref{theo-2-intro} is also valid when we replace the left-hand side condition in (\ref{ref-V-theta}) by the Feller property $Q^V_{s,t}(\Ca_b(E))\subset \Ca_b(E)$ with $s<t$ of the conjugate semigroup
$
Q^V_{s,t}(f):=Q_{s,t}(fV)/V$ and the $V$-positiveness property by the condition
$Q_{s,t}^V(f)\in \Ca_0(E)$ for any positive function $f\in\Ca_b(E)$.
In this context, the right-hand side condition in (\ref{ref-V-theta}) and (\ref{discrete-2-continuous}) become
$$
Q^V_{s,s+\tau}(1)\leq \Theta_{\tau}\in \Ba_0(E)\quad \mbox{\rm and}\quad \pi_{\tau}(V):=\sup_{\epsilon [0,\tau[}\Vert Q_{t,t+\epsilon}^V(1)\Vert<\infty.
$$

For normed finite spaces $(E,\Vert \cdot\Vert)$, the right-hand side Lyapunov condition in (\ref{ref-V-theta}) only relies on the design of a function $x\in E\mapsto\Theta_{\tau}(x)$ that tends to $0$ as $\Vert x\Vert\rightarrow\infty$.
For sub-Markovian semigroups with hard obstacles, the right-hand side Lyapunov condition in (\ref{ref-V-theta})  allows one to consider diffusion processes conditional to non absorption on the boundaries of
some open connected domains. For instance, for a  time homogeneous semigroup $Q_{t}(x,dy)=q_t(x,y)~dy$ with a bounded density $q_t$ on some bounded domain (open connected) $E\subset\RR^n$ with Lipschitz boundary $\partial E$, for any  $0<\delta<1$ we have
$$
\begin{array}{l}
V(x):=1/d(x,\partial E)^{1-\delta}\quad \mbox{\rm with}\quad d(x,\partial E):=\inf{\left\{\Vert x-y\Vert~:~y\in \partial E\right\}}\\
\\
\Longrightarrow  Q_{\tau}(V)/V\leq \Theta_{\tau}:=c_{\tau}/V\in \Ba_0(E)\quad \mbox{\rm with $c_{\tau}:=\Vert Q_{\tau}(V)\Vert<\infty$.}
\end{array}
$$
For a more concrete example, we also mention that Brownian motion conditioned to non absorption  on $E:=]0,1[$ has a Dirichlet heat kernel $q_t(x,y)>0$ on the open cell $E^2=]0,1[^2$ that satisfies the conditions of Theorem~\ref{theo-2-intro}  with the  Lyapunov function $V(x)=1/\sqrt{x}+1/\sqrt{1-x}$. For a more detailed discussion on the
 design of Lyapunov  functions $V$, we refer to the articles~\cite{examples-review-paper,douc-moulines-ritov,ferre,whiteley} 

For time homogeneous models,  the r.h.s. condition in  (\ref{ref-V-theta}) 
takes the form $Q_{\tau}(V)/V\leq \Theta_{\tau}\in\Ba_0(E)$.
In terms of the compact sets $\Ka_{\epsilon}:=\{\Theta_{\tau}\geq \epsilon\}$  this condition ensures that for any $\epsilon>0$ we have the Foster-Lyapunov inequality
\begin{equation}\label{ref-epsilon-V-K}
Q_{\tau}(V)(x)\leq \epsilon~V(x)+1_{\Ka_{\epsilon}}(x)~c_{\epsilon}\quad \mbox{with}\quad
c_{\epsilon}:=
\sup_{\Ka_{\epsilon}}(V\Theta_{\tau})<\infty.
\end{equation}
We can also slightly relax the above by assuming that for any $n\geq 1$ we have
 \begin{equation}\label{ref-epsilon-V-K-n}
Q_{\tau}(V)(x)\leq  \epsilon_n~V(x)+1_{\Ka_{\epsilon_n}}(x)~c_{\epsilon_n}
\end{equation}
where $\Ka_{\epsilon_n}\subset E$ stands for some increasing sequence of compacts sets and $c_{\epsilon_n}$ some finite constants, indexed by a decreasing 
sequence of parameters $\epsilon_n\in [0,1]$ such that $\epsilon_n\longrightarrow_{n\rightarrow\infty} 0$.  Assuming that $Q_{\tau}(V)/V$ is locally lower bounded and lower semicontinuous, condition (\ref{ref-epsilon-V-K-n}) ensures that $Q_{\tau}(V)/V\in \Ba_0(E)$. Indeed,
 for any $\delta>0$, there exists some $n\geq 1$ such that $\epsilon_n< \delta$ and we have
$$
\{Q_{\tau}(V)/V\geq \delta\}\subset\{Q_{\tau}(V)/V> \epsilon_n\}\subset \Ka_{\epsilon_n}
$$
Since $\{Q_{\tau}(V)/V\geq \delta\}$ is a closed subset of a compact set it is also compact.
 More generally, whenever (\ref{ref-epsilon-V-K-n}) holds for some exhausting sequence of compact sets $K_{\epsilon_n}$, in the sense that for any compact subset $K\subset E$ there exists some $n\geq 1$ such that $K\subset K_{\epsilon_n}$, we have
 $$
 \inf_K Q_{\tau}(V)/V\geq  \inf_{K_{\epsilon_n}} Q_{\tau}(V)/V \geq \epsilon_n
 $$
 This condition ensures that the function $Q_{\tau}(V)/V$ is locally lower bounded. In this situation, we have $Q_{\tau}(V)/V\in \Ba_0(E)$ as soon as 
 $Q_{\tau}(V)/V$ is  lower semicontinuous.

 For time varying semigroups, working with conditions of the form (\ref{ref-epsilon-V-K-n}) we also need to ensure that the same sequence of compacts can be used at every time horizon. Last, but not least, sub-Markov semigroup associated with hard-obstacles may be defined on some domain with a complex topological structure. In this context, the design of the Lyapunov function and  the compact subsequence (\ref{ref-epsilon-V-K-n}) depends on the stability properties of the free evolution as well as on the topological structure of the domain. 

We remark that our methodology also applies to non necessarily absolutely continuous semigroups. These models, termed stable $V$-positive semigroups,  are defined in section~\ref{V-sg-sec}. They relies on the contraction properties of the triangular array of Markov operators discussed above on the compact level sets of the function $V$
 (cf.~for instance condition (\ref{loc-dob})). In lemma~\ref{lemma-Nick} (mainly due to N.~Whiteley in~\cite[Proposition 1 \& Lemma 10]{whiteley}) we shall see that the semigroups considered in Theorem~\ref{theo-2-intro} are particular classes of stable $V$-positive semigroups.

In section~\ref{sec:not_res} we present a new unifying methodology that combines Dobrushin ergodic coefficient techniques developed in~\cite{dm-04,dm-13,dm-2000,dg-ihp,dm-sch-2} with the powerful Foster-Lyapunov methodologies developed in the articles~\cite{bansaye-2,douc-moulines-ritov,ferre,konto-meyn-2,whiteley} in the context of Feynman-Kac semigroups and nonlinear filtering.  
Our approach relies on the contraction analysis of a class of triangular arrays of Markov semigroups introduced in~\cite{dm-2000,dg-ihp,dm-sch-2}   in terms of the $V$-norm contraction ergodic theory for Markov operators presented in~\cite{hairer-mattingly} and further developed in a systematic way in the book~\cite{dmpenev}. A brief review on this subject is provided in section~\ref{sec-preliminary-V-dob}. For a more thorough discussion on these $V$-contraction coefficient principles, we refer to~\cite[Section 8.2.5]{dmpenev}.

\subsubsection{Time homogeneous semigroups}
As one might expect, for time homogeneous semigroups,
a variety of results that follow almost immediately from the estimates obtained in Theorem~\ref{theo-2-intro}.  
 These results are described in section~\ref{sec:timehom} and they include the existence of  an unique leading eigen-triple  \begin{equation}\label{def-eigen-triple}
 (\rho,\eta_{\infty},h)\in (\RR\times\Pa_V(E)\times \Ba_{0,V}(E))\quad \mbox{\rm satisfying (\ref{def-eigen-triple-intro})}.
 \end{equation}
Choosing $(\mu,H)=(\eta_{\infty},h)$ in (\ref{T-H-mu}), the operator $T^{\eta_{\infty},h}_{s,t}$ simplifies to the operator $T$ introduced in (\ref{def-Tt}); that is, for any $ f\in \Ba_V(E)$ we have
 $$
T^{\eta_{\infty},h}_{s,t}(f)=\frac{h}{\eta_{\infty}(h)}~\eta_{\infty}(f)=T(f)\in  \Ba_{0,V}(E).
 $$
In addition to the fact that the ground state function $h$ discussed above is generally unknown, whenever its existence is assured, the stability of a non necessarily absolutely continuous positive semigroup $Q_t$ is reduced to the one of a more conventional
 Markov semigroup defined by the $h$-transform of $Q_t$ given by 
\begin{equation}\label{def-Ph-intro} 
 f\in \Ba_{V^h}(E)\mapsto
 P^h_t(f):=Q_t(fh)/Q_t(h)\in \Ba_{V^h}(E)\quad \mbox{\rm with}\quad V^h:=V/h.
\end{equation}
The condition $Q_{\tau}(V)/V\in \Ba_0(E)$ ensures that $V^h\in \Ba_{\infty}(E)$ and there exists some $0<\epsilon<1$ and some constant $c>0$ such that
\begin{equation}\label{lyap-Ph-intro}
P^h_{\tau}(V^h)\leq \epsilon~V^h+c.
\end{equation}
A proof of the above assertion is provided in Lemma~\ref{ref-lem-Lyapunov-2}. We now introduce another condition.

{\em $(\Ha^h)$ There exists some $r_0>0$ and some $\alpha\,:\,r\in [r_0,\infty[\mapsto \alpha(r)\in ]0,1]$ such that for any $r\geq r_0$ we have
\begin{equation}\label{ref-P-h-n-2-hom}
\sup_{V^h(x)\vee V^h(y)\leq r}\left\Vert \delta_xP^h_{\tau}-\delta_yP^h_{\tau}\right\Vert_{\tiny tv}\leq 1-\alpha(r).
\end{equation}}

Using the Lyapunov inequality  (\ref{lyap-Ph-intro}) and  the local contraction estimate (\ref{ref-P-h-n-2-hom}) the stability properties of $P^h_t$ follows the conventional  $V$-norm contraction methodology for Markov semigroups developed in 
Section~\ref{sec-preliminary-V-dob}. 

For the rest of the section, we assume that $Q_t$ is  a (non necessarily absolutely continuous) time homogeneous 
$V$-positive semigroup in the sense that $Q_t(\Ba_V(E))\subset\Ba_{0,V}(E)$, for any $t>0$.
In addition, there exists and leading eigen-triple (\ref{def-eigen-triple})  with $h(x)>0$ for any $x\in E$, and the Doob's $h$-transform satisfies condition $(\Ha^h)$.
In this situation, the finite rank operator $T$ defined in (\ref{def-Tt}) maps $\Ba_V(E)$ into 
$ \Ba_{0,V}(E)$. As shown in Section~\ref{sec:timehom} these conditions are met under our regularity conditions (cf. (\ref{ref-Q-min-prop}), (\ref{ref-Q-min}) and Corollary~\ref{H1-prop}).

 The next theorem is a synthesis of Theorem~\ref{stab-h-transform-time} and its corollaries, Corollary~\ref{stab-h-Phi-transform-time} and Corollary~\ref{cor-projection-h}, adapting the arguments of (\ref{ref-Q-1-rho}) to $V$-norms. To the best of our knowledge, the next theorem does not exist in the literature. 
\begin{theo}
The Markov semigroup $P^h_{t}$ has a single invariant measure  $\eta^h_{\infty}\in \Pa_{V^h}(E)$ and $Q_t$ is a quasi-compact operator on $\Ba_V(E)$.
In addition, there exists  some $a<\infty$ and $b>0$, such that 
for any $\mu,\eta\in \Pa_{V^h}(E)$ and $t\in \Ta$ we have the contraction estimate
 \begin{equation}\label{first-h-estimate-intro}
 \Vert \mu P_{t}^h-\eta P_{t}^h\Vert_{V^h}\leq a~e^{-b t}~ \Vert \mu- \eta\Vert_{V^h}\qquad \mbox{and}\qquad \vertiii{e^{-\rho t}~Q_{t}-T}_{V}\leq 
a ~e^{-b t}.
 \end{equation}
The conjugate measure
$
\eta_{\infty}(dx):=1/h(x)~\eta^h_{\infty}(dx)/\eta^h_{\infty}(1/h)\in  \Pa_V(E) 
$ is the unique invariant measure of the semigroup $\Phi_{t}$.
For any $\mu_1,\mu_2\in \Pa_{V}(E)$ there also exists some finite constant $c(\mu_1,\mu_2)<\infty$  such that for any $t\in \Ta$ we have 
$$
\vertiii{\Phi_{t}(\mu_1)-\Phi_{t}(\mu_2)}_{V}\leq ~c(\mu_1,\mu_2)~e^{-b t}~\vertiii{\mu_1-\mu_2}_{V}.
$$
\end{theo}

 \subsection{On the design of Lyapunov functions}\label{reg-cond-sec}

Condition (\ref{ref-epsilon-V-K-n}) is often presented in the literature as an initial Lyapunov condition to analyze the stability property of time homogenous sub-Markov semigroups  (see for instance~\cite{ferre,guillin}, as well as section 17.5 in~\cite{dmpenev} in the context of Markov semigroups and the references therein). 
In this context, the Lyapunov condition in the right-hand side of (\ref{intro-QV-Theta}) provides a practical way to design Lyapunov functions satisfying (\ref{ref-epsilon-V-K-n}). This section provides some elementary principles to  design these Lyapunov functions.

Assume that $Q_{s,s+\tau}$ is dominated by  some auxiliary positive integral operators $\Qa_{s,s+\tau}$ in the sense that
\begin{equation}\label{ref-comparison}
Q_{s,s+\tau}(x,dy)\leq c_{\tau}~\Qa_{s,s+\tau}(x,dy)\quad \mbox{\rm for some finite constant $c_{\tau}<\infty$}.
\end{equation}
Then, we readily check that
$$
~\Qa_{s,s+\tau}(V)/V\leq \Theta_{\tau}\in \Ba_0(E)\Longrightarrow Q_{s,s+\tau}(V)/V\leq c_{\tau} \Theta_{\tau}\in \Ba_0(E).
$$
Given a time homogeneous  semigroup 
$\Qa_{t}$,  and some pair of functions $V\in \Ba_{\infty}(E)$, and $H\in \Ba_{V}(E)$ such that $V^H:=V/H\in \Ba_{\infty}(E)$ and $\Qa_{\tau}(V)/V\in \Ba_{0}(E)$ we have the conjugate principle
\begin{equation}\label{ref-comparison-H}
\begin{array}{l}
\displaystyle Q_{s,s+{\tau}}(x,dy)\leq c_{\tau}~\Qa_{\tau}(x,dy) H(y)/H(x)\quad \mbox{\rm for some $c_{\tau}<\infty$}\\
\\
\displaystyle\Longrightarrow
Q_{s,s+{\tau}}(V^H)/V^H\leq \Theta_{\tau}:=c_{\tau}~\Qa_{\tau}(V)/V\in \Ba_{0}(E).
\end{array}
\end{equation}

For instance, a sub-Markovian semigroup $Q_t$ associated with a linear hypoelliptic diffusion on $E:=\RR^n$ evolving in an absorbing potential that grows at least quadratically is dominated by the sub-Markovian semigroup $\Qa_t$ of a coupled harmonic oscillator. The latter of which has an explicit solution with an exponential decay total mass $\Qa_{\tau}(1)(x)\longrightarrow_{\Vert x\Vert\rightarrow\infty}0$, with also well known leading triplet and Lyapunov functions, such as the functions $V(x)=1+\Vert x\Vert^n$, for any given $n\geq 1$. In this context the ground state function $h\in \Ca_{0}(E)$ is the centered Gaussian density with a covariance matrix satisfying an algebraic Riccati equation, and the corresponding $h$-process is an Ornstein-Ulhenbeck diffusion. In this scenario, we recall that $\Qa_t$ and the semigroup $\Pa^h_t$ of the  $h$-process diffusion flow are connected by the formula $\Qa_t(f)=e^{\rho t} h~\Pa^h_t(f/h)$, for some $\rho<0$.  We refer to~\cite{dm-horton-21} for a detailed discussion on coupled harmonic oscillators.

Whenever the domination property  (\ref{ref-comparison}) is met for some time homogeneous semigroup 
$\Qa_{\tau}:=\Qa_{s,s+\tau}$,  such that  $\Qa_{\tau}(1)\in \Ba_0(E)$, for any $V\in\Ba_{\infty}(E)$ have
$$
\begin{array}{l}
\displaystyle c_{\tau}:=\Vert \Pa_{\tau}(V)/V\Vert<\infty
\quad \mbox{\rm with the Markov operator}\quad
\Pa_{\tau}(f):={\Qa_{\tau}(f)}/{\Qa_{\tau}(1)}\\
\\
\Longrightarrow
Q_{s,s+\tau}(V)/V\leq \Theta_{\tau}:=c_{\tau}~\Qa_{\tau}(1)\in \Ba_0(E).
\end{array}$$
If $c_{\tau}^{\prime}:=\Vert\Pa_{\tau}(V)\Vert<\infty$ and $\Qa_{\tau}(1)/V\in \Ba_0(E)$ we also have
$$
Q_{s,s+\tau}(V)/V\leq \Theta_{\tau}:=c^{\prime}_{\tau} ~\Qa_{\tau}(1)/V\in \Ba_0(E).
$$
For instance, when $\Qa_{\tau}$ is  the semigroup associated with the half-harmonic oscillator on $E=]0,\infty[$, choosing the function $V(x)=x^n+1/x$ for any given $n\geq 1$, we have
$\Vert \Qa_{\tau}(V)\Vert<\infty$, for any $\tau>0$.


The semigroup $Q_t$ of a non absorbed Ornstein-Ulenbeck diffusion flow on $E=]0,\infty[$  killed at the origin can be seen as the $h$-transform of a solvable sub-Markov semigroup associated with a one dimensional linear diffusion evolving 
in a quadratic potential well and killed at the origin. In this context, the ground state function $h$ is a centered Gaussian density and  we can choose the Lyapunov function $V^h(x)=(x^n+1/x)/h(x)$, for any given $n\geq 1$. The stability properties of this class of non absorbed one dimensional Ornstein-Ulenbeck diffusions are also discussed in~\cite{lladser} and more recently in~\cite{ocafrain} in terms of the tangent of the $h$-process.

Next, we illustrate (\ref{ref-comparison-H}) with a one-dimensional Langevin diffusion killed at the origin. Let $X_t(x)$ be the stochastic flow  with generator $L$ defined
by  $$L(f)=\frac{1}{2}~e^{2W}\partial \left(e^{-2W}\partial f\right)$$ for some non negative function $W$. We denote by $Q_t$ the sub-Markovian semigroup associated with the flow $X_t(x)$ starting at $x\in E=]0,\infty[$ killed at the boundary $\partial E=\{0\}$; that is, we have that
$$
Q_t(f)(x):=\EE(f(X_t(x))~1_{T_{\partial E}^X(x)>t})\quad \mbox{\rm with}\quad
T_{\partial E}^{X}(x):=\inf{\left\{t\geq 0~:~X_t(x)\in\partial E\right\}}.
$$
Consider the sub-Markovian semigroup $\Qa_{t}$ associated with a Brownian flow $B_t(x)$ on $E=]0,\infty[$  killed at the origin and at rate 
$$
U:=\frac{1}{2}~\left((\partial W)^2-\partial^2W\right)=H^{-1}\frac{1}{2}~\partial^2H\quad \mbox{\rm with}\quad H:=e^{-W}.
$$
When $W(x)=\varsigma~x^2/2$, we have $U:=\varsigma\left(\varsigma~x^2-1\right)/2$ and the semigroup $\Qa_{t}$ coincides with the semigroup of an half-harmonic oscillator 
for which we know that
$$
V(x):=x+{1}/{x}\Longrightarrow c_{\tau}:=\Vert \Qa_{\tau}(V)\Vert<\infty.
$$
This implies that $V^H\in \Ba_{\infty}(E)$ and by a change of probability measure we have
$$
Q_{\tau}(V^H)/V^H=\Qa_{\tau}(V)/V\leq c_{\tau}/V\in \Ba_0(E).
$$ 
More generally, assume that $W$ is chosen so that
$$
U(x)\geq \varsigma_0+\varsigma_1~x^2/2\quad\mbox{\rm for some parameters $\varsigma_0\in \RR$ and $\varsigma_1>0$}.
$$
In this situation $\Qa_t$ is now dominated by  the semigroup of an half-harmonic oscillator. The special case $
\partial W(x)={1}/{(2x)}-a~x+b~x^3$ with  $a,b>0
$ yields the logistic diffusion on the half-line discussed in~\cite{cattiaux-1} using spectral arguments.

When the dominating operator $\Qa_{t}=\Pa_{t}$ is a Markov integral operator, the literature is also abounds with the design of Lyapunov functions for Markov operators
 $\Pa_{\tau}$ which can be used without further work to check the right-hand side condition in (\ref{ref-V-theta}) as soon as  $Q_{s,s+\tau}$ is dominated by a Markov operator $\Pa_{\tau}$. For instance, when $\Qa_t=\Pa_t$ is the Markov semigroup associated with a Riccati matrix-valued diffusion $X_t(x)$ evolving in the space $E$ of positive definite matrices with real entries, under appropriate controllability and observability conditions, for any $t>0$ we have
 $\Vert\Pa_t(V)\Vert<\infty$ with $V(x):=\tr(x)+\tr(x^{-1})$, where $\tr(x)$ stands for the trace of a positive definite matrix $x\in E$. A proof of the above assertion is provided in~\cite{bishop-19-2}. 
 The same Riccati-type analysis applies to the logistic birth and death processes and the competitive and multivariate 
Lotka-Volterra birth and death process on $E:=\NN^n-\{0\}$ discussed in Theorem 1.1 in~\cite{champagnat-6}, with the Lyapunov function defined for any $x=(x_i)_{1\leq i\leq n}$ by $V(x)=\sum_{1\leq i\leq n} x_i$.
 
For more discussion on the
 design of Lyapunov  functions $V$ on not necessarily bounded domains, we refer to the articles~\cite{examples-review-paper,douc-moulines-ritov,ferre,whiteley} for more illustrations, as well as to section~\ref{sec-fa-not}.   
 
 \subsection{Organisation of the article}
The article is structured as follows:
In section \ref{basic-notation-sec} we introduce the notation that will be used throughout and state our main results. Section~\ref{all-sg-sec} is  dedicated to the detailed description of the different classes of semigroups considered in the article. 
 The main stability and contraction theorems of the article are described in section~\ref{sec:not_res}.
 In particular, in section~\ref{sec:timeinhom}, we present the main results for time-inhomogeneous models, with a more refined analysis of time homogeneous models being given in section~\ref{sec:timehom}.   In section \ref{sec-illustrations}, we illustrate the impact of our results with some selected illustrations on nonlinear conditional processes, sub-Markov models  and related Feynman-Kac measures on path spaces. Some comparisons between our regularity conditions and the ones used in existing literature are discussed in section~\ref{sec-comparisons}.
  Section~\ref{proof-theo-sec} is dedicated to the proofs of the main  theorems presented in this article. The appendix houses most of our technical proofs.

\section{Preliminary results}\label{basic-notation-sec}

\subsection{Some basic notation}

\subsubsection{Measure theoretic notation}
We equip the set $\Ma_b(E)$  of bounded signed measures $\mu$ on $E$ with the total variation norm $\Vert\mu\Vert_{tv}:=\vert \mu\vert(E)/2$,  where $
\vert\mu \vert:=\mu_++\mu_-$ stands for  the total variation measure associated with  a Hahn-Jordan decomposition $\mu=\mu_+-\mu_-$ of the measure.

We define the duality map, as well as the right and dual left action of a bounded integral operator $Q$ using the classical measure theoretic notation, as follows:
\begin{gather}
(\mu,f)\in (\Ma_b(E)\times \Ba_b(E))\mapsto \mu(f) \in \RR \quad \text{ with } \quad \mu(f) := \int f(x)\mu(dx) \notag \\
f\in\Ba_b(E)\mapsto Q(f) \in \Ba_b(E) \quad \text{ with } \quad Q(f)(x):= \int Q(x ,dy)~f(y)\label{mu-f-Q-f} \\
\mu\in \Ma_b(E)\mapsto \mu Q\in\Ma_b(E) \quad\text{ with } \quad\mu Q(dy) := \int \mu(dx)~Q(x,dy). \notag
\end{gather}
Given a pair of integral operators $(Q_1,Q_2)$, we denote by $Q_1Q_2$ the integral composition operator defined by
$$
(Q_1Q_2)(x ,dz)=\int Q_1(x ,dy)~Q_2(y ,dz).
$$
For any $n\in\NN$ we also write $Q^n=Q^{n-1}Q$ with the convention $Q^0=I$ the identity operator.
We denote by $I$ the identity integral operator. We say that $Q$ is positive if  $Q(f)\geq 0$ whenever $f\geq 0$. Whenever $Q(1)\leq 1$ we say that $Q$ is sub-Markovian, and $Q$ is said to be Markovian when $Q(1)=1$. 
The Boltzmann-Gibbs transformation $\Psi_h$ associated with some bounded positive function $h>0$ and defined 
by
 \begin{equation}\label{def-Psi-H}
 \Psi_{h}~:~\mu\in \Pa(E)\mapsto \Psi_{h}(\mu)\in \Pa(E)\quad \mbox{\rm with}\quad
 \Psi_{h}(\mu)(dx):=\frac{h(x)~\mu(dx)}{\mu(h)}.
\end{equation}
We recall the Lipschitz estimate
$$
\Vert \Psi_h(\mu_1)-\Psi_h(\mu_2)\Vert_{\tiny tv}\leq \frac{\Vert h\Vert}{\mu_1(h)\vee\mu_2(h)}~
\Vert \mu_1-\mu_2\Vert_{\tiny tv}.
$$
For a detailed proof of the above assertion we refer to \cite[lemma 9.5]{hurzeler}, or \cite[appendix B]{diaconis-freedman}, see also~\cite{dg-ihp} as well as \cite[proposition 3.1]{legland-mevel-99} and \cite[proposition~12.1.7]{dm-13}.

When $f=1_{A}$ is the indicator function of some measurable subset $A\subset E$, we will sometimes slightly abuse notation and write $\mu(A)$ instead of $\mu(1_{A})$.  We also set $a\wedge b=\min(a,b)$ and $a\vee b=\max(a,b)$, for $a,b\in \RR$ and we use the conventions $$\left(\sum_{\emptyset},\prod_{\emptyset}\right)=(0,1)\quad
\mbox{\rm and}\quad \left(\sup_{\emptyset},\inf_{\emptyset}\right)=(-\infty,+\infty).$$

\subsubsection{Functional analysis notation}\label{sec-fa-not}
When $E$ is a $\sigma$-compact Polish space, we let $\Ba_{\infty}(E)\subset \Ba(E)$ the subalgebra of  locally bounded and uniformly positive functions $V$ that grow at infinity; that is, $\sup_KV<\infty$ for any compact set $K\subset E$, and for any $r\geq V_{\star}:=\inf_E V>0$  the $r$-sub-level set $\{V\leq r\}\subset E$ is a non empty compact subset. For instance the function $V(x):=x+1/x$ when $E=]0,\infty[$ belongs to $\Ba_{\infty}(E)$.
Note that the compactness level set condition ensures that $V$ is necessarily lower-semicontinuous (abbreviated l.s.c.) and its infimum on every compact set is attained. 

We check that $\Ba_{\infty}(E)$  is an algebra by recalling that the product $V=V_1V_2$ of non negative l.s.c functions $V_1,V_2\in \Ba_{\infty}(E)$ is also l.s.c. so that the $r$-sub-level set $\{V\leq r\}$ is closed. In addition, it is included in the union of the compact $\sqrt{r}$-sub-level sets of the functions $V_1$ and $V_2$; that is
$$
\{V\leq r\}\subset \{V_1\leq \sqrt{r}\}\cup \{V_2\leq \sqrt{r}\} .
$$
Observe that for any  locally bounded l.s.c. function $V_1$ we have
$$
V_1\geq V_2 \in\Ba_{\infty}(E)\Longrightarrow V_1\in \Ba_{\infty}(E).
$$
Since $V$ is locally bounded,  any compact set $K\subset E$ is included in some sub-level set of $V$. Indeed, choosing $r_K:=\sup_KV$ we have
\begin{equation}\label{k-in}
K\subset \{V\leq r_K \}.
\end{equation}
Thus, for any pair of functions $V_1,V_2\in \Ba_{\infty}(E)$ and for any $r_1\geq V_{1,\star}$ there exists some parameters $r_2\geq V_{2,\star}$ and $r_3\geq V_{1,\star}$ such that
$$
\{V_1\leq r_1 \}\subset \{V_2\leq r_2 \}\subset \{V_1\leq r_3 \}.
$$

Let $\Ba_0(E)\subset \Ba_b(E)$ the subalgebra of bounded positive functions $h$ locally lower bounded that vanish at infinity; that is, $\inf_Kh>0$ for any compact set $K\subset E$ and for any $0<\epsilon\leq \Vert h\Vert<\infty$ the $\epsilon$-super-level set $\{h\geq \epsilon\}\subset E$ is a non empty compact subset.  Observe that 
$$
V\in \Ba_{\infty}(E)\quad \Longleftrightarrow\quad  1/V\in \Ba_0(E).
$$
In addition,  for any  locally lower bounded u.s.c. function $h_1$ we have
$$
h_1\leq h_2 \in\Ba_{0}(E)\Longrightarrow h_1\in \Ba_{0}(E).
$$
Also notice that $\epsilon$-super-level set of a non necessarily u.s.c. $h_1\leq h_2 \in\Ba_{0}(E)$ is included in the compact set $\{h_2\geq \epsilon\}$.

Finally note that the sub-algebras $\Ba_{\infty}(E)$ and $\Ba_{0}(E)$ have no unit unless $E$ is compact and the null function $0\not\in \Ba_0(E)$, nevertheless the unit function $1\in \Ca_{0,V}(E)$ for any $V\in \Ba_{\infty}(E)$.

\begin{rmk}
When $E$ is a locally compact space, its topology coincides with the weak topology induced by $\Ca_0(E):=\Ba_0(E)\cap \Ca_b(E)$, and inversely (cf. Proposition 2.1 in~\cite{aliabad}).  In this context a continuous function $h$ vanishes at infinity 
if and only if its extension to 
the one point compactification (a.k.a.
Alexandroff compactification)
$E_{\infty}:=E\cup\{\infty\}$ (obtained by setting $h(\infty)=0$) 
is continuous. For locally compact spaces, we recall that the one point extension
$E_{\infty}$ is compact. 
\end{rmk}
\begin{rmk}\label{rmk-E-normed}
When  $(E,\Vert \cdot\Vert)$ is a finite dimensional normed space, by (\ref{k-in}) for any positive locally lower bounded u.s.c. function $h$ on $E$, and any $\epsilon>0$ there exists some $r>0$ such that 
$$
\overline{\BB}(r):= \{x\in E~:~\Vert x\Vert\leq  r\}\subset \{h\geq \epsilon\}\quad \mbox{so that}\quad h\in \Ba_0(E)\Longleftrightarrow\lim_{\Vert x\Vert\rightarrow\infty}h(x)=0.
$$
In this context,  we have
$$
0\leq h_1\leq h_2\in \Ba_0(E)\Longrightarrow \forall \epsilon>0\quad \exists r>0\quad \mbox{\rm s.t.}\quad \forall x\not\in \overline{\BB}(r)\qquad h_1(x)\leq \epsilon.
$$
\end{rmk}

We let $\Ca(E)\subset\Ba(E)$ the sub-algebra of continuous functions. 
For a given $V\in \Ba_{\infty}(E)$, we let $\Ba_{V}(E)\subset \Ba(E)$ be the Banach space of functions $f\in \Ba(E)$ with  $\Vert f\Vert_V:=\Vert f/V\Vert<\infty$; and by $\Ca_{V}(E)\subset \Ba_V(E)$ be the subset of continuous functions. 
We also denote by  $\vertiii{Q}_{V}$ the operator norm of  a bounded linear operator $Q:f\in \Ba_{V}(E)\mapsto Q(f)\in \Ba_{V}(E)$; that is
\begin{equation}\label{def-op-norm}
\vertiii{Q}_{V}:=\sup\{\Vert Q(f)\Vert_V~:~f\in  \Ba_V(E)\quad\mbox{\rm such that}\quad \Vert f\Vert_V\leq 1\}.
\end{equation}
We also denote by $\Ma(E)$ the set of signed Radon measures on $E$ 
 and by $
\Ma_V(E)\subset \Ma(E)
$ the subset of measures $\mu\in  \Ma(E)$ such that $\vert\mu\vert(V)<\infty$. 

For a given function $V\geq 1/2$, the $V$-oscillation of a function $f\in\Ba(E)$ is given by 
$$
\mbox{\rm osc}_V(f):=\sup_{x,y}\frac{\vert f(x)-f(y)\vert}{V(x)+V(y)}\leq \Vert f\Vert_V
$$
and with a slight abuse of notation, the $V$-norm of  {a measure $\mu\in\Ma(E)$ with null mass $\mu(E)=0$} is given by
\begin{eqnarray*}
\vertiii{\mu}_{V}&=&\sup\{|\mu(f)|~:~\Vert f\Vert_V\leq 1\}=\sup\{|\mu(f)|~:~\mbox{\rm osc}_V(f)\leq 1\}
=\vert\mu|(V).
\end{eqnarray*}
  For a detailed proof of the equivalent formulations in the latter definition, we refer to Proposition 8.2.16 in~\cite{dmpenev}. The choice of condition $V\geq 1/2$ in the above two definitions is imposed only to recover the conventional total variation dual distance between probability measures when choosing $V=1/2$ in the dual norms.

When $V=1/2$ we recover the conventional total variation norm, that is for any $\mu_1,\mu_2\in \Pa(E)$  we have
 \begin{equation}\label{ref-tv-distance}
\vertiii{\mu_1-\mu_2}_{1/2}=\Vert \mu_1-\mu_2\Vert_{\tiny tv}
\end{equation} 
we recall that
 \begin{equation}\label{ref-coupling-tv}
\Vert \mu_1-\mu_2\Vert_{\tiny tv}\leq 1-\epsilon\Longleftrightarrow\left( \exists \nu\in \Pa(E)~:~\mu_1\geq \epsilon~\nu\quad\mbox{\rm and}\quad \mu_2\geq \epsilon~\nu\right).
\end{equation}

\begin{rmk}
Whenever the Polish space $E$ is locally compact metric space the integral map $(\mu,f)\mapsto \mu(f)$  gives the isometry
$$
\Ma_{V}(E)
\simeq\Ca_{0,V}(E)^{\prime}.
$$
In this context, the set of Radon measures reduces to the set of locally finite Borel regular measures.
The above assertion is a direct consequence of \cite[Theorem 3.1]{summers} applied to the Nachbin family $\Va:=\{V_{\alpha}=\alpha/V~:~\alpha>0\}$. See also \cite[Theorem 3.26]{summers-phd} as well as \cite[Theorem 2.1]{bucur}. 
\end{rmk}

Finally, we present a technical lemma regarding the $V-$norm estimates of the Boltzmann-Gibbs operators; the proof is given in appendix \ref{app:lem-BG-V}.

\begin{lem}\label{lem-BG-V}
For any  $V\in \Ba_{\infty}(E)$ and $0<h\in \Ba_{0,V}(E)$ and any $\mu_1,\mu_2\in \Pa_{V/h}(E)$ we have the estimate
\begin{equation}\label{Lip-Psi-h-2}
\vertiii{\Psi_{1/h}(\mu_1)-\Psi_{1/h}(\mu_2)}_{V}\leq\frac{1}{\mu_1(1/h)}
\left(1+\frac{\mu_2(V/h)}{\mu_2(1/h)}\right)~\vertiii{\mu_1-\mu_2}_{V/h}.
\end{equation}
 In addition, for any 
$\mu_1,\mu_2\in \Pa_V(E)$ we have
\begin{equation}\label{Lip-Psi-h-2-back}
\vertiii{\Psi_{h}(\mu_1)-\Psi_{h}(\mu_2)}_{V/h}\leq \frac{1}{\mu_1(h)}\left(1+\frac{\mu_2(V)}{\mu_2(h)}\right)~\vertiii{ \mu_1-\mu_2}_{V}.
\end{equation}
\end{lem}

\subsection{Discrete and continuous time models}\label{sec-d-2-continuous}

For a given  $s\in \Ta$ and $\tau\in \Ta$ with $\tau>0$, we consider the time mesh
$$
[s,t]_{\tau}:=\{s+n\tau \in [s,t]~:~n\in\NN\}\quad \mbox{\rm and }\quad [s,\infty[_{\tau}:=\{s+n\tau \in [s,\infty[~:~n\in\NN\}.
$$
We define $[s,t[_{\tau}$, $]s,t]_{\tau}$ and $]s,t[_{\tau}$ by replacing respectively in the above display  $[s,t]$ by
 $[s,t[$, $]s,t]$ and $]s,t[$.  
 For continuous time indices $\Ta=\RR_+$, we shall denote by $\lfloor t/ \tau\rfloor$ the integer part of $t/\tau$ and by $\left\{t/\tau\right\}$ the fractional part so that $t=\lfloor t/ \tau\rfloor\tau+ \left\{t/\tau\right\}\tau$. For discrete time indices $\Ta=\NN$, choosing $\tau=1$ we have $[0,\infty[_{\tau}=\NN$.

 We have assumed that $Q_{t,t+\tau}(1)>0$ for any $\tau>0$ and any $t\in \Ta$. This irreducibility condition ensures that $Q_{t,t+\tau}(f)>0$ for any $f>0$. 
   For sub-Markovian time homogeneous semigroups $Q_{t,t+\tau}=Q_{0,\tau}$ this condition can be relaxed 
by considering the Borel set $\{Q_{0,\tau}(1)=0\}$ as a part of an absorbing set. In this context, to analyse the behavior of a non absorbed particle there is no loss of generality in assuming that $Q_{0,\tau}(1)>0$. 

   In the discrete time setting, this condition can also be relaxed up to a time-rescaling or up to some modification of the state space; see for instance the construction in~\cite[section 4.4]{dm-04}, {as well as the one dimensional neutron transport model  discussed in~\cite[Example 4.4.3]{dm-04} and the soft and hard obstacle models discussed in~\cite{dm-doucet-04,dm-jasra-18}}.

Nevertheless, we emphasise that the analysis of continuous time models is sometimes based on a discrete time approach based on regularity conditions that depend on some parameter $\tau$ chosen by the user. In this context, the analysis is performed at the level of the $\tau$-discretised model and the estimation constants presented in our results may depend on the parameter $\tau$. These estimation constants are defined and discussed in some details below.

\begin{defi}\label{def-parameters}

For a given a function $W\geq 0$ and some probability measure $\eta$ on $E$ and some parameter $\tau>0$ we shall denote by $\kappa^-_{\tau,W}(\eta),\lambda^-_{\tau}(\eta)\in [0,+\infty[$ and $\kappa_{\tau,W}(\eta),\lambda_{\tau}(\eta)\in [0,+\infty]$ the parameters
\begin{alignat}{6}
\kappa^-_{\tau,W}(\eta)&:=&&\inf\Phi_{s,t}(\eta)(W)&& \leq \quad&& \kappa_{\tau,W}(\eta)&&:=\,&&\sup\Phi_{s,t}(\eta)(W)\nonumber\\
\lambda^-_{\tau}(\eta)&:=&&\inf\Phi_{s,t}(\eta)(Q_{t,t+\epsilon}(1))&&\leq \quad&& \lambda_{\tau}(\eta)&&:=\,&&\sup\Phi_{s,t}(\eta)(Q_{t,t+\epsilon}(1)).
\label{kappa-V-h-intro}
\end{alignat}
In the above display, the infimum and the supremum are taken over all  
time indices $s\in \Ta$ and $t\in [s,\infty[_{\tau}$ and $\epsilon=\tau$.  With a slight abuse of notation, we  also denote by $\kappa^-_{W}(\eta),\lambda^-(\eta),\kappa_{W}(\eta),\lambda(\eta)$, the parameters defined as above by taking the infimum and the supremum are taken over all  
time indices $s\in \Ta$ and $s\leq t\in \Ta$ and $\epsilon\in \Ta\cap[0,\tau]$.
\end{defi}
For discrete time models, we have assumed that $\tau=1$ so that $(\kappa^-_{\tau,W},\kappa_{\tau,W})=(\kappa^-_{W},\kappa_{W})$ and $(\lambda^-_{\tau}(\eta),\lambda_{\tau}(\eta))=(\lambda^-(\eta),\lambda(\eta)$.
For continuous time models, the next lemma provides conditions under which  the infimum and the supremum are taken over all  continuous time
time indices $s\geq 0$ and $t\geq s$. 

\begin{lem}\label{lem-d-to-c}
For continuous time models, assume  $\kappa^-_{\tau,H}(\mu)>0$ and $\kappa_{\tau,V}(\mu)<\infty$, and  condition (\ref{discrete-2-continuous-intro-K}) are met for some $\tau>0$, $\mu\in \Pa_V(E)$, $V\in \Ba_{\infty}(E)$ and  some locally lower bounded positive function $H$ on $E$. In this situation, we have
$\kappa^-_{H}(\mu)>0$ and $\kappa_{V}(\mu)<\infty$ as well as
\begin{equation}\label{ref-ratio-eta-eps}
0<\lambda^-(\mu)\leq \lambda(\mu)<\infty\quad \mbox{and}\quad
\sup\frac{\Vert Q_{t,t+\epsilon}(1)\Vert}{\Phi_{s,t}(\mu) Q_{t,t+\epsilon}(1)}<\infty
\end{equation}
where the infimum and the supremum are taken over all  
continuous time indices $s\in \Ta$ and $t\geq s$  and $\epsilon\in \Ta\cap[0,\tau]$. 
\end{lem}
The proof of the above technical lemma is provided in the appendix on page~\pageref{lem-d-to-c-proof}.
We end this section with a brief discussion on
 unnormalised and the normalised semigroups $Q_{s,t}$ and $\Phi_{s,t}$. These linear and nonlinear semigroups are connected for any $\mu\in \Pa(E)$, $f\in \Ba_{b}(E)$ and $s\in \Ta$ and $t\in [s,\infty[_{\tau}$ by the formula
\begin{equation}\label{product}
\mu Q_{s,t}(f)=\Phi_{s,t}(\mu)(f)~\prod_{u\in [s,t[_{\tau}}~\Phi_{s,u}(\mu)(Q_{u,u+\tau}(1)).
\end{equation}
The above formula coincides with the product  formula relating the unnormalised operators
$Q_{s,t}$
with the normalised semigroup $\Phi_{s,t}$ discussed in \cite[section 1.3.2]{dm-2000}, see also \cite[proposition 2.3.1]{dm-04} and \cite[section 12.2.1]{dm-13}.
For time homogeneous models, the product formula (\ref{product}) applied to $f=1$ reduces for any $t\in\Ta$ to 
\begin{equation}\label{product-hom}
\mu(Q_t(1))=\Phi_{\lfloor t/\tau\rfloor \tau}(\mu)\left(Q_{\left\{t/\tau\right\}\tau}(1)\right)~\prod_{0\leq n< \lfloor t/\tau\rfloor}~\Phi_{n\tau}(\mu)(Q_{\tau}(1)).
\end{equation}
Rewritten in a slightly different form, we have
$$
\mu(Q_t(1))=\frac{1}{\Phi_{t}(\mu)\left(Q_{(1-\left\{t/\tau\right\})\tau}(1)\right)}~\prod_{0\leq n\leq \lfloor t/\tau\rfloor}~\Phi_{n\tau}(\mu)(Q_{\tau}(1)).
$$

\subsection{$V$-Dobrushin coefficient}\label{sec-preliminary-V-dob}

Let us fix $0 \le s \le t$ and let $V_s, V_t$ denote a couple of measurable functions such that $V_s, V_t \ge 1$. The $V$-Dobrushin coefficient of  a Markov integral operator $P_{s,t}$ (non necessary a semigroup) from $\Ba_{V_t}(E)$ into  $\Ba_{V_s}(E)$ is the norm operator defined by
\begin{eqnarray}
 \beta_{V_s,V_t}(P_{s,t})&=&\sup_{\mu,\eta\in \Pa_{V_s}(E)}{\vertiii{(\mu-\eta) P_{s,t}}_{V_t}}/{\vertiii{ \mu-\eta}_{V_s}}.\label{beta-V-norm}
 \end{eqnarray}
We also have the equivalent formulation
\begin{eqnarray}
 \beta_{V_s,V_t}(P_{s,t})&=&\sup{\left\{\mbox{\rm osc}_{V_s}(P_{s,t}(f))~:~\mbox{\rm osc}_{V_t}(f)\leq 1\right\}}\nonumber\\
 &=&\sup_{(x,y)\in E^2}{\vertiii{ \delta_xP_{s,t}-\delta_y P_{s,t}}_{V_t}}/{(V_s(x)+V_s(y))}.\label{def-Dob-V-P}\
\end{eqnarray}
When  $V_s=V_t = V$, we write $ \beta_{V}(P_{s,t})$ instead of $ \beta_{V, V}(P_{s,t})$. For a more thorough discussion on these contraction coefficients, we refer the reader to sections 8.2.5 - 8.2.7  in \cite{dmpenev}.  If in addition $V=1/2$ we write  $ \beta(P_{s,t})$ instead of  $ \beta_{1/2}(P_{s,t})$, to denote
 the conventional Dobrushin ergodic coefficient with respect to the total variation distance
 $$
\beta(P_{s,t})= \sup_{(x,y)\in E^2}\Vert\delta_xP_{s,t}-\delta_yP_{s,t}\Vert_{\tiny tv}.
 $$ 
By (\ref{ref-coupling-tv}) the contraction condition $\beta(P_{s,t})\leq (1-\alpha)$ is satisfied for some parameter $\alpha\in ]0,1]$ {\rm if and only if} the following minorisation condition holds
$$
 \forall (x,y)\in E^2\quad \exists \nu\in \Pa(E)~:~\delta_xP_{s,t}\geq \alpha~\nu\quad\mbox{\rm and}\quad \delta_yP_{s,t}\geq \alpha~\nu.
$$
The next lemma provides some contraction conditions in terms of a Foster-Lyapunov inequality and a local minorisation condition on compact level sets.
The proof is in appendix \ref{app:lem-beta-12}.
\begin{lem}\label{lem-beta-12}
Assume that there exist locally bounded functions $V_{s},V_t \geq 1$ with compact level sets, $\epsilon\in [0,1[$ as well as a function $\alpha : r\in [r_0,\infty[\,\mapsto\alpha(r)\in \, ]0,1]$, for some $r_0\geq 1$ such that for any $r\geq r_0$ we have
\begin{equation}\label{PW-12}
 P_{s,t}(V_t)\leq \epsilon~ V_s+1\quad\mbox{and}\quad\sup_{V_s(x)\vee V_s(y)\leq r}\Vert\delta_xP_{s,t}-\delta_yP_{s,t}\Vert_{\tiny tv}\leq 1-\alpha(r).
\end{equation}
 {Then, for any $r>r_0\vee r_{\epsilon}$, with $r_{\epsilon}:=2/(1-\epsilon)$  we have the contraction coefficient estimate
\begin{equation}\label{def-alpha-delta}
\beta_{V_s^{\epsilon,r},V_t^{\epsilon,r}}(P_{s,t})\leq   1-\alpha_{\epsilon}(r)\quad\mbox{with}\quad\alpha_{\epsilon}(r):=\left(1-\epsilon\right)
\left(1-\frac{r_{\epsilon}}{r}\right)~\frac{\alpha(r)}{5}\in ]0,1[.
\end{equation}
In the above display, $V_u^{\epsilon,r}$ with $u\in \{s,t\}$ stands for the functions defined by
\begin{equation}\label{PW-iota}
V_u^{\epsilon,r}:=1/2+\rho_{\epsilon}(r)~V_u
\quad
\mbox{and}\quad\rho_{\epsilon}(r):=\frac{1}{1+\epsilon}\frac{\alpha(r)}{2r}.
\end{equation}
}
\end{lem}
 {Under the assumptions of lemma~\ref{lem-beta-12}, using (\ref{beta-V-norm}) for any  $\mu,\eta\in \Pa_{V_s}(E)$ we readily check the contraction estimate
\begin{equation}\label{ref-V-contraction}
\vertiii{ \mu P_{s,t}- \eta P_{s,t}}_{V_t^{\epsilon,r}}\leq   \left(1-\alpha_{\epsilon}(r)\right)~ \vertiii{\mu-   \eta}_{V^{\epsilon,r}_{s}}.
\end{equation}
with the parameter $\alpha_{\epsilon}(r)$ defined in (\ref{def-alpha-delta}).
Whenever $P_{s,t}$ is a Markov semigroup the Foster-Lyapunov inequality, on the left-hand side of (\ref{PW-12}), applied to some time homogeneous function $V_t=V$ ensures that for any initial probability measure $\mu\in \Pa_{V}(E)$, the distribution flow  $\mu_t:=\mu P_{0,t}$ indexed by $t\in [0,\infty[_{\tau}$ for some $\tau>0$ is tight. To check this claim, notice that
\begin{equation}\label{ref-Markov-Lyap}
(\ref{PW-12}) \Longrightarrow
\mu_t(V)\leq \epsilon^{t/\tau} \mu(V)+(1-\epsilon)^{-1}\leq \mu(V)+(1-\epsilon)^{-1}.
\end{equation}
By the Markov inequality, for any $\overline{\epsilon}\in ]0,1[$ this implies that
\begin{equation}\label{ref-Markov-ineq-compact}
K_{\overline{\epsilon}}:=
\{x \in E: \overline{\epsilon} V(x)\leq \mu(V)+(1-\epsilon)^{-1}\}\Longrightarrow
\sup_{t\in [0,\infty[_{\tau}}\mu_{t}\left(E-K_{\overline{\epsilon}}\right)\leq  \overline{\epsilon}.
\end{equation}
For continuous time models, for any $\tau>0$ we assume that
\begin{equation}\label{def-c-tau-P}
\pi_{\tau}(V):=\sup_{s\geq 0}\sup_{\delta\in [0,\tau[}\Vert P_{s,s+\delta}(V)/V\Vert<\infty \quad\Longrightarrow\quad
\sup_{t\geq 0}\mu P_{0,t}(V)<\infty.
\end{equation}
We check the right-hand side assertion in the above display using for any $t=[0,\tau[$  the estimate
  $$
  \mu P_{0,n\tau+t}(V)=\mu P_{0,n\tau}(P_{n\tau,n\tau+t}V)\leq \pi_{\tau}(V)~\mu P_{0,n\tau}(V).
  $$}
 {
We further assume $P_{s,t}$ is a continuous time semigroup and (\ref{PW-12}) is met for any $s\geq 0$ and $t=s+\tau$, for some parameter $\epsilon\in [0,1[$, some function  $\alpha : r\in [r_0,\infty[\,\mapsto\alpha(r)\in \, ]0,1]$, and some function $V=V_s=V_{t}$ that may depends on some given parameter $\tau>0$.
In this situation, iterating (\ref{ref-V-contraction}) for any $n\geq 0$ and any  $\mu,\eta\in \Pa_{V}(E)$ we have
\begin{equation}\label{ref-V-contraction-n}
\vertiii{ \mu P_{s,s+n\tau}- \eta P_{s,s+n\tau}}_{V^{\epsilon,r}}\leq   \left(1-\alpha_{\epsilon}(r)\right)^n~ \vertiii{ \mu-   \eta}_{V^{\epsilon,r}}
\end{equation}
with the function
$$
V^{\epsilon,r}:=1/2+\rho_{\epsilon}(r)~V.
$$
Observe that  for any $s\geq 0$ and $t=[0,\tau[$ we have
$$
\vertiii{ \mu P_{s,s+t}-   \eta P_{s,s+t}}_{V^{\epsilon,r}}\leq (1\vee \pi_{\tau}(V))~\vert\mu-\eta\vert(V^{\epsilon,r})= (1\vee \pi_{\tau}(V))~\vertiii{ \mu-   \eta }_{V^{\epsilon,r}}
$$
with the parameter $\pi_{\tau}(V)$ defined in (\ref{def-c-tau-P}).
This yields the estimate
$$
\vertiii{ \mu P_{s,s+t}- \eta P_{s,s+t}}_{V^{\epsilon,r}}\leq   (1\vee \pi_{\tau}(V))~\left(1-\alpha_{\epsilon}(r)\right)^{\lfloor t/\tau\rfloor}~ \vertiii{\mu -   \eta }_{V^{\epsilon,r}}.
$$}

 For time homogeneous semigroups $P_t:=P_{s,s+t}$  the above estimate ensures the existence of a single invariant probability measure $\mu_{\infty}=\mu_{\infty}P_{t}\in \Pa_V(E)$.  \textcolor{black}{The analysis discussed above combines the ergodic theory for Markov operators presented in~\cite{hairer-mattingly} with  the $V$-Dobrushin contraction methodology developed in~\cite{dmpenev}.
  Alternative contraction approaches mainly based on~\cite{hairer-mattingly} are discussed in~\cite{bansaye-2,champagnat-5,marguet,meyn-tweedie}.}

\subsection{Quasi-compact operators}\label{quasi-compact-sec}
Next, we recall some standard definitions and compactness principles from time homogeneous positive semigroup theory.
Let $(Q_{t})_{t \ge 0}$ be a time homogenous positive semigroup  from $\Ba_{V}(E)$ into $\Ba_{V}(E)$. Then $Q_t$ is said to be irreducible (a.k.a.~ideal irreducible) if there exists no closed $Q_t$-invariant ideals distinct from $\{0\}$ and $\Ba_V(E)$ (cf.~\cite[Definition 4.2.1]{peter-meyer}). It is also well known (e.g.~\cite[Proposition 2.1]{drnovsek} or \cite[Proposition 4.1]{gao}) that this condition is met if and only if for any non zero function $f\geq 0$ on $\Ba_V(E)$ and any non zero positive measure $\mu $ on $E$, there exists some $t\in \Ta$ such that
$
 \mu Q_t (f)>0
$.

The spectral radius of $(Q_t)_{t \ge 0}$ is defined as
 \begin{equation}\label{positive-condition-Q-def}
r_{V}(Q):=\lim_{t\rightarrow\infty}\vertiii{Q_t}^{1/t}_{V}=\lim_{t\rightarrow\infty}\Vert Q_t(V)\Vert_V^{1/t}=\inf_{t\geq 0}\Vert Q_t(V)\Vert_V^{1/t}.
\end{equation} 
The semigroup $Q_t$   is said to be quasi-compact if its essential spectral radius 
$$
\overline{r}_V(Q):=\lim_{t\rightarrow\infty}\left(\inf{\left\{\vertiii{Q_t-T}_V~:~T\quad\mbox{\rm compact}\right\}}\right)^{1/t}
$$
satisfies $\overline{r}_V(Q)<r_V(Q)$.
Recalling that the product of positive operators $Q_1Q_2$ is compact as soon as the $Q_1$ is compact (cf.~\cite[Theorem VI.12]{reed-simon}),
the quasi-compactness property of the operator $Q_{t}$ (for sufficiently large $t$) is clearly met  as soon $Q_{\tau}$ is a compact operator for some $\tau\in\Ta$. Such one-parameter semigroups are sometimes called eventually compact semigroups (see \cite[Section 3]{engel}).

Now assume there exists some $\tau\in \Ta$ such that the discrete generation semigroup $Q_{t}$  indexed by $t\in [0,\infty[_{\tau}$ is irreducible and quasi-compact on $\Ba_V(E)$.
By a variant of the Krein-Rutman theorem (cf. for instance Theorem 1.1 in~\cite{du,nussbaum}), $r_{V}(Q)>0$ is an eigenvalue  corresponding to a non null eigenfunction $h\in \Ba_V(E)$.  More precisely, there exists some a  non null function  $h\in \Ba_{V}(E)$ such that
$$
\forall t\in [0,\infty[_{\tau}\qquad 
Q_t(h)=e^{\rho t}~h\quad \mbox{\rm with}\quad
 \rho:=\log{r_V(Q)}.$$
Note that since $Q_{\tau}$ is irreducible, for any $x\in E$  there exists some   $t\in [0,\infty[_{\tau}$ such that
$$
Q_t (h)(x)=e^{\rho t}~h(x)>0.
$$
Whenever  $Q_{\tau}(\Ba_{V}(E))\subset\Ca_{0,V}(E)$, the function $h$ belongs to $\Ca_{0,V}(E)$.

Sufficient conditions in terms of the Lyapunov functions $V\geq 1$ ensuring the compactness of $Q_t$ for some $t\in \Ta$ are discussed in~\cite{ferre,bellet}.
For instance, we have the following  lemma, the proof of which is housed 
in the appendix on page~\pageref{lem-eventually-compact-proof}..

\begin{lem}\label{lem-eventually-compact}
Assume that 
$Q_{\tau}(V)/V\in \Ba_{0}(E)$ for some $\tau>0$. In addition, assume that for any compact set $K\subset E$  the operator $Q_{\tau}^{K}(f):=1_KQ_{\tau}(1_K f)$ is compact on $\Ba_V(E)$. 
In this situation, for any $t\in \Ta$ with $t\geq \tau$, the operator  $Q_t$ is a compact operator from $\Ba_V(E)$ into itself.
\end{lem}

\begin{rmk}\label{rmk-compact-Q-K}
The condition $Q_{\tau}(V)/V\in \Ba_{0}(E)$ allows one to localise the operators on compact sets.
The compactness condition of the semigroup $Q^K_{\tau}$ is readily checked for absolutely continuous operators of the form   
\begin{equation}\label{ref-continuous-Q}
Q_{\tau}(x,dy)=q_{\tau}(x,y)~\nu_{\tau}(dy)
\end{equation}
where $q_{\tau}(x,y)$ is a continuous density with respect to some Radon measure $\nu_{\tau}$ on $E$.  The proof of the above assertion 
is rather standard. For completeness,  it is provided in the appendix on page~\pageref{ref-continuous-Q-proof}.

 The above models encapsulate
Markov transitions restricted to a compact set $K\subset \RR^r$, for some $r\geq 1$, defined by
$$
Q_{\tau}^K(x,dy):=1_K(x)~p_{\tau}(x,y)~1_K(y)~dy
$$
where $p_{\tau}(x,y)$ is a continuous probability transition density with respect to the Lebesgue measure $dy$  on $\RR^r$.
 For a more thorough discussion on this class of compact integral operators we refer to~\cite{sloan} and the more recent article~\cite{castro}.
\end{rmk}

\begin{rmk}\label{rmk-countable}
Now assume that $E$ is a countable space. In this situation, whenever $V=1$ we can also use the following equivalence principle (see for instance Theorem 2.1 in~\cite{zucca})
$$
Q_{\tau}\mbox{ is compact on $\Ba(E)$}\Longleftrightarrow \forall \epsilon>0~\exists K_{\epsilon~}\mbox{ finite s.t. }~
\sup_{x\in E}Q_{\tau}(1_{E-K_{\epsilon}})(x)\leq \epsilon.
$$
 More generally, assume that $Q_{\tau}(V)/V\in \Ba_0(E)$,  In this situation we have
\begin{equation}\label{compact-countable-V}
Q_{t}\quad\mbox{compact on $\Ba_V(E)$}\Longleftrightarrow \forall \epsilon>0\quad \exists K_{\epsilon}\quad\mbox{finite}\quad \mbox{\rm s.t.}\quad
\Vert Q_{t}(1_{K^c_{\epsilon}}V)\Vert_V\leq \epsilon.
\end{equation}
 The proof of the above assertion is provided in the appendix on page~\pageref{compact-countable-V-proof}.

\end{rmk}
Whenever $E$ is compact, by a theorem  of de Pagter (cf.~\cite[Theorem 4.2.2]{peter-meyer}), a compact and irreducible positive operator $Q_{t}$ on $\Ca(E)$ has a positive spectral radius
$r(Q_{t})>0$, while its essential spectral radius is null. Applying the Krein-Rutman theorem (cf.~\cite[Theorem 1.1]{du,nussbaum}),
there exists some non-negative and non-zero measure $\nu_{\infty}\in\Ma(E)=\Ca(E)^{\prime}$ and  a non-negative and non-zero function $h\in \Ca(E)$ such that
\begin{equation}\label{eta-infty-fixed-point}
\nu_{\infty}Q_{t}=e^{\rho t}~\nu_{\infty}\quad \mbox{\rm and}\quad 
Q_{t}(h)=e^{\rho t}~h\quad \mbox{\rm with}\quad \rho=\log{r(Q_{1})}.
\end{equation}
This yields the fixed point equation
$$
\eta_{\infty}:=\nu_{\infty}/\nu_{\infty}(1)=\Phi_{t}(\eta_{\infty}).
$$

Several sufficient conditions in terms of the Lyapunov functions $V\geq 1$ ensuring the quasi-compactness properties of $Q_{\tau}$ are discussed in~\cite{ferre,hennion-2007,bellet}, see also Lemma~\ref{lem-eventually-compact} and Remark~\ref{rmk-compact-Q-K} in the present article. 
For a more detailed discussion on quasi-compact and compact operators we refer to~\cite{engel,nussbaum,peter-meyer,reed-simon} and references therein.

\section{Some classes of positive semigroups}\label{all-sg-sec}
\subsection{Triangular array semigroups}\label{R-sg-sec}

 \textcolor{black}{Fix a measure $\overline{\eta}_0\in \Pa(E)$ and  some locally bounded positive functions $H>0$.} We associate with these objects the
 functions $H_{s,t}$ and the normalised semigroups $\overline{Q}_{s,t}$ defined for any $s\leq t$ by the formulae
\begin{equation}\label{def-over-Q-H}
H_{s,t}:=\overline{Q}_{s,t}(H)\quad\mbox{\rm with}\quad
\overline{Q}_{s,t}(f):={Q_{s,t}(f)}/{\overline{\eta}_sQ_{s,t}(1)}\quad
\quad\mbox{\rm and}\quad
\overline{\eta}_s=\Phi_{0,s}(\textcolor{black}{\overline{\eta}_0}).
\end{equation}
From the definitions, for any $s\leq u\leq t$ it also follows that
\begin{equation}\label{def-over-Q-H-2}
\overline{\eta}_s\overline{Q}_{s,t}=\overline{\eta}_t\quad\mbox{\rm and}\quad
\overline{Q}_{s,u}\overline{Q}_{u,t}=\overline{Q}_{s,t},
\end{equation}
as well as
$$
\overline{\eta}_s(H_{s,t})=\overline{\eta}_t(H)~~\mbox{\rm and}~~H_{t,t}=H.
$$
For any $s\leq u\leq v\leq t$, we also have
$$
Q_{s,v}(H_{v,t})= \lambda_{s,v}~H_{s,t}
\quad\mbox{\rm
with}\quad \lambda_{s,v}:=\overline{\eta}_sQ_{s,v}(1)=
\lambda_{s,u}~\lambda_{u,v}.
$$
Observe that
\begin{equation}
\lambda^-(\overline{\eta}_0)\leq \lambda^-:=\inf_{t\geq 0}\overline{\eta}_t(Q_{t,t+\tau}(1))\leq \lambda:=\sup_{t\geq 0}\overline{\eta}_t(Q_{t,t+\tau}(1))\leq \lambda(\overline{\eta}_0).
\label{inf-sup}
\end{equation}
with the parameters $\lambda^-(\overline{\eta}_0)$ and $\lambda(\overline{\eta}_0)$ introduced in Definition~\ref{def-parameters}.
\begin{defi}
We define the triangular array of Markov operators $R^{(t)}_{u,v}$, indexed by the parameter $t\geq 0$, for any $0\leq u\leq v\leq t$ and  $f\in\Ba_b(E)$  by
\begin{equation}\label{def-R-intro}
R^{(t)}_{u,v}(f):=\frac{Q_{u,v}(f~ Q_{v,t}(H))}{Q_{u,v}(Q_{v,t}(H))}=\frac{1}{\lambda_{u,v}~H_{u,t}}~Q_{u,v}(H_{v,t}f).
\end{equation}
\end{defi}
For any given time horizon $t\geq 0$ and any $s\leq u\leq v\leq t$,  we have  the semigroup property
 $$
 R^{(t)}_{s,v}=R^{(t)}_{s,u}R^{(t)}_{u,v}\quad\mbox{\rm and}\quad  R^{(t)}_{s,s}=I.
 $$

Several remarks are of interest here:
\begin{enumerate} 
 \item The stability properties of these stochastic models play a crucial role in the analysis of positive operators.  
{The use of these triangular arrays of Markov semigroups in the stability analysis of the time varying positive semigroups has been considered in~\cite{dm-2000,dg-ihp,dg-cras,dm-sch-2} \textcolor{black}{and more recently in~\cite{champagnat-7}} in the context of Feynman-Kac semigroups and when $H=1$. We also refer the reader to chapter 4 in~\cite{dm-04}, and chapter 12 in~\cite{dm-13} for a systematic analysis of the contraction properties of these semigroups with respect to the total variation norm.} 
  
  \item \textcolor{black}{The same class of triangular array semigroups associated with some positive function $H$ are also discussed in the article~\cite{bansaye-2}  in the context of time-homogeneous sub-Marko\-vian models, under 
  different set of regularity conditions. We shall discuss these conditions in section~\ref{sec-comparisons}.}
  
  \item {The interpretation of the triangular array semigroups discussed above depends on the application. For time homogeneous semigroups, the semigroups associated with the positive eigenstate $H=h$} of the \textcolor{black}{semigroup $Q_{t}$} coincides with the semigroup of the so-called $h$-process, see for instance (\ref{rho-h}) and (\ref{appli-h-Hst}). {In the context of nonlinear filtering with $H=1$, these semigroups represent the forward evolution of the optimal smoother on some observation time interval~\cite{dm-2000,dg-ihp,dg-cras}. In branching processes theory, these semigroups reflect the evolution of an auxiliary process often used to describe the ancestral lineage of an individual~\cite{bansaye-3,marguet}. The link between these seemingly disconnected subjects is the so-called many-to-one formula connecting the first moment of the occupation measure of a spatial branching process with a Feynman-Kac semigroup (see for instance~\cite{bansaye-3,bertoin,marguet}, as well as \cite[section 1.4.4]{dm-04} and \cite[section 28.4]{dmpenev}).}
\end{enumerate}

We now provide some important Markov transport  formulae relating the semigroups introduced so far.  The first formula connecting the  triangular array semigroup discussed above with the flow of measures $\eta_u=\Phi_{s,u}(\eta_s)$  is given for any $s\leq u\leq t$ by   
$$
\Psi_{H_{s,t}}\left(\eta_s\right)R^{(t)}_{s,u}(f)=\frac{\eta_s(Q_{s,u}(H_{u,t}f))}{\eta_s(Q_{s,u}(H_{u,t}))}=\frac{\eta_u(H_{u,t}f)}{\eta_u(H_{u,t})}.
$$
This yields  the formula
$$
\Psi_{H_{s,t}}\left(\eta_s\right)R^{(t)}_{s,u}=\Psi_{H_{u,t}}\left(\eta_u\right).
$$
Choosing $(u,\eta_s)=(t,\eta)$ so that $\eta_t=\Phi_{s,t}(\eta)$ in the above display, we obtain the next lemma.
\begin{lem}
For any $s\leq t$ and any $\eta\in \Pa(E)$ we have
\begin{equation}\label{link-s-t-Q-intro}
\Psi_{H}\left(\Phi_{s,t}(\eta)\right)=
\Psi_{H_{s,t}}(\eta)R^{(t)}_{s,t}\quad \mbox{and}\quad
\eta_s(H_{s,t})~\Psi_{H_{s,t}}\left(\eta_s\right)(1/H_{s,t})=1,
\end{equation}
with the flow of probability measures given for any $0\leq s\leq u\leq t$ by the formulae
$$
\eta_u:=\Phi_{0,u}(\eta_0)=\Phi_{s,u}(\eta_s)\quad \mbox{and}\quad \Psi_{H_{s,t}}\left(\eta_s\right)R^{(t)}_{s,u}=\Psi_{H_{u,t}}\left(\eta_u\right).
$$
\end{lem}

\subsection{A class of R-semigroups}\label{def-R-sg-sec}
In the further development of this section,  $R^{(t)}_{u,v}$ is the triangular array of Markov operators defined in (\ref{def-R-intro}) with  $H=1$.
\begin{defi}
We say that $Q_{s,t}$ is an  $R$-semigroup as soon as  
 there exists  some parameters $\tau>0$ and  $\epsilon_{\tau}\in ]0,1[$ s.t. for any  $s+\tau\leq t$  we have
\begin{equation}\label{ref-H1-2}
\beta\left(R^{(t)}_{s,s+\tau}\right):=\sup_{(x,y)\in E^2}\left\Vert \delta_xR^{(t)}_{s,s+\tau}-\delta_yR^{(t)}_{s,s+\tau}\right\Vert_{\tiny tv}\leq 1-\epsilon_{\tau}.
\end{equation}
\end{defi}
The parameter $\beta\left(R^{(t)}_{s,s+\tau}\right)$ defined above is called the Dobrushin ergodic coefficient of the Markov transition $R^{(t)}_{s,s+\tau}$.
The above condition  is satisfied  {\em if and only if} the follo\-wing   condition holds
\begin{equation}\label{beta-P-nu-2}
\begin{array}{l}
\forall s\geq 0\quad \forall t\geq s+\tau\quad   \forall (x,y)\in E^2~\\
 \\
 \displaystyle\exists \nu\in \Pa(E)\quad \mbox{such that}\quad\delta_xR^{(t)}_{s,s+\tau}\geq \epsilon_{\tau}~\nu\quad\mbox{\rm and}\quad \delta_yR^{(t)}_{s,s+\tau}\geq \epsilon_{\tau}~\nu.
\end{array}
\end{equation}
Note that the measure $\nu$ in the above display may depend upon the parameters $(x,y)$ as well as $(s,t,\tau)$.

 For Markovian semigroups, this condition reduces to the well-known Dobrushin's condition~\cite{dobrushin-1}. In addition, when the measure $\nu$ in (\ref{beta-P-nu-2}) does not depend on the state variables $(x,y)$, condition (\ref{beta-P-nu-2}) coincides with Doeblin's condition~\cite{doeblin}.

\begin{rmk}\label{rmk-Aa-intro}
The condition $(\Aa)$ discussed in section~\ref{contribution-statements-sec} ensures that for any $x_1,x_2\in E$ and $t\in \Ta$ we have
\begin{eqnarray}\label{ref-Q-x-y}
\iota(\tau)~ Q_{t,t+\tau}(x_1,dy)\leq Q_{t,t+\tau}(x_2,dy)\quad \mbox{\rm with}\quad
\iota(\tau):={\iota^-(\tau)}/{\iota^+(\tau)}.
\end{eqnarray}
Thus, (\ref{ref-H1-2}) is met with $\epsilon_{\tau}:=\iota(\tau)^2$. In addition, for any $t\in \Ta$ and $s\in ]0,\infty[_{\tau}$ and $\epsilon\in [0,\tau]$ we also have
$$
0<\iota(\tau)\leq\frac{Q_{t,u+\epsilon}(1)}{\mu Q_{t,u+\epsilon}(1)}=\frac{Q_{t,u}(Q_{u,u+\epsilon}(1))}{\mu Q_{t,u}(Q_{u,u+\epsilon}(1))}\leq \frac{1}{\iota(\tau)}
\quad\mbox{with}\quad u:=t+s.
$$
Therefore, for any $\mu\in\Pa(E)$ we have the estimates
\begin{equation}\label{def-q-mu}
0<\iota(\tau)\leq q_{\tau}^-(\mu):=\inf\frac{Q_{t,t+s}(1)(x)}{\mu Q_{t,t+s}(1)}\leq q_{\tau}(\mu):=\sup\frac{Q_{t,t+s}(1)(x)}{\mu Q_{t,t+s}(1)}\leq 1/\iota(\tau)
\end{equation}
where the infimum and the supremum are taken over all  $x\in E$ and all pair of
indices $t\in \Ta$ and $s\in [0,\infty[_{\tau}$. For continuous time semigroups, the supremum in (\ref{def-q-mu}) can  be also be taken over all  continuous time indices
$s,t\geq 0$,
as soon as for any $\mu\in \Pa(E)$  we have the uniform local estimate
$$
\sup_{t\geq 0}\sup_{\epsilon\in [0,\tau]}\frac{\Vert Q_{t,t+\epsilon}(1)\Vert}{\mu Q_{t,t+\epsilon}(1)}<\infty.
$$ 
The above condition is clearly met as soon as (\ref{hyp-d-cont-A}) is satisfied.
\end{rmk}

 To the best of our knowledge, the  use of 
the triangular arrays $R^{(t)}_{s,u}$ and the extension of Dobrushin's contraction stability theory to time-varying and possibly random positive semigroups goes back to the 1990s with the article~\cite{dg-cras}, see also~\cite{dm-04,dm-13,dm-2000,dg-ihp,dm-sch-2}. 

Dobrushin's and Doeblin's conditions are popular conditions in probability theory. They are not always easily checked but are met for a large class of discrete or continuous time irreducible stochastic processes, mainly on compact domains.  For instance, in the context of continuous time sub-Markovian semigroups this condition is satisfied for elliptic diffusions on compact manifolds killed at a bounded rate, as well as elliptic diffusions killed at the boundary of a bounded domain, cf.~\cite[Proposition 3.1]{dm-sch-2} and \cite[Lemma 3]{dm-villemonais}. Combining the proof of \cite[Proposition 3.1]{dm-sch-2} with the two-sided estimates presented in \cite[Lemma 3.3]{bansaye}, one can check that (\ref{beta-P-nu-2}) is also met for sub-Markovian  semigroups  associated with reflected diffusions  on bounded connected domains. 

Condition (\ref{beta-P-nu-2}) is also met for absorbed time homogenous neutron transport processes for sufficiently smooth domains with an absorbing boundary, as well as for time varying multidimensional birth and death processes with mass extinction, see \cite[sections 4.1 \& 4.2]{champagnat}.   Further examples of sub-Markovian semigroups satisfying (\ref{beta-P-nu-2}) are discussed in~\cite{champagnat-6,champagnat-7,dm-04,dm-13,dm-doucet-04,dm-2000,dm-jasra-18}.
Sufficient conditions and further examples are discussed in section~\ref{sec-comparisons}.

\subsection{A class of $V$-positive semigroups}\label{uf-sg-sec}
 In the further development of this section $V\in \Ba_{\infty}(E)$ stands for some function such that $V_{\star}\geq 1$.
 \begin{defi}
 A semigroup of positive integral operators $Q_{s,t}$ on $\Ba_b(E)$ is called a
 $V$-positive semigroup  as soon as there exists some $\tau>0$ and some  function
 $\Theta_{\tau}\in \Ba_0(E)$ such that for any positive function $f\in\Ba_V(E)$ and any $0\leq s<t$ we have
 \begin{equation}\label{ref-V-theta-V-po}
 Q_{s,t}(f)\in \Ba_{0,V}(E)\quad \mbox{and}\quad Q_{s,s+\tau}(V)/V\leq \Theta_{\tau}. 
\end{equation}
 For continuous time semigroups, we assume that
(\ref{discrete-2-continuous-intro-K}) and (\ref{discrete-2-continuous}) are satisfied.

 \end{defi}

The  function $\Theta_{\tau}$ in condition (\ref{ref-V-theta-V-po}) is used to control, uniformly in time, the compact super-level sets of the time varying u.s.c. 
function $ Q_{t,t+\tau}(V)/V$ in terms 
of a time homogenous function $\Theta_{\tau}\in \Ba_0(E)$ that often depends on $V$. 
For instance, we can choose in (\ref{ref-V-theta-V-po}) a function $\Theta_{\tau}$ of the form $\theta_{\tau}(V)\in \Ba_0(E)$ for some  decreasing function $\theta_{\tau}:[1,\infty[\rightarrow\RR_+$.
 Indeed, for  any $\Theta_{\tau}\in \Ba_0(E)$ and $\epsilon>0$ there exists some  $\epsilon_1,\epsilon_2>0$  such that
\begin{equation}\label{def-V-epsilon}
\Va_{\epsilon}:= \{\Theta_{\tau}\geq \epsilon\}\subset \{V\leq \theta^{-1}(\epsilon_1)\}=\{\theta_{\tau}(V)\geq \epsilon_1\}\subset \Va_{\epsilon_2}.
\end{equation}
Similarly, for any $r>1$ there exists some $\epsilon>0$ and $r_1>1$ such that
\begin{equation}\label{def-V-epsilon-2}
\{V\leq r\} \subset \Va_{\epsilon} \subset \{V\leq r_1\}.
\end{equation}
 
The right-hand side~condition in (\ref{ref-V-theta-V-po}) ensures that for any $t\geq \tau$ and $s\geq 0$ we have
 $$
  Q_{s,s+t}(V)/V\leq  c_{s,t}(\tau) ~\Theta_{\tau}\in \Ba_0(E)\quad \mbox{with}\quad c_{s,t}(\tau):=\Vert Q_{s+\tau,s+t}(V)/V\Vert.
 $$
Condition (\ref{ref-V-theta-V-po}) ensures that  the normalized semigroup $\Phi_{s,t}$ maps $\Pa_V(E)$ into itself, and the right action operator $f\mapsto Q_{s,t}(f)$  maps $\Ba_V(E)$ into itself.

\subsection{A class of stable $V$-positive semigroups}\label{V-sg-sec}
\begin{defi}\label{ref-def-V-sg-intro}
A given $V$-positive semigroup $Q_{s,t}$ with respect to the Lyapunov function $V$ and some parameter $\tau>0$  is said to be stable  when the following conditions are satisfied:
\begin{itemize}
\item  There exists some  $\overline{\eta}_0\in \Pa_V(E)$  such that 
\begin{equation}\label{ref-V-over-eta}
\lambda^-(\overline{\eta}_0)>0\quad\mbox{and for any $\mu\in \Pa_V(E)$ we have} \quad\kappa_V(\mu)<\infty. 
\end{equation}
with the parameters $\lambda^-(\overline{\eta}_0)$ and $\kappa_V(\mu)$ introduced in Definition~\ref{def-parameters}. 
\item There exists some  positive function $H\in \Ba_{0,V}(E)$ as well as some $r_H>0$ such that for any $r\geq r_H$ we have
 \begin{equation}\label{ref-H0-2}
 \varsigma_{r}(H):=\inf\inf_{V(x)\leq r}~H_{s,t}(x)>0
\quad \mbox{and}\quad\vertiii{H}_V:=\sup\Vert H_{s,t}\Vert_V<\infty
\end{equation}
where the infimum is taken over all  indices $s\in\Ta$ and $s\leq t\in \Ta$, and $H_{s,t}$ stands for the normalized function depending on the measure $\overline{\eta}_0$ introduced in (\ref{ref-V-over-eta}) and defined in (\ref{def-over-Q-H}). For continuous time semigroups, we also assume that $\pi_{\tau}(H)<\infty$, with $\pi_{\tau}(H)$ defined as (\ref{discrete-2-continuous}) by replacing $V$ by $H$.
\item There exists $r_0\geq 1$ and some function $\alpha\,:\,r\in [r_0,\infty[\,\mapsto \alpha(r)\in \, ]0,1]$ such that for any $s\in \Ta$ and $r\geq r_0$ we have
\begin{equation}\label{loc-dob}
\sup_{V(x)\vee V(y)\leq r}\left\Vert \delta_xR^{(t)}_{s,s+\tau}-\delta_yR^{(t)}_{s,s+\tau}\right\Vert_{\tiny tv}\leq 1-\alpha(r)
\end{equation}
where $R^{(t)}_{u,v}$ stands for the triangular array of Markov operators defined in (\ref{def-R-intro}) with the function $H$ introduced in (\ref{ref-H0-2}).
\end{itemize}
\end{defi}

The local Doeblin minorisation condition (\ref{loc-dob}) is rather standard in the stability analysis of Markov semigroups (see for instance \cite[Theorem 6.15]{numelin}, \cite[Proposition 2]{tierney}, and \cite[Theorem 16.2.3]{meyn-tweedie}). For Markov semigroups $Q_{s,t}(1)=1$, choosing $H=1$ the left-hand side condition in (\ref{ref-V-over-eta}) and (\ref{ref-H0-2}) are clearly met. In this context, we recall that the Lyapunov condition in (\ref{ref-V-theta-V-po}) ensures that $\Phi_{s,t}(\mu)(V)$ is uniformly bounded with respect to the time parameters $s\leq t$ (cf. for instance Lemma~\ref{lemma-Nick} or the Lyapunov condition (\ref{ref-P-h-nh}) applied to the function $H=1$). 

In summary, a Markov semigroup $P_{s,t}$ is a stable $V$-positive semigroup when  $P_{s,s+\tau}$ satisfies the Lyapunov condition in (\ref{ref-V-theta-V-po}) as well as the local Doeblin minorisation condition (\ref{loc-dob}) (applied to $H=1$ and $Q_{s,t}=P_{s,t}$ so that $R^{(t)}_{s,s+\tau}=P_{s,s+\tau}$).

\begin{rmk}\label{rmk-discrete-2-continuous}
For continuous time models, by lemma~\ref{lem-d-to-c} , it suffices to check conditions (\ref{ref-V-over-eta}) with the parameters $\lambda^-_{\tau}(\overline{\eta}_0)$ and 
$\kappa_{\tau,V}(\mu)$  introduced in definition~\ref{def-parameters}. 
In the same vein, it suffices to check (\ref{ref-H0-2}) by taking the infimum and the supremum over time indices $s\in \Ta$ and $t\in [s,\infty[_{\tau}$. To check this claim, observe that for any $s\geq 0$, $n\in\NN$ and $t\in [0,\tau[$ we have
$$
\begin{array}{l}
s_n=s+n\tau\quad \mbox{and}\quad u:=s+t\\
\\
\displaystyle\Longrightarrow u_n:=u+n\tau=s_n+t\quad \mbox{and}\quad
H_{s,s_n+t}=\frac{Q_{s,u} Q_{u,u_n}(H)}{\overline{\eta}_s Q_{s,u} Q_{u,u_n}(1)}.
\end{array}$$
Using (\ref{ref-H0-2}) and (\ref{discrete-2-continuous-intro-K}) for any $x\in K_r:=\{V\leq r\}$ and $\overline{r}\geq r\geq r_H$ we check that~
$$
H_{s,s_n+t}(x)=\frac{Q_{s,u} (H_{u,u_n})(x)}{\overline{\eta}_s Q_{s,u}(1)}\geq (\varsigma_{\overline{r}}(H)/\pi_{\tau})~Q_{s,u} (1_{K_{\overline{r}}})(x)\geq \varsigma_{\overline{r}}(H)~\pi^-_{\tau}(K_{r,\overline{r}})/\pi_{\tau}.
$$
with the parameter $\pi_\tau^-(K_{r,\overline{r}})$ defined in (\ref{discrete-2-continuous-intro-K}).

Correspondingly, recalling that $\overline{\eta}_s$ is a tight sequence, for any $\delta\in ]0,1[$ there exists some $K$ such that $\overline{\eta}_s(K)\geq (1-\delta)$. In this notation, by (\ref{discrete-2-continuous-intro-K})we have
$$
H_{s,s_n+t}(x)/V(x)\leq \frac{Q_{s,u} (V(H_{u,u_n}/V))(x)/V(x)}{\overline{\eta}_s Q_{s,u}(1)}\leq \frac{\pi_{\tau}(V)~\Vert H_{u,u_n}\Vert_V}{(1-\delta)\pi^-_{\tau}(K)}.
$$
with the parameter $\pi_\tau^-(K_{r})$ defined in (\ref{pi-min}).
We conclude  that $\vertiii{H_{s,t}}_V$ is uniformly bounded with respect to~the continuous time parameters $s\geq 0$ and $t\geq s$.
\end{rmk}

Also note that for any compact set $K\subset E$ we have
$$
\left(V\geq 1\quad \mbox{\rm and}\quad
H/V\in \Ba_{0}(E)\right)\Longrightarrow {1}/{\inf_K H}\leq \sup_K(V/H)<\infty
\Longrightarrow \inf_K H>0.
$$
Whenever (\ref{ref-H0-2}) is met, replacing $H$ with $H/\vertiii{H}_V$, there is no loss of generality if we  assume the uniform estimate $\Vert H_{s,t}\Vert_V\leq 1$.
Also notice that  for any $s<t$ and
$$
H\in \Ba_{0,V}(E)\Longrightarrow
H_{s,t}={Q_{s,t}(H)}/{\overline{\eta}_sQ_{s,t}(1)}
\in\Ba_{0,V}(E).
$$

Checking the estimates (\ref{ref-V-over-eta}) and (\ref{ref-H0-2}) may involve delicate calculations. For time homogeneous models, conditions (\ref{ref-V-over-eta}) and (\ref{ref-H0-2})  are
equivalent to the existence of a leading eigen-pair of the positive semigroup (cf.~for instance Theorem~\ref{theo-equivalence} and Corollary~\ref{H1-prop}). For a more thorough discussion, we refer to  section~\ref{sec:timehom} dedicated to time homogenous models.
Next we present another stronger but simpler and more tractable condition that applies to absolutely continuous semigroups.

The following lemma  is a slight modification of  \cite[Proposition 1 \& Lemma 10]{whiteley}, based on technical approaches from~\cite{douc-moulines-ritov} in the context of stability for nonlinear filtering 

\begin{lem}\label{lemma-Nick}
Under the assumptions of theorem~\ref{theo-2-intro},
for any  locally bounded  positive functions $H\in \Ba_{0,V}(E)$ and any  $\overline{\eta}_0,\mu\in \Pa_V(E)$, we have (\ref{ref-H0-2}) as well as 
\begin{equation}\label{ref-P-h-nh-2}
0<\kappa^-_H(\mu)\leq \kappa_V(\mu)<\infty\quad \mbox{and}\quad
0<\lambda^-(\mu)\leq \lambda(\mu)<\infty.
\end{equation}
\end{lem}

For the convenience of the reader a detailed proof in our context is provided in the appendix, see section \ref{app:lemma-Nick-proof}. In the context of
absolutely continuity,
the above lemma also ensures for any bounded 
$f\geq 0$, and any $V(x)\leq r$  with $ r\geq r_1$ we have
$$
c_1(r)~\nu_{\tau}(1_{V\leq r}~H_{s+\tau,t}f)\leq Q_{s,s+\tau}(H_{s+\tau,t}f)(x)$$
as well as
\begin{eqnarray*}
Q_{s,s+\tau}(H_{s+\tau,t})(x)&=&\lambda_{s,s+\tau}~ H_{s,t}(x)\leq r~\lambda~\vertiii{H}_V\\
&\leq& \left(1\vee\frac{ r~\lambda~\vertiii{H}_V}{ \varsigma_{r}(H)\nu_{\tau}(V\leq r)}\right)~\varsigma_{r}(H)\nu_{\tau}(V\leq r)~
\leq c_2(r)~\nu_{\tau}(1_{V\leq r}~H_{s+\tau,t})
\end{eqnarray*}
for some constant $c_2(r)\geq c_1(r)>0$. This yields
$$
R^{(t)}_{s,s+\tau}(f)(x)\geq c(r)~\frac{\nu_{\tau}(1_{V\leq r}~H_{s+\tau,t}f)}{\nu_{\tau}(1_{V\leq r}~H_{s+\tau,t})},\quad \mbox{\rm with}\quad c(r):=c_1(r)/c_2(r)>0,
$$
which implies (\ref{loc-dob}). Thus, the above lemma ensures that absolutely continuous semigroups  satisfying condition $(\Aa)_V$ are stable  $V$-positive semigroups.

\section{Stability and contraction theorems}\label{sec:not_res}

\subsection{Contraction of $R$-semigroups}\label{sec:not_res-unif}
In the further development of this section, $R^{(t)}_{u,v}$ is the triangular array of Markov operators defined in (\ref{def-R-intro}) with  $H=1$.
Also assume that $Q_{s,t}$ is an  $R$-semigroup (that is  (\ref{beta-P-nu-2}) is satisfied with
$H=1$).
\subsubsection{An uniform stability theorem}
This short section is concerned with a brief review on the stability properties of the non-linear semigroup 
$\Phi_{s,t}$ under the uniform minorisation condition (\ref{beta-P-nu-2}). We recall  the rather well known strong stability theorem which is valid when $E$ is a measurable space.
\begin{theo}[\cite{dm-2000,dg-ihp,dg-cras}]\label{theo-1}
Then for any  $s,t\in \Ta$ and any $\mu_1,\mu_2\in\Pa(E)$ we have the uniform stability estimate
\begin{equation}\label{beta-sup}
\Vert \Phi_{s,t}(\mu_1)-\Phi_{s,t}(\mu_2)\Vert_{\tiny tv}\leq 
(1-\epsilon_{\tau})^{{\lfloor (t-s) /\tau \rfloor}}.
\end{equation}
In addition we have the local Lipschitz estimate
\begin{equation}\label{lipschitz-inq}
\Vert \Phi_{s,t}(\mu_1)-\Phi_{s,t}(\mu_2)\Vert_{\tiny tv}\leq \frac{\Vert Q_{s,t}(1)\Vert}{\mu_1(Q_{s,t}(1))\vee\mu_2(Q_{s,t}(1))}
(1-\epsilon_{\tau})^{{\lfloor (t-s)/\tau\rfloor}}\Vert \mu_1-\mu_2\Vert_{\tiny tv}.
\end{equation}
\end{theo}
A detailed and remarkably simple proof of Theorem~\ref{theo-1}  based on the nonlinear transport formula (\ref{link-s-t-Q-intro}) is provided in~\cite[Theorem 3]{dm-2000}, see also \cite[Lemma 2.1 \& Lemma 2.3]{dg-ihp}, \cite[Lemma 2.1 \& Proposition 2.3]{dm-sch-2}, \cite[section 12.2]{dm-13},  \cite[section~2.1.2 \& section 3.1.3]{dm-2000} and \cite[section 4]{dm-04}. 
The key semigroup oscillation formula \cite[Theorem 2.3]{dm-2000}, \cite[Proposition 2.3]{dm-sch-2}  connecting the  {\em uniform} exponential decay (\ref{beta-sup}) with Dobrushin's ergodic coefficient of the $R$-semigroup  defined in section~\ref{sec-preliminary-V-dob} is given by
$$
\beta\left(R^{(t)}_{s,t}\right)=\sup_{(\mu,\eta)\in \Pa(E)^2}\Vert \Phi_{s,t}(\mu)-\Phi_{s,t}(\eta)\Vert_{\tiny tv}.
$$ 
The above formula shows that the  {\em uniform} exponential decays (\ref{beta-sup}) are dictated by the contraction properties of the triangular arrays $R^{(t)}_{s,u}$. 
An extended version of Theorem~\ref{theo-1}  for
general relative entropy criteria is provided in \cite[section 4.3.1]{dm-04}. 

\subsubsection{Quasi-invariant measures}
For the rest of this section, we place ourselves in the time homogeneous setting. 
In this case, a variety of results follows almost immediately from the \textcolor{black}{uniform} estimates obtained in this theorem. Before stating them, we first discuss some relevant properties of time homogeneous models.

The uniform estimate (\ref{beta-sup}) implies that $\Phi_{t}(\mu)$ is a Cauchy sequence in the complete set of probability measures $\Pa(E)$ equipped with the total variation distance. Thus, for any $\mu\in \Pa(E)$, the flow $\Phi_{s}(\mu)$ converges, as \textcolor{black}{$ s\rightarrow \infty$},  exponentially fast to a single probability measure $\eta_{\infty}\in\Pa(E)$ that does not depend on $\mu$. 
Choosing $\mu=\eta_{\infty}$ for any $s,t\in \Ta$ and $f\in \Ba_b(E)$ we have
$$
\Phi_t(\Phi_s(\eta_{\infty}))(f)=\frac{\Phi_s(\eta_{\infty})Q_t(f)}{\Phi_s(\eta_{\infty})Q_t(1)}=\Phi_{s+t}(\eta_{\infty})(f)\longrightarrow_{s\rightarrow \infty} \Phi_t(\eta_{\infty})(f)=\eta_{\infty}(f).
$$
 In continuous time settings, note that  the fixed point $\eta_{\infty}^{[\tau]}=\Phi_{\tau}(\eta_{\infty}^{[\tau]})$ does not depend on the time step $\tau>0$. To check this claim, note that for any $\mu\in \Pa(E)$ and $t\geq 0$ we have the decomposition
$$
\eta_{\infty}^{[\tau_1]}-\eta_{\infty}^{[\tau_2]}=
\left(\eta_{\infty}^{[\tau_1]}-\Phi_{\lfloor t/\tau_1\rfloor\tau_1}\left(
\Phi_{\{ t/\tau_1\}\tau_1}\left(\mu\right)
\right)\right)+\left(\Phi_{\lfloor t/\tau_2\rfloor\tau_2}\left(
\Phi_{\{ t/\tau_2\}\tau_2}\left(\mu\right)
\right)-\eta_{\infty}^{[\tau_2]}\right).
$$
The uniform estimate (\ref{beta-sup}) yields the estimate
$$
\Vert\eta_{\infty}^{[\tau_1]}-\eta_{\infty}^{[\tau_2]}\Vert_{\tiny tv}\leq (1-\epsilon_{\tau_1})^{{\lfloor t/\tau_1 \rfloor}}+ (1-\epsilon_{\tau_2})^{{\lfloor t/\tau_2 \rfloor}}\longrightarrow_{t\rightarrow\infty}0.
$$
The invariant measure $\eta_{\infty}$ is sometimes called the quasi-invariant measure of the semigroup $Q_t$.

Using the fixed point equation $\Phi_t(\eta_{\infty})=\eta_{\infty}$, for any $s,t\in\Ta$ we readily check that
$$
\eta_{\infty}(Q_{s+t}(1))=\eta_{\infty}(Q_{s}(1))~\eta_{\infty}(Q_{t}(1)).
$$
Thus for any $t\in\Ta$ we have the exponential formula
\begin{equation}\label{expo-rho-def}
\eta_{\infty}(Q_{t}(1))=e^{\rho t}\quad\mbox{\rm for some}\quad \rho\in \RR \quad\mbox{\rm and}\quad \overline{Q}_{t}(1)=e^{-\rho t}~Q_t(1).
\end{equation}
Choosing $\overline{\eta}_0=\eta_{\infty}$ in (\ref{def-over-Q-H}) and
using (\ref{product-hom}), for any $\mu\in \Pa(E)$ we check that
\begin{equation}\label{expo-norm-ineq}
\frac{1}{\Vert \overline{Q}_{(1-\left\{t/\tau\right\})\tau}(1)\Vert }~\mu\overline{Q}_{(\lfloor t/\tau\rfloor+1)\tau}(1)
\leq \mu\overline{Q}_t(1)\leq \Vert \overline{Q}_{\left\{t/\tau\right\}\tau }(1)\Vert~\mu\overline{Q}_{\lfloor t/\tau\rfloor \tau}(1).
\end{equation}
Notice that
$$
\sup_{\epsilon\in [0,1]}\Vert\overline{Q}_{\epsilon\tau}(1)\Vert=e^{-\rho\epsilon\tau}~\sup_{\epsilon\in [0,1]}\Vert Q_{\epsilon\tau}(1)\Vert<\infty.
$$
On the other hand, for any $t\in [0,\infty[_{\tau}$  we have
\begin{equation}\label{prod-series}
\mu\overline{Q}_t(1)=\prod_{s\in [0,t[_{\tau}}~\left\{1+
\left[\Phi_{s}(\mu)(\overline{Q}_{\tau}(1))-\Phi_{s}(\eta_{\infty})(\overline{Q}_{\tau}(1))\right]\right\}.
\end{equation}
\textcolor{black}{The exponential version of the above formula in the context of sub-Markovian semigroups with soft killing is discussed in (\ref{int-exp}). }

Conversely, by (\ref{beta-sup}) \textcolor{black}{ for any $n\geq 1$} we have
$$
\sum_{s\in [0,\infty[_{\tau}}\left\vert\Phi_{s}(\mu)(\overline{Q}_{\tau}(1))-\Phi_{s}(\eta_{\infty})(\overline{Q}_{\tau}(1))\right\vert^n<\infty
$$
as well as
$$
\Vert \overline{Q}_t(1)\Vert\leq \prod_{s\in [0,\infty[_{\tau}}~\left(1+(1-\epsilon_{\tau})^{{\lfloor s/\tau\rfloor}}
{\Vert \overline{Q}_{\tau}(1)\Vert}\right)<\infty.
$$
The two estimates discussed above ensure that $\mu\overline{Q}_{t}(1)$ converges, as $[0,\infty[_{\tau}\ni t\rightarrow \infty$, to a non-zero number and by (\ref{expo-norm-ineq}) we have
\begin{equation}\label{con-unif}
 0<\inf_{t\in\Ta}\mu(\overline{Q}_{t}(1))\leq
 \sup_{t\in\Ta}\Vert\overline{Q}_{t}(1)\Vert<\infty.
 \end{equation}
Also observe that (\ref{con-unif}) implies that for any $ \mu\in \Pa(E)$ we have
$$
q(\mu):=\sup_{t\in\Ta}{{\Vert Q_{t}(1)\Vert}/{\mu(Q_{t}(1))}}<\infty.
$$
We are now in a position to state our first corollary of Theorem~\ref{theo-1}.

\begin{cor}\label{est-mui-hom-cor}
Under the assumptions of Theorem~\ref{theo-1},
 for any $\mu,\eta\in \Pa(E)$ and $t\in \Ta$ we have the local contraction estimate
\begin{equation}\label{est-mui-hom}
\displaystyle\Vert \Phi_{t}(\mu)- \Phi_{t}(\eta)\Vert_{\tiny tv}\leq \left(q(\mu)\wedge q(\eta)\right)~(1-\epsilon_{\tau})^{{\lfloor t/\tau\rfloor }}~\Vert \mu-\eta\Vert_{\tiny tv}.
\end{equation}
\end{cor}

\subsubsection{Ground state functions}
Choosing $\mu=\delta_x$ in (\ref{prod-series}), we also readily check  that there exists $h\in \Ba_b(E)$ such that $\eta_{\infty}(h)=1$ and for any $x\in E$ we have the pointwise convergence
$$
\lim_{n \to \infty}\overline{Q}_{n\tau}(1)(x)=  h(x)>0.
$$
By (\ref{con-unif}) and the dominated convergence theorem, for any $s\in [0,\infty[_{\tau}$  this implies that
\begin{equation}\label{link-rmk-h-rho}
\overline{Q}_{s+n\tau}(1)(x)=\overline{Q}_s(\overline{Q}_{n\tau}(1))(x)\longrightarrow_{n\rightarrow\infty} h(x)=\overline{Q}_s(h)(x)=e^{-\rho s}~Q_s(h)(x).
\end{equation}
In continuous time settings, note that the ground state function $h^{[\tau]}=e^{\rho \tau}~Q_{\tau}(h^{[\tau]})$ does not depend on the time step $\tau>0$.  To check this claim, observe that
$$
Q_t(h^{[\tau]})=e^{\rho \tau}~Q_{\tau}(Q_t(h^{[\tau]}))
$$
is also an eigenfunction of $Q_{\tau}$ with the same eigenvalue $e^{\rho \tau}$. By uniqueness we conclude that $Q_t(h^{[\tau]})=e^{\rho t}h^{[\tau]}$.
Alternatively,  applying the uniform estimate 
(\ref{beta-sup}) to $\mu=\delta_x$ and $\eta=\eta_{\infty}$  for any $s,t\in \Ta$ and $x\in E$ we have
$$
\vert{\overline{Q}_{t+s}(1)(x)}/{\overline{Q}_{t}(1)(x)}-1\vert=\vert\Phi_t(\delta_x)(\overline{Q}_s(1))-1\vert\leq 2~q(\eta_{\infty})~(1-\epsilon_{\tau})^{\lfloor t/\tau\rfloor}
$$
and therefore
\begin{equation}\label{ref-Q1-cauchy}
\Vert{\overline{Q}_{t+s}(1)}-\overline{Q}_{t}(1)\Vert\leq 2~q(\eta_{\infty})^2~(1-\epsilon_{\tau})^{\lfloor t/\tau\rfloor}.
\end{equation}
This shows that $\overline{Q}_{t}(1)$ is an uniformly bounded Cauchy sequence in $\Ba_b(E)$ that converges to $h\in \Ba_b(E)$ as $t\rightarrow\infty$.
If in addition $E$ is a Polish space and $Q_{t}$ is Feller, that is $Q_{t}(\Ca_{b}(E))\subset\Ca_b(E)$ arguing as above we also check that $h\in\Ca_b(E)$.
The eigenfunction $h$ is sometimes called the ground state of the semigroup $Q_t$.
  
Choosing $H=h$ in (\ref{def-R-intro}), formula
 (\ref{link-s-t-Q-intro}) reads
\begin{equation}\label{link-h-phi}
\Psi_{h}(\Phi_{t}(\eta))=
\Psi_{h}(\eta)P^h_t\quad \mbox{\rm with}\quad P^h_s(f):={Q_{s}( hf)}/{Q_{s}(h)}=R^{(t)}_{0,s}(f).
\end{equation}
Also observe that
$$
\delta_xP^h_t=\Psi_h\left(\Phi_t(\delta_x)\right).
$$
The second corollary is a direct consequence of Theorem~\ref{theo-1}. 
\begin{cor}\label{cor-intro-h-0-0}
There exists a positive function $h\in \Ba_b(E)$ and a constant $\rho\in\RR$ such that
 for any  $t\in \Ta$ we have 
\begin{equation}
Q_t(h)=e^{\rho t}~h\quad\mbox{and}\quad\eta^h_{\infty}P^h_t= \eta^h_{\infty}
\quad\mbox{with}\quad
 \eta^h_{\infty}:=\Psi_h(\eta_{\infty}).
\end{equation}
In addition, we have the total variation exponential decays
\begin{equation}\label{add-expo-intro-Ph}
\Vert\delta_xP^h_t-\eta^h_{\infty}\Vert_{\tiny tv}\leq\frac{\Vert h\Vert}{\eta_{\infty}(h)}~(1-\epsilon_{\tau})^{\lfloor t/\tau\rfloor }.
\end{equation}
\end{cor}

By (\ref{ref-Q1-cauchy}) we have
$$
\Vert\overline{Q}_{t}(1)-h/\eta_{\infty}(h)\Vert=\lim_{s\rightarrow\infty}
\Vert\overline{Q}_{t}(1)-\overline{Q}_{t+s}(1)\Vert\leq 2~q(\eta_{\infty})^2~(1-\epsilon_{\tau})^{\lfloor t/\tau\rfloor}.
$$
On the other hand, for any $f\in \Ba_b(E)$ with $\Vert f\Vert\leq 1$ and any $x\in E$ we have
\begin{eqnarray*}
\vert\overline{Q}_{t}(f)(x)-\frac{h(x)}{\eta_{\infty}(h)}~\eta_{\infty}(f)\vert & \leq & \vert\overline{Q}_{t}(1)(x)-h(x)/\eta_{\infty}(h)\vert\\ & &\hskip3cm+\frac{\Vert h\Vert}{\eta_{\infty}(h)}~\Vert\Phi_t(\delta_x)(f)-\eta_{\infty}(f)\vert.
\end{eqnarray*}
This leads us to our final corollary of the section. 
\begin{cor}\label{cor-intro-h-0}
For any $t\in\Ta$,  we have the operator norm exponential decays
(\ref{add-expo-intro}) with $c(\eta_{\infty})=q(\eta_{\infty})$,
as well as for any $s,t\in \Ta$ and $\eta\in \Pa(E)$ the uniform total variation estimates
\begin{equation}\label{add-expo-intro-2}
\begin{array}{l}
\displaystyle\Vert\Psi_{Q_t(1)}(\Phi_s(\eta))-\Psi_h(\eta_{\infty})\Vert_{\tiny tv}\\
\\
\displaystyle\qquad \qquad\leq q(\eta_{\infty})~
(1-\epsilon_{\tau})^{{\lfloor s/\tau\rfloor}} + 2~(1-\epsilon_{\tau})^{{\lfloor t/\tau\rfloor }}~\left(\Vert h\Vert/\eta_{\infty}(h)+~q(\eta_{\infty})^2
\right).
\end{array}
\end{equation}
\end{cor}
The last assertion comes from the decomposition
$$
\begin{array}{l}
\displaystyle
\overline{\eta}_0=\eta_{\infty}
\Longrightarrow
\Psi_{Q_t(1)}(\mu)(f)-\Psi_h(\eta_{\infty})(f)\\
\\
\displaystyle\hskip2.5cm=[\Psi_{\overline{Q}_t(1)}(\mu)-\Psi_{\overline{Q}_t(1)}(\eta_{\infty})](f)+\eta_{\infty}(f(\overline{Q}_t(1)-h/\eta_{\infty}(h))).
\end{array}$$
The estimate (\ref{add-expo-intro-2}) is often used in the analysis of the ergodic properties
of particle absorption models, see for instance~\cite{champagnat-3,he} as well as equation (\ref{add-expo-apply}) in the present article.

The above corollary can be interpreted as an extended version of the Krein-Rutman theorem to 
positive semigroups satisfying the uniform minorisation condition (\ref{beta-P-nu-2}). In this connection, note that
Corollary~\ref{cor-intro-h-0} readily yields the uniqueness of the eigenfunction $h$ (up to some constant). Indeed, by (\ref{add-expo-intro}) we have
\begin{equation}\label{uniqueness-ref}
\left(\forall t\in\Ta\qquad
\overline{Q}_{t}(g)=g\in \Ba_b(E) \right)\Longrightarrow g=\frac{\eta_{\infty}(g)}{\eta_{\infty}(h)}~h.
\end{equation}
Letting $T_t:=e^{\rho t}T$, we also have
\begin{eqnarray}
\lim_{t\rightarrow\infty}\vertiii{Q_{t}-T_t}^{1/t}&\leq& e^{\rho}~(1-\epsilon_{\tau})^{1/\tau}\nonumber\\
&<&e^{\rho}=\eta_{\infty}(Q_t(1))^{1/t}\leq \Vert Q_t(1)\Vert^{1/t}\longrightarrow_{t\rightarrow\infty}\lim_{t\rightarrow\infty}\vertiii{Q_t}^{1/t}\label{ref-Q-1-rho}
\end{eqnarray}
where $\vertiii{\cdot}$ denotes the operator norm. The estimates stated in (\ref{ref-Q-1-rho}) ensure that the essential spectral radius of $Q_t$ is strictly smaller than its spectral radius. For a more thorough discussion on these spectral quantities, we refer to section~\ref{quasi-compact-sec}.

A more refined analysis under weaker conditions, including exponential stability theorems and contraction properties of time homogeneous semigroups, is provided in section~\ref{sec:timehom}.

\subsection{Contraction of stable $V$-positive semigroups}\label{sec:timeinhom}
 In the further development of this section $V\in \Ba_{\infty}(E)$ stands for  some function such that $V_{\star}\geq 1$.  We consider a stable   $V$-positive semigroup $Q_{s,t}$  satisfying  the Lyapunov condition (\ref{ref-V-theta-V-po}) for some $\tau>0$ as well as (\ref{ref-V-over-eta}) and (\ref{ref-H0-2}) for some 
 measure $\overline{\eta}_0\in \Pa_V(E)$ and some positive function $H\in\Ba_{0,V}(E)$.
 In addition, $H_{s,t}$ and $R^{(t)}_{u,v}$  stands for the corresponding normalized function defined in (\ref{def-over-Q-H}) and  the triangular array of Markov operators defined in (\ref{def-R-intro}). 
 
  For continuous time semigroups, we recall  that $\pi_{\tau}(H)<\infty$ and
(\ref{discrete-2-continuous-intro-K}) and (\ref{discrete-2-continuous}) are satisfied. By Remark~\ref{rmk-discrete-2-continuous}, these conditions ensures that all infimum and supremum in (\ref{ref-H0-2}) as well as in the definition of the parameters $\lambda^-(\overline{\eta}_0)$ and $\kappa_V(\mu)$ defined in (\ref{kappa-V-h-intro})
 can be taken over continuous time indices.
 The proof of the following theorem and its corollary can be found in section~\ref{theo-intro-1-proof}.
\begin{theo}\label{theo-intro-1}
There exist constants $a<\infty$ and $b>0$
such that for any $s\in \Ta$,  $u\in [s,t]_{\tau}$, $t\in [s,\infty[_{\tau}$ and and $\mu,\eta\in \Pa_{V/H_{s,t}}(E)$, we have the uniform contraction estimate
\begin{equation}\label{stab-time-varying-h}
\Vert \mu R^{(t)}_{s,u}-\eta R^{(t)}_{s,u}\Vert_{V/H_{u,t}}\leq a~ e^{-b (u-s)}~
\Vert \mu-\eta\Vert_{V/H_{s,t}}.
\end{equation}
\end{theo}

As we shall see in Lemma~\ref{lem-proof-P-H-t-hom} our regularity conditions ensure the existence of some $0< \epsilon<1$
and some constant $c>0$ such that for any time horizon $s\in \Ta$ and
$t\in [s,\infty[_{\tau}$  we have the Lyapunov estimate
\begin{equation}\label{ref-P-H-t-hom}
R^{(t)}_{s,s+\tau}(V/H_{s+\tau,t})\leq \epsilon~V/H_{s,t}+c.
\end{equation}
In this direction, we also emphasise that the main ingredient of the proof of Theorem \ref{theo-intro-1} is the $V$-contraction for Markov operators discussed in Lemma~\ref{lem-beta-12}.

\begin{cor}\label{theo-intro-1-cor-continuous}
For continuous time stable $V$-positive semigroups, there exist constants $a<\infty$ and $b>0$
such that for any $s\leq u\leq t$ and $\mu,\eta\in \Pa_{V/H_{s,t}}(E)$, we have the uniform contraction estimate
\begin{equation}\label{stab-time-varying-h-continuous}
\Vert \mu R^{(t)}_{s,u}-\eta R^{(t)}_{s,u}\Vert_{V/H_{u,t}}\leq a~({\pi_{\tau}(V)}/{\lambda^-(\overline{\eta}_0)})~ e^{-b (u-s)}~
\Vert \mu-\eta\Vert_{V/H_{s,t}}.
\end{equation}

\end{cor}

The estimate on the left-hand side~of (\ref{ref-H0-2}) allows one to control, uniformly with respect to the time parameter, the quantities $\mu(H_{s,t})$ as a function of $\mu(V)$, for any $\mu\in \Pa_V(E)$. Since these uniform estimates will be used several times in the sequel, we present them here in a general form. Applying the Markov inequality, for any $\mu\in \Pa_V(E)$ the left-hand side condition in (\ref{ref-H0-2}) ensures the existence of some $n\geq 1$  such that
\begin{equation}\label{ref-min-int-n}
r_n:=\mu(V)+n\Longrightarrow
\mu(H_{s,t})\geq {\varsigma_{r_n}(H)}\mu\left(V\leq r_n\right)\geq {\varsigma_{r_n}(H)}/{(1+\mu(V)/n)}>0.
\end{equation}
For any  $\mu\in \Pa_V(E)$, we conclude that
\begin{equation}\label{def-omega}
0<\omega_H(\mu):=\inf_{s\geq 0}\inf_{t\geq s}\mu(H_{s,t})\leq \mu(V).
\end{equation}
Similarly, we check that the condition $\kappa_V(\mu)<\infty$ ensures the tightness of sequence of measures $\Phi_{s,t}(\mu)$ indexed by $s\geq 0$ and $t\geq s$, \textcolor{black}{for any $\mu\in\Pa_V(E)$}. In the same vein, we check that the flow of measures $\overline\eta_t$ is tight. Thus, choosing
$$
r_n=\overline{\eta}(V)+n\geq r_H\quad\mbox{\rm with}\quad  \overline{\eta}(V):=\sup_{t\geq 0}\overline{\eta}_t(V)
$$
 we also check that
$$
\textcolor{black}{
\overline{\eta}_-(H):=
\inf_{t\geq 0}\overline{\eta}_t(H)}\geq 
\varsigma_{r_n}(H)/(1+\overline{\eta}(V)/n)>0.
$$

We are now in position to discuss some direct consequences of Theorem~\ref{theo-intro-1}.
Defining the finite rank (and hence compact) operator
$$
f\in \Ba_V(E)\mapsto T_{s,t}(f):=\frac{H_{s,t}}{\overline{\eta}_s(H_{s,t})}~\overline{\eta}_t(f)\in \Ba_{0,V}(E),
$$ 
the first corollary and its time homogeneous version discussed in Corollary~\ref{cor-projection-h} can be interpreted as an extended version of the Krein-Rutman theorem to time varying positive semigroups. 
\begin{cor}\label{cor-projection}
For any $s\in \Ta $ and $t\in [s,\infty[_{\tau}$ we have the exponential decay
$$
\vertiii{\overline{Q}_{s,t}-T_{s,t}}_V\leq ~ a~ e^{-b \textcolor{black}{(t-s)}}~\textcolor{black}{\left(1+\vertiii{H}_V~\overline{\eta}(V)/\overline{\eta}_-(H)\right)},
$$ 
where $(a,b)$ were defined in (\ref{stab-time-varying-h}).
In addition, we have the uniform norm estimate
\begin{equation}\label{norm-u-est-V}
\vertiii{\overline{Q}_{s,t}}_{V}\leq (1+a)~\left(1+\vertiii{H}_V~\overline{\eta}(V)/\overline{\eta}_-(H)\right).
\end{equation}
For continuous time semigroups, the above estimates remain valid for any continuous time indices $s\leq t$ with the parameter $a$ replaced by the parameter
\begin{equation}\label{def-a-eta}
a(\overline{\eta}_0):=a~({\pi_{\tau}(V)}/{\lambda^-(\overline{\eta}_0)}).
\end{equation}
\end{cor}
\begin{proof}
Using (\ref{def-over-Q-H-2}) and (\ref{link-s-t-Q-intro}) we check that
$$
\Psi_{H_{s,t}}\left(\overline{\eta}_s\right) R^{(t)}_{s,t}(f/H)={\Phi_{s,t}(\overline{\eta}_s)(f)}/{\Phi_{s,t}(\overline{\eta}_s)(H)}=\overline{\eta}_t(f)/\overline{\eta}_t(H).
$$
This yields the decomposition
$$
\overline{Q}_{s,t}(f)-\frac{H_{s,t}}{\overline{\eta}_s(H_{s,t})}~\overline{\eta}_t(f)=H_{s,t}~\left(R^{(t)}_{s,t}(f/H)-\Psi_{H_{s,t}}\left(\overline{\eta}_s\right) R^{(t)}_{s,t}(f/H)\right).
$$
Applying (\ref{stab-time-varying-h}) to $u=t$, $\mu=\delta_x$ and $\eta=\Psi_{H_{s,t}}\left(\overline{\eta}_s\right)$, for any $\Vert f\Vert_V=\Vert f/H\Vert_{V/H}\leq 1$ we check the estimate
\begin{equation}\label{ref-cor-f}
\left\vert \overline{Q}_{s,t}(f)-\frac{H_{s,t}}{\overline{\eta}_s(H_{s,t})}~\overline{\eta}_t(f)\right\vert\leq ~ a~ e^{-b (t-s)}~\left(V+\frac{H_{s,t}}{\overline{\eta}_s(H_{s,t})}~\overline{\eta}_s(V)\right)
\end{equation}
\textcolor{black}{where $(a,b)$ were defined in (\ref{stab-time-varying-h}).} This concludes the proof.
\end{proof}

When $H=1$, the extended version of the above corollary  in the context of random semigroups arising in filtering is provided in~\cite{whiteley}. The proof in~\cite{whiteley} relies on rather sophisticated coupling and  decomposition techniques given in~\cite{klepsyna-3}, which were further developed in~\cite{douc-moulines-ritov}. 

\begin{rmk}\label{rmk-1-H}
 In Theorem~\ref{theo-intro-1} and Corollary~\ref{cor-projection}, the tightness condition $\kappa_V(\mu)<\infty$ for any $\mu\in \Pa_V(E)$ in the right-hand side of (\ref{ref-V-over-eta}) can be replaced by the condition $ \overline{\eta}(V)<\infty$. In this situation, choosing $f=1$ in (\ref{ref-cor-f}) for any $\mu\in \Pa_V(E)$ we readily check that
\begin{equation}\label{ref-H-1-est}
\left\vert \mu\overline{Q}_{s,t}(1)-{\mu(H_{s,t})}/{\overline{\eta}_t(H)}\right\vert\leq ~ a~\mu(V)~ e^{-b (t-s)}~\left(1+\vertiii{H}_V~{\overline{\eta}(V)}/{\overline{\eta}_-(H)}\right),
\end{equation}
where $(a,b)$ were defined in (\ref{stab-time-varying-h}).
    \end{rmk}

Our next result transfers the stability of the $R$-semigroup to that of the normalised semigroup $\Phi$. 
 \begin{theo}\label{theo-intro-2}
For any  $s\in \Ta $ and $t\in [s,\infty[_{\tau}$, and any $\mu,\eta\in \Pa_V(E)$ we have the local contraction estimate 
\begin{equation}\label{stab-time-varying-Phi}
\vertiii{ \Phi_{s,t}(\mu)-\Phi_{s,t}(\eta)}_{V}
 \leq  a~\kappa(\eta,\mu)~ e^{-b (t-s)}~
~\vertiii{ \mu-\eta}_{V}
\end{equation}
with  $(a,b)$ as in (\ref{stab-time-varying-h}) and $\kappa(\eta,\mu)$ given by
\begin{eqnarray*}
\kappa(\eta,\mu)&:=&\kappa_H(\mu)
\left(1+ \kappa_V(\eta)\right)~ \left(1+{\eta(V)}/{\omega_H(\eta)}\right)/\omega_H(\mu).
\end{eqnarray*}
For continuous time semigroups, the above estimates remain valid for any continuous time indices $s\leq t$ with the parameter $a$ replaced by 
the paremeter $a(\overline{\eta}_0)$ defined in (\ref{def-a-eta}).
 \end{theo}
This result is a direct consequence of the $V$-contraction
 estimates (\ref{stab-time-varying-h}) stated in Theorem~\ref{theo-intro-1} and the rather 
 elementary Boltzmann-Gibbs estimates (\ref{Lip-Psi-h-2})  and
 (\ref{Lip-Psi-h-2-back}).  The full proof is provided in section~\ref{theo-intro-2-proof}.

As in Remark~\ref{rmk-1-H},  the condition $\kappa_V(\mu)<\infty$ for any $\mu\in \Pa_V(E)$  in Theorem~\ref{theo-intro-2}   can be  replaced by condition $$ \overline{\eta}(V)<\infty\quad\mbox{and}\quad \kappa_H(\mu)<\infty \quad\mbox{for any $\mu\in \Pa_V(E)$}. $$ 
The right-hand side condition in the above display is clearly met for any bounded function $H$. 
In this situation, following word-for-word  the proof of Theorem~\ref{theo-intro-2},  we check that
 $$
 \vertiii{\Phi_{s,t}(\mu)-\overline{\eta}_t}_{V}
 \leq  a~ \kappa(\overline{\eta},\mu)~ e^{-b (t-s)}~
~\vertiii{ \mu-\overline{\eta}_s}_{V}
 $$
 with
 $$
 \kappa(\overline{\eta},\mu):=\kappa_H(\mu)
\left(1+ \overline{\eta}(V)\right)\left(1+\overline{\eta}(V)/\overline{\eta}_-(H)\right)/\omega_H(\mu).
 $$
Using the above estimate we readily check that $\kappa_V(\mu)<\infty$ for any $\mu\in \Pa_V(E)$.

 The $V$-norm stability of the semigroup $\Phi_{s,t}$ is also discussed in~\cite{whiteley}
 (for instance \cite[Corollary 1]{whiteley}). The proof in~\cite{whiteley} is based on Corollary~\ref{cor-projection} and it does not provide local Lipschitz contraction estimates.

  \begin{rmk} The time varying Lyapunov function $V/H_{u,t}$ associated with the triangular array of Markov operators $R^{(t)}_{s,u}$ discussed in (\ref{ref-P-H-t-hom})  depends on the terminal time horizon $t$.  This property allows one to control the exponential decays (\ref{stab-time-varying-h}) of the corresponding $(V/H_{u,t})$-norms uniformly in $t$. These estimates are crucial in the proof of Theorem~\ref{theo-intro-2}. 
  
  We note that more conventional approaches, based on time homogeneous Lyapunov functions $V$, are discussed in~\cite{marguet}. This approach also ensures that the Markov semigroup  $R^{(t)}_{s,u}$ forgets its initial state with respect to a common time homogenous $V$-norm. However, it seems difficult to deduce  any local Lipschitz estimates of the form \eqref{stab-time-varying-Phi} from these uniform estimates. 
 \end{rmk}
 
Choosing $H=1$ in (\ref{link-s-t-Q-intro}) we have
\begin{equation}\label{Phi-P}
\left[\Phi_{s,t}(\mu)-\Phi_{s,t}(\eta)\right](f)= \frac{1}{\mu(H_{s,t})}~(\mu-\eta)\left(D_{\eta}\Phi_{s,t}(f)\right)
\end{equation}
with the first order linear operator $D_{\eta}\Phi_{s,t}$ defined by the  formula
$$
D_{\eta}\Phi_{s,t}(f)(x):=H_{s,t}(x)~\left(\Phi_{s,t}(\delta_x)-\Phi_{s,t}(\eta)\right)(f).
$$
Taylor expansions of higher order are also discussed in~\cite{arnaudon-dm}.
A weak version of the total variation estimate  (\ref{stab-time-varying-Phi}) is now easily obtained from the above perturbation formula.

\begin{cor}\label{cor-DPhi}
Consider the triangular array semigroup (\ref{link-s-t-Q-intro}) associated to the unit function $H=1$. In this case, for any  $s\in \Ta $, $t\in [s,\infty[_{\tau}$ and any $\eta\in \Pa_V(E)$ we have
\begin{equation}
\sup_{\Vert f\Vert_{V/H_{s,t}}\leq 1}\Vert D_{\eta}\Phi_{s,t}(f)\Vert_V\leq 
a~ e^{-b (t-s)}~\left(1
+\eta(V)/ \omega_H(\eta)\right).
\end{equation}
\end{cor}
\begin{proof}
We have
$$
D_{\eta}\Phi_{s,t}(f)(x):=H_{s,t}(x)\int~\eta(dy)~\frac{H_{s,t}(y)}{\eta(H_{s,t})}~(\delta_x R^{(t)}_{s,t}-\delta_y R^{(t)}_{s,t})(f).
$$
Using (\ref{stab-time-varying-h}) we check that
$$
\Vert \delta_x R^{(t)}_{s,t}-\delta_y R^{(t)}_{s,t}\Vert_{V/H_{s,t}}\leq a~ e^{-b (t-s)}~(
(V/H_{s,t})(x)+(V/H_{s,t})(y)).
$$
This implies that
$$
\sup_{\Vert f\Vert_{V/H_{s,t}}\leq 1}\Vert D_{\eta}\Phi_{s,t}(f)\Vert_V\leq 
a~ e^{-b (t-s)}~\left(1
+\eta(V)~{\Vert H_{s,t}\Vert_V}/{\eta(H_{s,t})}\right)
$$
and we can now conclude.
\end{proof}

\begin{rmk}
 These weak form estimates are particularly useful in the convergence analysis of the mean field particle models associated with sub-Markovian integral operators.
 Taylor expansions at any order are discussed in section 3.1.3 and chapter 10 in~\cite{dm-13}, see also section 2.3 in  the more recent article~\cite{arnaudon-dm}.
 \end{rmk}

\subsection{Time homogenous models}\label{sec:timehom}
\subsubsection{Leading eigen-triple}
 Consider a time homogenous version $Q_t$ of the stable $V$-positive semigroup discussed in section~\ref{sec:timeinhom}.  This section is concerned with the existence of an unique leading eigen-triple   $(\rho,\eta_{\infty},h)\in (\RR\times\Pa_V(E)\times \Ba_{0,V}(E))$ satisfying (\ref{def-eigen-triple-intro}).
We follow word-for-word the arguments developed in the end of section~\ref{sec:not_res-unif} in terms of $V$-normed spaces.
  In this context, for any $s\in \Ta$ we have
$$
\kappa_H(\Phi_s(\overline{\eta}_0))\leq \kappa_H(\overline{\eta}_0)<\infty\quad
\mbox{\rm and}\quad 0<\omega_H(\overline{\eta}_0)\leq \omega_H(\Phi_s(\overline{\eta}_0)).
$$
Thus, by (\ref{stab-time-varying-Phi}) we readily check that
$$
\vertiii{ \Phi_{t}(\overline{\eta}_0)-\Phi_{s+t}(\overline{\eta}_0)}_{V}
 \leq  a~\kappa(\overline{\eta}_0)~ e^{-b t}~
~\vertiii{ \overline{\eta}_0-\Phi_s(\overline{\eta}_0)}_{V}\leq 2~a~\kappa(\overline{\eta}_0)\kappa_V(\overline{\eta}_0)~ e^{-b t}~
$$
with $\kappa(\overline{\eta}_0)~:=\kappa(\overline{\eta}_0,\overline{\eta}_0)$. The above Lipschitz exponential decay estimate ensures that 
$\overline{\eta}_t$ is a Cauchy sequence in the complete set $\Pa_V(E)$ equipped with the $V$-distance. Thus it converges exponentially fast to a single probability $\eta_{\infty}=\Phi_t(\eta_{\infty})\in \Pa_V(E)$. The fixed point equation yields the exponential formula  (\ref{expo-rho-def}). Observe that
$$
\frac{Q_t(f)}{\eta_{\infty}Q_t(1)}=\frac{\overline{Q}_t(f)}{\eta_{\infty}(\overline{Q}_t(1))}
$$
with the semigroup $\overline{Q}_t$ defined in (\ref{def-over-Q-H}) in terms of the measure $\overline{\eta}_0\in \Pa_V(E)$ satisfying (\ref{ref-V-over-eta}) and (\ref{ref-H0-2}).
Combining  (\ref{norm-u-est-V}) with (\ref{ref-H-1-est}) we check that  the function $t\mapsto\Vert e^{-\rho t}Q_t(1)\Vert_V$ is uniformly bounded.
Now choosing $\overline{\eta}_0=\eta_{\infty}$ in (\ref{def-over-Q-H}) the normalized semigroup takes the following form $$\overline{Q}_t(1)={Q_t(1)}/{\eta_{\infty}Q_t(1)}=e^{-\rho t} Q_t(1)\quad \mbox{\rm and we have}\quad \sup_{t\geq 0}\Vert \overline{Q}_t(1)\Vert_V<\infty.
$$

 As in  (\ref{ref-product-series-h}), now applying Theorem~\ref{stab-time-varying-Phi} we have the pointwise convergence of product series expansion (\ref{prod-series}); that is, for any $x\in E$  we have the product series formula
$$
h(x):=\lim_{n\rightarrow\infty} \overline{Q}_{n\tau}(1)(x)=\prod_{n\geq 0}~\left\{1+
\left[\Phi_{n\tau}(\delta_x)(\overline{Q}_{\tau}(1))-\Phi_{n\tau}(\eta_{\infty})(\overline{Q}_{\tau}(1))\right]\right\}>0.
 $$
Applying the dominated convergence theorem as in (\ref{link-rmk-h-rho}), we also check that $\eta_{\infty}(h)=1$ as well as for any $s\in [0,\infty[_{\tau}$ the formulae
\begin{equation}\label{ref-dominated-cv}
Q_s(h)=e^{\rho s}~h\in \Ba_{0,V}(E).
\end{equation}
For continuous time models, the ground state does not depend on the time step so that the above formula is satisfied for any $s\geq 0$.
Choosing $\eta=\eta_{\infty}$ and $H$ such that $\eta_{\infty}(H)=1$ and $\Vert H\Vert_V\leq 1$ in (\ref{stab-time-varying-Phi}) we check that for any $V(x)\leq r$ and $s,t\in\Ta$ we have the exponential estimate
$$
\vert \Phi_{t}(\delta_x)(\overline{Q}_s(H))-1\vert\leq \Vert \overline{Q}_s(H)\Vert_V
\vertiii{ \Phi_{t}(\delta_x)-\eta_{\infty}}_{V}
 \leq  c(r)~ e^{-b t}~
$$
with some finite constant $c(r)<\infty$ and the parameter $b$ as in (\ref{stab-time-varying-h}). Since $\Vert \overline{Q}_t(H)\Vert_V$ is uniformly bounded with respect to the time parameter,
this implies that for any compact subset $K\subset E$ we have
\begin{equation}\label{ref-compact-cv}
\sup_{K}\vert \overline{Q}_{t+s}(H)-\overline{Q}_{s}(H)\vert \leq  c_K~ e^{-b t}~\quad\mbox{\rm with some finite constant $c_K<\infty$.}
\end{equation}
This shows that $\overline{Q}_{t}(H)$ is a uniformly Cauchy sequence on compact sets   and we have the pointwise convergence $\overline{Q}_{t}(H)(x)\longrightarrow_{t\rightarrow\infty}
h(x)$. 
Applying the dominated convergence theorem, for any $s>0$ we have that
$$
\overline{Q}_{s}(\overline{Q}_{t}(H))=\overline{Q}_{t+s}(H)\longrightarrow_{t\rightarrow\infty}
h=\overline{Q}_{s}(h)=e^{\rho s}~h\in \Ba_{0,V}(E).
$$
For strong Feller semigroups (in the sense that $Q_{s,t}(\Ba_V(E))\subset\Ca_{V}(E)$, for any $s<t $), $\overline{Q}_{t}(H)$  compactly converges as $t\rightarrow\infty$ to some continuous function $h=\overline{Q}_s(h)\in \Ca_{0,V}(E)$ as soon as the Polish space $E$ is locally compact. The same result applies when $H\in \Ca_V(E)$ and $Q_t(\Ca_V(E))\subset \Ca_V(E)$.
We recall that Polish spaces are separable metric spaces so they are second countable (and thus locally compact Polish spaces are $\sigma$-compact). In this context, $\Ca(E)$ equipped with the compact uniform topology is a complete metric space. This ensures that $h\in \Ca(E)$. We also check that $\Vert h\Vert_V<\infty$ by recalling that $\Vert Q_t(H)\Vert_V$ is uniformly bounded with respect to~the time horizon.

The next theorem connects the stability of the semigroup $Q_t$ with the one of the Doob's $h$-transform $P^h_t$ defined in (\ref{def-Ph-intro}). The proof is provided in section~\ref{theo-equivalence-proof}.
\begin{theo}\label{theo-equivalence}
The  semigroup  $Q_t$ is $V$-positive and stable if and only if there exists an eigen-triple  $(\rho,\eta_{\infty},h)$ satisfying (\ref{ref-dominated-cv})  and $P^h_t$ is stable $V^h$-positive semigroup, with the function $V^h=V/h$. 
\end{theo}

 \subsubsection{Sub-integral semigroups}
 In Section~\ref{V-sg-sec} we have seen that  absolutely continuous semigroups  satisfying condition $(\Aa)_V$ are stable $V$-positive semigroups. Our next objective is to relax this absolutely continuity condition. Consider the following sub-integral condition
\begin{equation}\label{ref-Q-min-prop}
Q_{\tau}(x_1,dx_2)\geq q_{\tau}(x_1,x_2)~\chi_{\tau}(dx_2),
\end{equation}
for some $\tau>0$, some density function $q_{\tau}(x_1,x_2)>0$ and some positive Radon measure $\chi_{\tau}$. Also assume that for any compact set $K$ there exists some positive measurable function 
$q^K_{\tau}(x_2)$ such that
$$
\inf_{x_1\in K}q_{\tau}(x_1,x_2)\geq q^K_{\tau}(x_2)>0\quad \mbox{\rm and}\quad \chi_{\tau}(q^K_{\tau})>0.
$$
For instance, the left-hand side~condition is satisfied for lower semi-continuous function $q_{\tau}(x_1,x_2)$ with respect to the first variable, and upper-semicontinuous with respect to the second. In this situation, for any compact set $K\subset E$ we have
$$
\forall x\in K\qquad
Q_{\tau}(h)(x)=e^{\rho {\tau}}~h(x)\geq ~\chi_{\tau}(hq^K_{\tau})>0\quad \mbox{\rm and thus}\quad \inf_Kh>0.
$$

Whenever  (\ref{ref-Q-min-prop}) is satisfied, for any compact set $K\subset E$  we have
\begin{equation}\label{ref-Q-min}
\forall x\in K\qquad
Q_{\tau}(x,dz)\geq \iota_K~ \nu_K(dz)
\end{equation}
 for some  Radon probability measure $\nu_K$  and some $\iota_K>0$ whose values may depends on the parameter $\tau$. This minorisation condition ensures that (\ref{ref-lem-t}) is satisfied.
 To check this claim, observe that for any $x\in K_r:=\{V\leq r\}$ we have
\begin{eqnarray*}
P^h_{\tau}(x,dz)&\geq& \frac{1}{e^{\rho \tau}\sup_{K_r}h}~Q_{\tau}(x,dz)~h(z)\\
&\geq& \alpha(r)~\nu_{K_r}^h(dz)~~\mbox{\rm with}~~
\alpha(r)=  \iota_{K_r}~e^{-\rho \tau}
\frac{\nu_{K_r}(h)}{\sup_{K_r}h}~~\mbox{\rm and}~~\nu_{K_r}^h:=\Psi_h(\nu_{K_r}).
\end{eqnarray*}
This clearly implies (\ref{ref-lem-t}).

 As we shall see in Lemma~\ref{ref-lem-Lyapunov-2}, the condition $Q_{\tau}(V)/V\in \Ba_0(E)$ implies the Lyapunov inequality (\ref{lyap-Ph-intro}).
This property also ensures that  for any $\eta\in \Pa_V(E)$ we have $ \kappa_V(\eta)<\infty$.
 For a more thorough discussion on the consequences of the Foster-Lyapunov inequality we refer the reader to Proposition~\ref{prop-unif-eta-h-V}.

This yields the following corollary of Theorem~\ref{theo-equivalence}.
\begin{cor}\label{H1-prop}
Consider a $V$-positive semigroup $Q_t$ satisfying the local minorisation condition (\ref{ref-Q-min-prop}).
In this situation, $Q_t$ is  stable  if and only if  there exists a leading  eigen-triple $(\rho,\eta_{\infty},h)\in(\RR\times \Pa_V(E)\times\Ba_{0,V}(E)).$
\end{cor}

Note that the minorisation condition (\ref{ref-Q-min-prop}) is less stringent than the absolutely continuous condition (\ref{ref-Q-chi}) imposed by condition $(\Aa)_V$. 
In the context of particle absorption models, it applies to jump processes including regular piecewise deterministic processes as well as  Metropolis-Hastings transitions.

 \subsubsection{Doob's $h$-tranform semigroup}
This section presents a more refined analysis of a  time homogenous $V$-positive semigroups $Q_t$ satisfying the following \textcolor{black}{weaker} condition:
\begin{equation}\label{rho-h}
 Q_{t}(h)=e^{\rho t}~h>0\quad
\mbox{\rm for some}~\tau>0,~\rho\in\RR,
~~\mbox{\rm and}~h\in \Ba_{0,V}(E).
\end{equation}
Recall that for any compact set $K\subset E$ we have
$$
V\geq 1\quad\mbox{\rm and}\quad
h/V\in \Ba_{0}(E)\Longrightarrow \inf_K h\geq \inf_K(h/V)>0.
$$
Arguing as above  condition $(\Ha^h)$ introduced in (\ref{ref-P-h-n-2-hom}) is satisfied as soon as the local minorisation condition (\ref{ref-Q-min}) is satisfied. The main drawback of condition $(\Ha^h)$ is that it requires some knowledge of the function $h$ which is often unknown.

Several illustrations and some sufficient conditions ensuring the existence of the leading eigen-pair $(\rho,h)$ satisfying (\ref{rho-h}) are discussed in section~\ref{quasi-compact-sec}, which is dedicated to the study of quasi-compact positive operators. See also Corollary~\ref{cor-intro-h-0-0}.
For instance, in remark~\ref{rmk-compact-Q-K} we shall see that the existence of a leading eigen-pair $(\rho,h)$ satisfying (\ref{rho-h}) is granted for absolutely continuous semigroups of the form (\ref{ref-Q-chi}) equipped with a continuous density. Note that in this case condition $(\Aa)_V$ is satisfied. 

For a more detailed discussion on the
 design of  functions $V$ satisfying condition $Q_{\tau}(V)/V\in \Ba_0(E)$, we refer to~\cite{examples-review-paper,douc-moulines-ritov,ferre,whiteley}.  
 The article~\cite{bansaye-2} also provides different conditions ensuring the existence of a leading eigen-pair $(\rho,h)$ for semigroups that are not necessarily absolutely continuous. We shall discuss these conditions in section~\ref{sec-comparisons}, dedicated to comparisons with the existing literature on this subject. We also refer the reader to~\cite{bansaye-2} for some additional illustrations of these conditions in the context of the growth-fragmentation equations.

We are now in position to state the main result of this section, the proof of which can be found in section~\ref{theo-homo-sec}.

 \begin{theo}\label{stab-h-transform-time}
 Assume that  $(\Ha^h)$ is satisfied.
Then the Markov semigroup $P^h_{t}$ has a single invariant measure  $\eta^h_{\infty}\in \Pa_{V/h}(E)$. In addition,
there also exists  some finite constant $a_h<\infty$ and some parameter $b_h>0$, such that 
for any $\mu,\eta\in \Pa_{V/h}(E)$ and \textcolor{black}{$t\in \Ta$} we have the contraction estimate
 \begin{equation}\label{first-h-estimate}
 \Vert \mu P_{t}^h-\eta P_{t}^h\Vert_{V/h}\leq a_h~e^{-b_h t}~ \Vert \mu- \eta\Vert_{V/h}.
 \end{equation}
\textcolor{black}{ In addition, for any $\eta\in\Pa_V(E)$ we have the estimates 
\begin{equation}\label{ref-to-prop-unif-eta-h-V}
0<\kappa_h^-(\eta)\leq   \kappa_V(\eta)<\infty.
\end{equation}}
\end{theo}
As with Theorem~\ref{theo-intro-1}, the main ingredient of the proof is the $V$-contraction for Markov operators discussed in Lemma~\ref{lem-beta-12}.
Also note that for continuous time semigroups, we have
$$
\pi^h_{\tau}(V^h):=\sup_{s\geq 0}\sup_{\delta\in [0,\tau[}\Vert P^h_{\delta}(V^h)/V^h\Vert\leq e^{|\rho| \delta}\pi_{\tau}(V)<\infty.
$$

\begin{rmk}\label{rmk-h-ref-bertrand}
Theorem~\ref{stab-h-transform-time} ensures the uniqueness of the invariant measure $\eta^h_{\infty}=\eta^h_{\infty} P^h_t\in \Pa_{V/h}(E)$ and exponential decay to equilibrium of the $h$-process.

Choosing $\overline{\eta}_0=\eta_{\infty}:=\Psi_{1/h}(\eta^h_{\infty})$ we have $\overline{\eta}_0=\Phi_{\tau}(\overline{\eta}_0)$.
In this scenario, we readily check  (\ref{ref-H0-2})  choosing $H=h$ and using the fact that
 \begin{equation}\label{appli-h-Hst}
H_{s,t}=h\quad \mbox{and}\quad\overline{\eta}_0\,Q_{t,t+\tau}(1)=e^{\rho \tau}.
\end{equation}

\end{rmk}

We illustrate the impact of Theorem~\ref{stab-h-transform-time} with some direct corollaries.

\begin{cor}\label{cor-ratio-P-h}
Under condition $(\Ha^h)$, for any $\eta\in \Pa_{V}(E)$ and $s,t\geq 0$ we have
 \begin{equation}\label{ratio-P-h}
 \eta_t^h:=\Psi_h(\eta)P_t^h\Longrightarrow
\eta_s^h\left\vert\frac{P^{h}_{t}(1/h)}{\eta^h_s P_{t}^h(1/h)}-1\right\vert\leq c(\eta)~a_h ~e^{-b_ht},
 \end{equation}
where $(a_h,b_h)$ were introduced in (\ref{first-h-estimate}) and 
 $$
 c(\eta):=2~\kappa_V(\eta)\kappa_h(\eta)/\kappa_h^-(\eta),
  $$
  with $\kappa_V(\eta)$ and $\kappa_V^-(\eta)$ defined in~\eqref{kappa-V-h-intro}.
\end{cor}
\begin{proof}
Recall that $\eta_t:=\Phi_t(\eta)=\Psi_{1/h}(\eta^h_t)$ so that $$\eta^h_{t}(1/h)~\eta_t(h)=1
\quad\mbox{\rm and}\quad \eta^h_t(V/h)=\eta_t(V)/\eta_t(h).
$$ The estimate (\ref{ratio-P-h}) is now easily checked applying (\ref{first-h-estimate}) to $(\mu,\eta)=(\delta_x,\delta_y)$ and using the inequality
$$
\eta^h_s\left\vert\frac{P^{h}_{t}(1/h)}{\eta^h_s P_{t}^h(1/h)}-1\right\vert\leq \kappa_h(\eta)~ \int~\eta^h_s(dx)\eta^h_s(dy)~\left\vert P^{h}_{t}(1/h)(x)-P^{h}_{t}(1/h)(y)\right\vert.
$$
This ends the proof of the Corollary.
\end{proof}
\begin{cor}\label{stab-h-Phi-transform-time}
Under condition $(\Ha^h)$, the measure $\eta_{\infty}:=\Psi_{1/h}(\eta^h_{\infty})\in  \Pa_V(E)$ is the unique invariant measure of the semigroup $\Phi_{t}$.

In addition, for any $\mu,\eta\in \Pa_{V}(E)$  and \textcolor{black}{$t\in \Ta$} we have $\kappa_V(\mu)<\infty$ and
 \begin{equation}\label{first-h-estimate-2}
\vertiii{\Phi_{t}(\mu)-\Phi_{t}(\eta)}_{V}\leq ~a_h~\kappa(\mu,\eta)~e^{-b_h t}~\vertiii{\mu-\eta}_{V},
 \end{equation}
with $(a_h,b_h)$ as in (\ref{first-h-estimate}) and the  parameters $\kappa(\mu,\eta)$ defined by 
\begin{eqnarray*}
\kappa(\mu,\eta)&:=&\kappa_h(\mu)
\left(1+\kappa_V(\eta)\right)\left(1+{\eta(V)}/{\eta(h)}\right)/\mu(h).
\end{eqnarray*}
\end{cor}
 \begin{proof}
Combining (\ref{link-h-phi}) with  the Boltzman-Gibbs estimate (\ref{Lip-Psi-h-2})  we readily check the estimate
$$
\vertiii{ \Phi_{t}(\mu)-\Phi_{t}(\eta)}_{V}\leq ~ \kappa_h(\mu)
\left(1+\kappa_V(\eta)\right)~\vertiii{ (\Psi_{h}(\mu)-\Psi_{h}(\eta))P^h_t}_{V/h}.
$$
The contraction estimate (\ref{first-h-estimate}) now implies that
$$
\vertiii{(\Psi_{h}(\mu)-\Psi_{h}(\eta))P^h_t}_{V/h}\leq a_h~\kappa_h(\mu)
\left(1+\kappa_V(\eta)\right)~e^{ -b_{h}t}~  \vertiii{\Psi_{h}(\mu)-\Psi_{h}(\eta)}_{{V}/{h}}.
$$
Using  the Boltzman-Gibbs estimate (\ref{Lip-Psi-h-2-back}) we also have
$$
\vertiii{\Psi_{h}(\mu)-\Psi_{h}(\eta)}_{{V}/{h}}\leq
 \frac{1}{\mu(h)}\left(1+\frac{\eta(V)}{\eta(h)}\right)~\vertiii{ \mu-\eta}_{V}.
$$
This ends the proof of the corollary.
 \end{proof}

Corollary~\ref{stab-h-Phi-transform-time} is closely related to the exponential decay estimates stated in Theorem 1 of~\cite{ferre} under different regularity conditions. In contrast with~\cite{ferre}, our approach is based on the $V$-contraction properties of the $h$-semigroups and it allows one to derive local contraction estimates.

Defining the finite rank (and hence compact) operator
$$
f\in \Ba_V(E)\mapsto T(f):=\frac{h}{\eta_{\infty}(h)}~\eta_{\infty}(f)\in \Ba_{0,V}(E),
$$
the next corollary follows word-for-word the same arguments as the proof of Corollary~\ref{cor-projection}, thus it is skipped.

\begin{cor}\label{cor-projection-h}
Under condition $(\Ha^h)$, for any  \textcolor{black}{$t\in \Ta$} we have the operator norm exponential decay
\begin{equation}\label{k-r-frob}
\vertiii{ \overline{Q}_{t}-T}_{V}\leq 
a_h ~e^{-b_h t}~\left(1+ \eta_{\infty}(V)/\eta_{\infty}(h)\right)
\end{equation}
with $\overline{Q}_{t}:=e^{-t\rho}~Q_{t}$ and the same parameters $(a_h,b_h)$ as in (\ref{first-h-estimate}).
\end{cor}

\section{Some illustrations}\label{sec-illustrations}
\subsection{Nonlinear conditional processes}

Whenever $Q_{s,t}$ is sub-Markovian, we have the nonlinear transport equation
$$
\Phi_{s,t}(\mu)=\mu M^{\mu}_{s,t}
$$
with the collection of Markov transition $M^{\mu}_{s,t}$ indexed by $\mu\in \Pa(E)$ given by the formula
$$
M^{\mu}_{s,t}(f)(x)=Q_{s,t}(1)(x)~\frac{Q_{s,t}(f)(x)}{Q_{s,t}(1)(x)}+(1-Q_{s,t}(1)(x))~\Phi_{s,t}(\mu)(f).
$$
We also recall that for any $s\leq u\leq t$ we have the nonlinear semigroup equation
$$
M^{\mu}_{s,t}=M^{\mu}_{s,u}M^{\Phi_{s,u}(\mu)}_{u,t}.
$$
This shows that the normalised semigroup $\Phi_{s,t}$ is the semigroup of a nonlinear Markov process sometimes called   process conditioned to non-absorption at every time step.
For time homogeneous models, unless $\mu$ coincides with the quasi-invariant measure $\Phi_{t}(\eta_{\infty})=\eta_{\infty}$, the process is a nonlinear interacting jump process. In this interpretation, the distribution on path-space of the nonlinear process coincides with the McKean-measure associated with a jump process whose jumps intensity depends on the distribution of the random states.

For a more thorough discussion on the nonlinear interacting jump processes associated with these nonlinear Markov semigroups we refer to
 section 12.3 in~\cite{dm-13} and the articles~\cite{arnaudon-dm,dm-2000,dm-2000-moran}. Next proposition is a direct consequence of Theorem~\ref{theo-intro-2}.
 \begin{prop}
Under the assumptions of Theorem~\ref{theo-intro-1}, for any $s\leq t$ and any $\mu,\eta\in \Pa_V(E)$ we have the local Lipschitz operator norm estimate 
$$
\vertiii{\delta_xM^{\mu}_{s,t}-\delta_xM^{\eta}_{s,t}}_V\leq   a~\kappa(\eta,\mu)~ e^{-b (t-s)}~
~\vertiii{ \mu-\eta}_{V}
$$
with  $(a,b,\kappa(\eta,\mu))$ as in (\ref{stab-time-varying-Phi}).
\end{prop}

\subsection{Sub-Markovian semigroups}

Sub-Markovian operators are naturally associated with killed or absorbed stochastic processes.
Consider a stochastic flow $t\in [s,\infty[\mapsto X^c_{s,t}(x)$  starting at $X^c_{s,s}(x)=x\in E$ when $t=s$ and absorbed in a cemetery state $c$ at some random time $T^c_s(x)$. 
For instance, suppose we are given an auxiliary stochastic flow $X_{s,t}(x)$ evolving on $E$, which is sent to the cemetery at some uniformly bounded rate $U_t(y)\geq 0$ when at $y \in E$. In this situation, we have the so-called Feynman-Kac propagator formulae 
\begin{eqnarray}
Q_{s,t}(f)(x) & = & \EE\left(f(X^c_{s,t}(x))~1_{T^c_s(x)>t}\right)
\nonumber \\
& = &\EE\left(f(X_{s,t}(x))~\exp{\left(-\int_s^tU_u(X_{s,u}(x))~du\right)}\right).\label{ref-FK}
\end{eqnarray}
In this context, it is readily  checked that the normalised and unnormalised semigroups are connected by the formula
\begin{equation}\label{FK-expo}
\mu Q_{s,t}(f)=\Phi_{s,t}(\mu)(f)~\exp{\left(\textcolor{black}{-}\int_s^t~\Phi_{s,u}(\mu)(U_u)~du\right)}.
\end{equation}
\textcolor{black}{In terms of the absorption time, the above formula reads
$$
\int\mu(dx)~\PP(X^c_{s,t}(x)\in dy,~T^c_s(x)>t)=\Phi_{s,t}(\mu)(dy)~\exp{\left(-\int_s^t~\Phi_{s,u}(\mu)(U_u)~du\right).}
$$
This shows that the killing time of the process starting from $\mu$ at time $s$ is a Poisson process with a time varying rate function $\Phi_{s,t}(\mu)(U_t)$ that depends on $\mu$. The discrete time version of the above formula coincides with the product formula (\ref{product}). For a more thorough discussion on this subject we refer to  \cite[section 1.3.2]{dm-2000},~\cite[proposition 2.3.1]{dm-04} or \cite[section 12.2.1]{dm-13}. As noted in~\cite{collet-martinez-1,meleard-villemonais}, in the context of time homogeneous models, we readily check that the killing time is exponentially distributed
as soon as  $\mu=\eta_{\infty}=\Phi_{t}(\eta_{\infty})$. }
 
 Applying the above to $f=1$ for any $\mu_1,\mu_2\in \Pa(E)$ and $s\leq t$ we check that
\begin{equation}\label{int-exp}
\mu_1Q_{s,t}(1)/\mu_2Q_{s,t}(1)=\exp{\left(\int_s^t~\left(\Phi_{s,u}(\mu_2)(U_u)-\Phi_{s,u}(\mu_1)(U_u)\right)~du\right)}.
\end{equation}
The discrete time version of the above formula coincides with (\ref{product}).
Under the assumptions of Theorem~\ref{theo-1} 
the norm of the first order operator introduced in (\ref{Phi-P}) decays exponentially. That is for any $f$ such that $\mbox{\rm osc}(f):=\sup_{(x,y)\in E^2}{\vert f(x)-f(y)\vert}= 1$, we have 
\begin{equation}\label{hyp-q}
\Vert D_{\mu_2}\Phi_{s,t}(f)\Vert\leq q~(1-\epsilon_{\tau})^{(t-s)/\tau}~\quad \mbox{\rm with}\quad
q:=\sup\frac{\mu_1Q_{s,t}(1)}{\mu_2Q_{s,t}(1)}<\infty.
\end{equation}
where the supremum is taken over all  
indices $(s,t)$ such that $s\in\Ta$, $t\in [s,\infty[_{\tau}$ and $\mu_1,\mu_2\in \Pa(E)$.

In the context of time homogeneous models $X_{s,s+t}(x)=X_t(x):=X_{0,t}(x)$ and $U_t=U$, using formula (\ref{FK-expo}) we readily check that the ground state $h$ discussed in (\ref{link-rmk-h-rho}) takes the following form
$$
h(x)=\lim_{t\rightarrow\infty}\frac{\delta_x Q_t(1)}{\eta_{\infty}Q_t(1)}=\exp{\left(\int_0^{\infty} (\Phi_s(\eta_{\infty})(U)-\Phi_s(\delta_x)(U))~ds\right)}.
$$

\subsection{Path space Feynman-Kac measures}

Let   $\Omega=D([0,\infty[,E)$ be the space of c\`adl\`ag paths $\omega:s\in \RR_+:=[0,\infty[\mapsto \omega_s\in E$. 
Consider a canonical Markov process  $(\Omega,(X_s)_{s\geq 0},$ $(\Fa_s)_{s\geq 0},\PP_{\mu})$ with generator $L$ and initial distribution $\mu\in\Pa_V(E)$. In this notation, the Feynman-Kac measure 
on path-space associated with the time homogeneous version of (\ref{ref-FK}) is defined for any $t\geq 0$ and $\omega\in \Omega$ by the formula
\begin{equation}\label{path-space-fK}
\QQ_{\mu,t}(d\omega):=\frac{1}{\ZZ_{\mu,t}}~
\exp{\left(-\int_0^t U(\omega_s)~ds\right)}~\PP_{\mu}(d\omega),
\end{equation} 
where $\ZZ_{\mu,t}$ is a normalising constant. 
In continuous time, the leading eigen-pair $(\rho,h)$, discussed in (\ref{rho-h}), is obtained by solving the equation
$$
L(h)-Uh=\rho h\Longleftrightarrow -U=\rho-L(h)/h.
$$
Let $\PP^h_{\Psi_h(\mu)}$ be the distribution of the $h$-process 
starting with the initial probability measure $\Psi_h(\mu)\in\Pa_{V/h}(E)$. In this notation,
an exponential change of probability measure yields, for any $t\geq s\geq 0$ and any bounded $\Fa_s$-measurable function $F_s$,
$$
\QQ_{\mu,t}(F_s)-\PP^h_{\Psi_h(\mu)}(F_s)=\int_{\Omega}~\PP^h_{\Psi_h(\mu)}\left(d\omega\right)~F_s(\omega)~\left(\frac{P^{h}_{t-s}(1/h)(\omega_s)}{\eta^h_s P_{t-s}^h(1/h)}-1\right).
$$ 
Using (\ref{ratio-P-h}) we check the following estimate.
\begin{cor}
Under the assumptions of Corollary~\ref{cor-ratio-P-h} for any $s\leq t$ and $\mu\in\Pa_V(E)$ we have
$$
\vert\QQ_{\mu,t}(F_s)-\PP^h_{\Psi_h(\mu)}(F_s)\vert\leq c(\mu)~a_h ~e^{-b_h(t-s)}~\Vert F_s\Vert,
$$
where the parameters $(a_h, b_h,c(\mu))$ were defined in Corollary~\ref{cor-ratio-P-h}. 
\end{cor}
The  above exponential estimate improves the asymptotic result presented in Proposition 6.1 in~\cite{liming-W-2}, and simplifies the analysis in~\cite{champagnat,champagnat-3}. We also mention that the Feynman-Kac measures on path space $\QQ_{\mu,t}$ and thus the path-distribution of the $h$-process can be approximated using genealogical tree based Monte Carlo methods, see for instance~\cite{arnaudon-dm,dm-04,dm-13,dm-gen-01} and references therein. In Quantum physics and more particularly in statistical mechanics, the measure  $\QQ_{\mu,t}$ is sometimes called the grand-ensemble associated with the interaction potential $U$~\cite{glim,liming-W-2}. In particle absorption literature, 
the distribution 
$\PP^h_{\Psi_h(\mu)}$ of the $h$-process is sometimes called the distribution of the $Q$-process, that is the process conditioned to never be extinct~\cite{champagnat,champagnat-3}.

We now state some results pertaining to the limiting behaviour of the occupation measure of the time homogeneous stochastic flow. To simplify the presentation, we  shall only work under the strong regularity conditions stated in Theorem~\ref{theo-1}.
In this context, these results are direct consequences of our semigroup analysis that build on the analysis developed in  ~\cite{champagnat-3,chen,he}. 

Observe that for any $s\leq u\leq t$ we have
$$
\EE\left(f(X^c_{s,u}(x))~|~T^c_s(x)>t\right)=\Psi_{Q_{u,t}(1)}\left(\Phi_{s,u}(\delta_x)\right)(f).
$$
Setting $X^c_{u}(x):=X^c_{0,u}(x)$ and $T^c(x):=T^c_0(x)$, due to \eqref{add-expo-intro-2} we obtain the following corollary.
 
\begin{cor}\label{cor-expo-apply}
Under the assumptions of Corollary~\ref{est-mui-hom-cor}, we have
\begin{equation}\label{add-expo-apply}
\begin{array}{l}
\displaystyle\sup_{\Vert f\Vert\leq 1}\sup_{x\in E}\left\vert\EE\left(\frac{1}{t}\int_0^t~f(X^c_{u}(x))~du~|~T^c(x)>t\right)-\Psi_h(\eta_{\infty})(f)\right\vert\\
\\
\displaystyle \qquad\qquad\leq  \frac{\tau}{t}~\frac{1}{\vert\log{(1-\epsilon_{\tau})}\vert}
 \left( q(\eta_{\infty})+2\left(\Vert h\Vert/\eta_{\infty}(h)+~q(\eta_{\infty})^2
\right)\right).
\end{array}
\end{equation}
\end{cor}

In the above display, $h$ stands for the eigenfunction defined in Corollary~\ref{cor-intro-h-0}. The measure $\Psi_h(\eta_{\infty})$ which coincides with the invariant measure of the $h$-process is sometimes called the quasi-ergodic measure of the non-absorbed process.
Similarly, we also obtain a uniform bound on the $L^2$ distance.
\begin{prop}\label{prop-L2}
Under the assumptions of Corollary~\ref{est-mui-hom-cor}, for any $f\in \Ba(E)$ with $\mbox{\rm osc}( f)\leq 1$ and any $t\in \Ta$ we have the uniform estimate
$$
\begin{array}{l}
\displaystyle\left\vert\EE\left(\left(
\frac{1}{t}\int_0^t~f(X^c_{u}(x))~du-\Psi_h(\eta_{\infty})(f)\right)^2~|~T^c(x)>t\right)\right\vert\\
\\
\displaystyle\leq \frac{8\tau}{t}~\frac{\left(\Vert h\Vert/\eta_{\infty}(h)+(1\vee q(\eta_{\infty}))^2\right)}{\vert \log{(1-\epsilon_{\tau})}\vert}.
\end{array}
$$
\end{prop}
 \begin{proof}
First observe that for any $s\leq u\leq v\leq t$ we have
$$
\begin{array}{l}
\displaystyle
\EE\left(f(X^c_{s,u}(x))~f(X^c_{s,v}(x))~|~T^c_s(x)>t\right)\\
\\
\displaystyle=\frac{\delta_xQ_{s,u}\left(f~ Q_{u,t}(1)~{Q_{u,v}(f~Q_{v,t}(1))}/{Q_{u,v}(Q_{v,t}(1))}\right)}{\delta_xQ_{s,u}(Q_{u,t}(1))}\\
\\
\displaystyle=\int~
\Psi_{Q_{u,t}(1)}\left(\Phi_{s,u}(\delta_x)\right)(dy)~f(y)~\int
\Psi_{Q_{v,t}(1)}\left(\Phi_{u,v}(\delta_y)\right)(dz)~f(z).
  \end{array}
$$
Replacing $f$ by $(f-\Psi_h(\eta_{\infty})(f))$ there is no loss of generality to assume that
$\Psi_h(\eta_{\infty})(f)=0$. In this situation, due to~\eqref{add-expo-intro-2}, we have
$$
\begin{array}{l}
\displaystyle\left\vert\EE\left(\left(
\frac{1}{t}\int_0^t~f(X^c_{u}(x))~du\right)^2~|~T^c(x)>t\right)\right\vert~\\
\\
\displaystyle\leq \frac{4}{t^2}~\left(\Vert h\Vert/\eta_{\infty}(h)+(1\vee q(\eta_{\infty}))^2\right)
\int_{0\leq u\leq v\leq t}
\left(e^{-a(v-u)}+e^{-a(t-v)}\right) dv~du
 \end{array}
$$
with
$$
 a:={\vert \log{(1-\epsilon_{\tau})}\vert}/{\tau}.
$$
This implies that
$$
\left\vert\EE\left(\left(
\frac{1}{t}\int_0^t~f(X^c_{u}(x))~du\right)^2~|~T^c(x)>t\right)\right\vert\leq
\frac{8}{at}~\left(\Vert h\Vert/\eta_{\infty}(h)+(1\vee q(\eta_{\infty}))^2\right). 
$$ 
\end{proof}

We end our discussion of sub-Markov operators with bounded potentials by noting that we may also extend Proposition~\ref{prop-L2} to also obtain bounds in $\LL^p$ of the order $t^{-1/2}$. The idea behind the proof is to write the averages in terms of the h-process (in particular, the trajectorial version) and then apply $\LL^p$ bounds for occupation measures of a Markov process, using the Poisson equation associated with the $h$-semigroup (see for instance Lemma 8.4.11 in~\cite{dmpenev} in discrete time settings).

\subsection{Comparisons of our conditions with the literature}\label{sec-comparisons}
In this section, we highlight some of the comparisons between the models and the regularity conditions discussed in the present article and conditions often used in the literature for positive integral operators. 

We begin by remarking that the class of positive semigroups discussed in this article encapsulates discrete generation Feynman-Kac semigroups defined  for any $t\in[0,\infty[_{\tau}$ by the formula
$$
Q_{t,t+\tau}(x,dy)=G_{t,t+\tau}(x)~P_{t,t+\tau}(x,dy)
$$
with the potential function $G_t$ and the Markov transition $P_{t,t+\tau}$ defined by
\begin{equation}\label{G-M}
G_{t,t+\tau}:=Q_{t,t+\tau}(1)\quad \mbox{\rm and}\quad P_{t,t+{\tau}}(f):={Q_{t,t+\tau}(f)}/{Q_{t,t+\tau}(1)}.
\end{equation}
This class of probabilistic models arises in a variety of disciplines including statistical physics, biology, signal processing, rare event analysis, and many others; see \cite{dm-04,dm-13,dm-2000,dmpenev} and the relevant references therein.

 In this context, 
the uniform minorization condition (\ref{beta-P-nu-2}) is a well known strong condition ensuring the stability of the semigroups $\Phi_{s,t}$ and the existence of fixed point invariant measures for time homogeneous semigroups; see for instance~\cite[Theorem 2.3]{dm-2000}, 
\cite[Lemma 2.1 \& Lemma 2.3]{dg-ihp}, \cite[Lemma 2.1]{dm-sch-2}, \cite[section 12.2]{dm-13},  as well as \cite[section~2.1.2 \& section 3.1.3]{dm-2000} and \cite[section 4]{dm-04}.

The next condition is taken from~\cite{champagnat}. In terms of the Markov transition $P_{s,s+\tau}$ discussed in (\ref{G-M}) it takes the following form: 

{\it $(\Pa):$  For any $s\in \Ta$ and $(x_1,x_2)\in E^2$ there exists some $\nu\in \Pa(E)$ such that for $i=1,2$  and any $t\geq s+\tau$ we have the estimates
\begin{equation}\label{P-R-Hyp}
\delta_{x_i}P_{s,s+\tau}\geq \epsilon_1~\nu\quad \mbox{and}\quad \epsilon_2
~\Vert Q_{s+\tau,t}(1)\Vert \leq \nu(Q_{s+\tau,t}(1))
\end{equation}
for some parameters $0<  \epsilon_i\leq 1$ whose values do not depend on $x_i$, nor on $s\in\Ta$.
}

We also refer the reader to conditions (A1) and (A2) in \cite{champagnat-2}, as well as~\cite{champagnat-coulibaly,champagnat-3, champagnat-7} for further work and discussion on this condition.
The time-homogeneous version of the above condition is appears in a variety of contexts in the literature. For example, as shown in~\cite[Theorem 2.1]{champagnat}, condition $(\Pa)$ is a sufficient and necessary condition for the {\em uniform} exponential decay (\ref{beta-sup}), which was also applied in \cite[Theorem 10]{phd-emma} and \cite[Theorem 7.1]{horton-kyprianou-villemonais} to neutron transport models. In addition, we refer the reader to \cite[section 2.2]{bansaye} for the use of condition $(\Pa)$ in the design of admissible coupling constants (a.k.a.~generalised Doeblin's conditions) and to~\cite{chazotte} on birth-and-death processes where it is the main ingredient of the proof of Theorem 3.1, and to~\cite[section 2]{he}. 
\textcolor{black}{Moreover, for time-homogeneous models satisfying (\ref{beta-P-nu-2}), the right-hand side estimate in (\ref{P-R-Hyp}) is a direct consequence of  the right-hand side estimate in (\ref{con-unif})}.

When $(\Pa)$ is met, for any $f\geq 0$ and $i=1,2$ choosing 
$H=1$ we have the lower bound estimate
\begin{eqnarray*}
R^{(t)}_{s,s+\tau}(f)(x_i)&\geq &\epsilon_1~ \frac{\nu\left(Q_{s+\tau,t}(1)~f\right)}{P_{s,s+\tau}\left(Q_{s+\tau,t}(1)\right)(x_i)}\geq\epsilon_1\epsilon_2~\Psi_{Q_{s+\tau,t}(1)}(\nu_{s})(f).\end{eqnarray*}
\textcolor{black}{This implies that condition $(\Pa)$ is stronger than the Dobrushin's condition discussed in (\ref{beta-P-nu-2}). 
More precisely, we have}
$$
(\Pa)\Longrightarrow (\ref{beta-P-nu-2})\quad \mbox{\rm with}\quad H=1\qquad
\epsilon_{\tau}:=\epsilon_1\epsilon_2\quad \mbox{\rm and}\quad
\nu=\Psi_{Q_{s+\tau,t}(1)}(\nu_{s}).
$$
As previously mentioned, the uniform minorisation condition (\ref{beta-P-nu-2}) as well as $(\Pa)$ are difficult to check in practice; several sufficient conditions are discussed in~\cite{dm-04,dm-13,dm-2000,dg-ihp,dm-sch,dm-sch-2}. 

The next condition is a slight extension of \cite[condition (68)]{whiteley-4} and \cite[condition (3.24)]{baudel-guyader-lelievre}. 

{\it$(\Qa):$ There exists a positive measurable function $\varsigma_{s}(x) >0$, a constant $\rho>0$ and probability measures $\nu_{s}$ such that 
\begin{equation}\label{Q-ref}
\varsigma_{s}(x)~\nu_{s}(dy)\leq \delta_xQ_{s,s+\tau}(dy)\leq \rho ~\varsigma_{s}(x)~\nu_{s}(dy).
\end{equation}}

Choosing $H=1$ in (\ref{def-R-intro}), this condition implies that  for any $f\geq 0$ we have
$$
R^{(t)}_{s,u}(f):=\frac{Q_{s,u}( Q_{u,t}(1)~f)}{Q_{s,u}(Q_{u,t}(1))}\geq \rho^{-1}~\nu_{s,u}^{(t)}(f)
\quad \mbox{\rm with}\quad 
 \nu_{s}^{(t)}(f):=\frac{\nu_{s}(Q_{s+\tau,t}(1) f)}{\nu_{s}(Q_{s+\tau,t}(1) )}.
$$
This shows that
$$
(\Qa)\Longrightarrow (\ref{beta-P-nu-2})\quad \mbox{\rm with}\quad H=1\qquad
\epsilon_{\tau}:=\rho^{-1}\quad \mbox{\rm and}\quad
\nu= \nu_{s}^{(t)}.
$$
Condition $(\Qa)$ with \textcolor{black}{$\varsigma_{s}(x)=1$} is also discussed in \cite[section 4.3.2]{dm-04},  as well as in \cite{legland-oudjane} and \cite[section 3.3]{ferre}. We also refer the reader to~\cite{champagnat-coulibaly}, where it was shown that condition $(\Qa)$ implies the uniform exponential decay~\eqref{beta-sup} in the time-homogeneous case.

\textcolor{black}{Also notice that
$$
(\Qa)\Longleftrightarrow \rho^{-1}~\nu_{s}\leq \delta_xP_{s,s+\tau}\leq \rho ~\nu_{s},
$$
where the Markov transition $P_{s,s+\tau}$ was defined in (\ref{G-M}).} These rather strong two-sided minorisation conditions are well-known: \textcolor{black}{see for instance the uniformly positive condition discussed in~\cite{birkoff}, \cite[condition (19)]{atar} in the framework of Hilbert projective metrics}, 
\cite[condition $(B)$]{dg-ihp} and \cite[Theorem 2.3]{dm-2000}. \textcolor{black}{From the above discussion it should be clear that
conditions $(\Pa)$ and $(\Qa)$ are stronger than the Dobrushin condition presented in (\ref{beta-P-nu-2}).}

The class of triangular array semigroups introduced in (\ref{def-R-intro}) are also considered in the article~\cite{bansaye-2}  in the context of time homogeneous sub-Markovian models. In our framework, the authors assume the existence of positive functions $H\leq V$, a probability measure $\nu$ defined
on some compact $K\subset E$, and some finite constant $c$, such that 
\begin{equation}\label{ref-bertrand}
\sup_K V/H<\infty \qquad
Q_{\tau}(V)\leq a~V+c~1_K H \quad \mbox{\rm with}\quad 0<a< \inf_{E} (Q_{\tau}(H)/H).
\end{equation}
In addition, there exists some $\epsilon\in ]0,1]$ such that for any positive function $f\in \Ba_{V/H}(E)$ and any $x\in K$ we have
\begin{equation}\label{ref-bertrand-2}
{Q_{\tau}(fH)(x)}/{Q_{\tau}(H)(x)}\geq \epsilon~ \nu\quad \mbox{\rm and}\quad
\sup_{t\in [0,\infty[_{\tau}}\sup_{K}\frac{Q_{t}(H)/H}{\nu(Q_{t}(H)/H)}<\infty.
\end{equation}
In Lemma 3.1 in~\cite{bansaye-2}, using the right-hand side condition in the above display, the authors obtain a Lyapunov equation defined as (\ref{ref-P-H-t-hom}) by replacing $H_{s,t}$ by
the function
$$
\Ha_{s,t}:=\frac{Q_{t-s}(H)}{\nu (Q_{t-s}(H)/H)}\neq H_{s,t}=\frac{Q_{t-s}(H)}{\overline{\eta}_sQ_{t-s}(1)}.
$$
Theorem 2.1 in~\cite{bansaye-2} ensures the existence of a leading triple $(\rho,\eta_{\infty},h)$, as well as exponential estimates similar to the ones discussed in Corollary~\ref{cor-projection-h}.  Thus, 
for the class of semigroups considered in Corollary~\ref{H1-prop} in the present article,  the conditions (\ref{ref-bertrand}) and (\ref{ref-bertrand-2}) ensures that the semigroup $Q_t$ is a $V$-positive semigroup.
Conversely, the authors show that the existence of a leading triple
 $(\rho,\eta_{\infty},h)$ satisfying these exponential decays imply that the pair $(V,h)$ satisfies condition (\ref{ref-bertrand}) and (\ref{ref-bertrand-2}).  

In the context of time homogeneous models, up to a change of Lyapunov function as discussed above, Proposition 3.3 in~\cite{bansaye-2} is closely related to Theorem~\ref{theo-intro-1} in the present article. In contrast with the $V$-norm Lipschitz estimates stated in Corollary~\ref{stab-h-Phi-transform-time} presented in this article,
Corollary 3.7 in~\cite{bansaye-2} does not provide any Lipschitz estimates but also yields some exponential decays of the normalised semigroup $\Phi_t$ to equilibrium with respect to the total variation norm.

\section{Proofs of the stability theorems}\label{proof-theo-sec}

\subsection{Proof of Theorem~\ref{krylov-bogo-theo}}\label{krylov-bogo-theo-proof}

\begin{proof}
{One direction of the proof is obvious. Indeed, if  $\Phi_{1}$ has at least one invariant probability measure $\eta_{\infty}$ choosing $\eta=\eta_{\infty}$ the measure $\Phi_n(\eta_{\infty})=\eta_{\infty}$ is tight and $\beta_n(\eta_{\infty})=\eta_{\infty}(Q(1))$.}

Conversely, assume that  for some $\eta$ the sequence of probability measures $\Phi_n(\eta)$ is tight and (\ref{tight-ref-cond}) is satisfied. In this situation, for any $\epsilon>0$ there exists a compact $K_{\epsilon}\subset E$ such that
  $$
\Phi_n(\eta)(Q(1))\geq \Phi_n(\eta)(1_{K_{\epsilon}}Q(1))\geq (1-\epsilon)~\inf_{K_{\epsilon}}Q(1)
\quad \mbox{\rm and}\quad\beta_{\infty}(\eta)>0.
$$
We simplify the notation and write $\beta_n$ and $\beta_{\infty}$ instead of $\beta_{n}(\eta)$ and
$\beta_{\infty}(\eta)$.
Consider the probability measures
$$
\eta_m:=\frac{1}{m}\sum_{0\leq k<m}\Phi_{k}(\eta)
\Longrightarrow
\eta_m(Q(1))=\frac{1}{m}\sum_{0\leq k<m}\beta_k\longrightarrow_{m\rightarrow\infty}\beta_{\infty}.
$$
There exists at least one probability measure $\eta_{\infty} :=\varpi(\eta)$  and a sub-sequence $m_k\rightarrow_{k\rightarrow\infty}\infty$ such that 
$\eta_{m_k}$ converges weakly to $\eta_{\infty}$ as $k\rightarrow\infty$.
Hence, $\eta_{m_k}$ and $\Phi_{1}(\eta_{m_k})$ converge weakly to $\eta_{\infty}$ and $\Phi_{1}(\eta_{\infty})$, respectively, as $m\rightarrow\infty$. In addition, we have
$$
\Phi_{n}(\eta)(Q(f))=\Phi_{n}(\eta)(Q(1))~\Phi_{n+1}(\eta)(f).
$$
This yields the formula
$$
\Phi_{1}(\eta_m)=\sum_{0\leq k<m}\frac{\beta_{k}}{
\sum_{0\leq l<m}\beta_{l}}~\Phi_{k+1}(\eta)
$$
from which we check that
\begin{eqnarray*}
\Phi_{1}(\eta_m)-\eta_m
& = &
\frac{1}{m}\sum_{0\leq k<m}\left(\frac{\beta_{k}}{
\frac{1}{m}\sum_{0\leq l<m}\beta_l}-1\right)~\Phi_{k+1}(\eta)\\ & &+
\frac{1}{m}\sum_{0\leq k<m}\left(\Phi_{k+1}(\eta)-\Phi_{k}(\eta)\right).
\end{eqnarray*}
As a result 
$$
\begin{array}{l}
\displaystyle \Phi_{1}(\eta_m)-\eta_m\\
\\
\displaystyle=
\frac{1}{m}\sum_{0\leq k<m}\left(\frac{\beta_{k}}{
\frac{1}{m}\sum_{0\leq l<m}\beta_l}-1\right)~\Phi_{k+1}(\eta)+\frac{1}{m}~
(\Phi_{m}(\eta)-\eta).
\end{array}$$
On the other hand, we have
$$
\frac{1}{m}\sum_{0\leq k<m}\left\vert\beta_{k}-
\frac{1}{m}\sum_{0\leq l<m}\beta_l\right\vert\leq \frac{2}{m}
\sum_{0\leq k<m}\vert\beta_{k}-\beta_{\infty}\vert\longrightarrow_{m\rightarrow\infty}0.
$$
For any $f\in \Ca_b(E)$ we conclude that
$$
\vert(\Phi_{1}(\eta_{m_k})-\eta_{m_k})(f)\vert \longrightarrow_{k\rightarrow\infty}0 \quad \mbox{\rm and therefore}\quad
\eta_{\infty}=\Phi_{1}(\eta_{\infty}).
$$

The last assertion follows from the fact that 
$$
\eta_{m_k}(Q(1))=\frac{1}{m_k}\sum_{l=0}^{m_k-1}\beta_l\longrightarrow_{m\rightarrow\infty}\beta_{\infty}=\eta_{\infty}(Q(1)).
$$
This ends the proof of the theorem.
\end{proof}
\subsection{Proof of Theorem~\ref{theo-intro-1}}\label{theo-intro-1-proof}

This section is mainly concerned with the proof of Theorem~\ref{theo-intro-1}.
For any $\mu\in \Pa_V(E)$ and $n\geq 1$, applying Markov's inequality we have
$$
r_n:=\kappa_V(\mu)+n\Longrightarrow
\Phi_{s,t}(\mu)(V>r_n)\leq \kappa_V(\mu)/(\kappa_V(\mu)+n)>0.
$$
We conclude that
\begin{equation}\label{ref-min-K}
\inf\Phi_{s,t}(\mu)(V\leq r_n)\geq{1}/{(1+\kappa_V(\mu)/n)}\longrightarrow_{n\rightarrow\infty}~1.
\end{equation}
In the above display, the infimum are taken over all 
 $s\in\Ta$ and $t\geq s$.

\begin{lem}\label{lem-proof-P-H-t-hom}
The estimate (\ref{ref-P-H-t-hom}) holds as soon as (\ref{ref-V-over-eta}) and (\ref{ref-H0-2}) are satisfied.
\end{lem}

 \begin{proof}
For any $n\geq 1$ we set $r_n:=\overline{\eta}(V)+n$.
For any $s\in \Ta$ and $r>  \lambda^-r_1$ we have the estimate
$$
Q_{s,s+\tau}(V)/V<(\lambda^-/r)~1_{K_{r}(s)^c}+\Vert  \Theta_{\tau}\Vert~1_{K_{r}(s)},
$$
with the sets $K_r(s)$ are defined by 
$$
K_{r}(s):=\left\{Q_{s,s+\tau}(V)/V\geq \lambda^-/r\right\}
\subset \Ka_{r}^-:= \left\{\Theta_{\tau}\geq \lambda^-/r\right\}.
$$
On the other hand, for any  $s\in \Ta $ such that $s+\tau\leq t$ we have
\begin{eqnarray*}
\frac{R^{(t)}_{s,s+\tau}(V/H_{s+\tau,t})}{(V/H_{s,t})}&=&\frac{1}{\lambda_{s,s+\tau}}~
\frac{Q_{s,s+\tau}(V)}{V}.
\end{eqnarray*}
This yields the estimate
\begin{eqnarray*}
\frac{R^{(t)}_{s,s+\tau}(V/H_{s+\tau,t})}{(V/H_{s,t})}&\leq &\frac{\lambda^-}{\lambda_{s,s+\tau}}~
\left(\frac{1}{r}1_{K_{r}(s)^c}+\frac{\Vert  \Theta_{\tau}\Vert}{\lambda^-}1_{K_{r}(s)}
\right)\\
&\leq &\frac{1}{r}~1_{K_{r}(s)^c}+\frac{\Vert  \Theta_{\tau}\Vert}{\lambda^-}~1_{K_{r}(s)},
\end{eqnarray*}
from which we check that
\begin{eqnarray*}
R^{(t)}_{s,s+\tau}(V/H_{s+\tau,t})
&\leq &\frac{1}{r}~(V/H_{s,t})~+\frac{\Vert \Theta_{\tau}\Vert}{\lambda^-}~~
\sup_{\Ka_{r}^-}(V/H_{s,t}).
\end{eqnarray*}
Now, choosing $n\geq 1$ sufficiently large such that
$$
r_{n}~\geq r/\lambda^-\geq r_{1}\vee r_H,
$$
by (\ref{k-in}) there exists some $\overline{r}_n$ such that
$$
 \Ka_{r}^-= \left\{\Theta_{\tau}\geq \lambda^-/r\right\}\subset  \Ka_{r_{n}}:=\left\{\Theta_{\tau}\geq 1/r_{n}\right\}\subset \left\{V\leq  \overline{r}_{n}\right\}.
$$
By (\ref{ref-H0-2}) this yields the uniform estimate
$$
\sup_{ \Ka_{r}^-}(V/H_{s,t})\leq  \overline{r}_{n}/\varsigma_{ \overline{r}_{n}}(H)
<\infty.
$$
This ends the proof of the lemma.
\end{proof}

 \begin{proof}[Proof of Theorem~\ref{theo-intro-1}]
Observe that
$$
\vertiii{H}_V=1\Longrightarrow H\leq V\quad \mbox{\rm and}\quad
H_{s,t}=\overline{Q}_{s,t}(H)\leq \overline{Q}_{s,t}(V).
$$
This implies that
$$ \lambda^-~
H_{s,t}/V\leq 
\lambda_{s,s+\tau}~
H_{s,t}/V=
Q_{s,s+\tau}(H_{s+\tau,t})/V\leq
Q_{s,s+\tau}(V)/V.
$$
from which we check that
$$
\lambda^-~V_{\theta}\leq V/H_{s,t}\quad \mbox{with}\quad V_{\theta}:=1/\Theta_{\tau}.
$$

Notice that $V_{\theta}$ has compact level sets and for any $r\geq 1$ there exists some $\varphi(r)\geq 1$ such that for any $\overline{r}\geq \varphi(r)$ we have
$$
\{V/H_{s,t}\leq r\}\subset \{ V_{\theta}\leq r/\lambda^-\}\subset \{V\leq \varphi(r)\}\subset \{V\leq \overline{r}\}.
$$
 By (\ref{loc-dob}) for any $ \overline{r}\geq \varphi(r)\vee r_0$ we have
$$
\sup_{(V/H_{s,t})(x)\vee(V/H_{s,t})(y)\leq r}\left\Vert \delta_xR^{(t)}_{s,s+\tau}-\delta_yR^{(t)}_{s,s+\tau}\right\Vert_{\tiny tv}\leq
1-\alpha(\overline{r})<1.
$$

Now, due to the previous lemma,~\eqref{ref-P-H-t-hom} holds which in turn implies that
$$
R^{(t)}_{s,s+\tau}(W_{s+\tau,t})\leq \epsilon~W_{s,t}+1
$$
with the collection of functions $W_{s,t}\geq 1$ defined by
$$
 \frac{\epsilon}{c}~\frac{V}{H_{s,t}}\leq 
W_{s,t}:=1+\frac{\epsilon}{c}~\frac{V}{H_{s,t}}\leq \left(1+\frac{\epsilon}{c}\right)~\frac{V}{H_{s,t}}.
$$
Then, for any $\overline{r}\geq \epsilon(\varphi(r)\vee r_0)/c$   we have
$$
 \sup_{W_{s,t}(x)\vee W_{s,t}(y)\leq r}\Vert\delta_xR^{(t)}_{s,s+\tau}-\delta_yR^{(t)}_{s,s+\tau}\Vert_{\tiny tv}\leq 1-\alpha\left(
{c\,\overline{r}}/{\epsilon}\right)<1.
$$
Applying  Lemma~\ref{lem-beta-12} (see for instance the contraction estimate (\ref{ref-V-contraction})).
 for any  $s\in \Ta$,  $u\in [s,t]_{\tau}$, $t\in [s,\infty[_{\tau}$  and $\mu,\eta\in \Pa_{V/H_{s,t}}$, we have the uniform contraction estimate
$$
\Vert \mu R^{(t)}_{s,u}-\eta R^{(t)}_{s,u}\Vert_{V/H_{u,t}}\leq a~ e^{-b (u-s)}~
\Vert \mu-\eta\Vert_{V/H_{s,t}}.
 $$
 This ends the proof of the theorem.
\end{proof}
 \begin{proof}[Proof of Corollary~\ref{theo-intro-1-cor-continuous}]

 For continuous time semigroups, for any $s_n=s+n\tau$ and $u\in [0,\tau[$ we have
 \begin{eqnarray*}
 R^{(t)}_{s_n,s_n+u}(V/H_{s_n+u,t})&=&(V/H_{s_n,t})~\frac{ R^{(t)}_{s_n,s_n+u}(V/H_{s_n+u,t})}{(V/H_{s_n,t})}\\
 &=&(V/H_{s_n,t})~\frac{1}{\lambda_{s_n,s_n+u}}~\frac{Q_{s_n,s_n+u}(V)}{V}\leq \frac{\pi_{\tau}(V)}{\lambda^-(\overline{\eta}_0)}~(V/H_{s_n,t}).
 \end{eqnarray*}
 This implies that
 \begin{eqnarray*}
 \Vert \mu R^{(t)}_{s,s_n+t}- \eta R^{(t)}_{s,s_n+t}\Vert_{V/H_{s_n+u,t}}&\leq &
\vert( \mu R^{(t)}_{s,s_n}- \eta R^{(t)}_{s,s_n})\vert R^{(t)}_{s_n,s_n+u}(V/H_{s_n+u,t})\\
&\leq & \frac{\pi_{\tau}(V)}{\lambda^-(\overline{\eta}_0)}~ \Vert \mu R^{(t)}_{s,s_n}- \eta R^{(t)}_{s,s_n}\Vert_{V/H_{s_n,t}}.
 \end{eqnarray*}
 This ends the proof of the theorem.
\end{proof}
 
 \subsection{Proof of Theorem~\ref{theo-intro-2}}\label{theo-intro-2-proof}
  \begin{proof}
Observe that
$$
\lambda_{s,t}~H_{s,t}~R^{(t)}_{s,t}(f/H)=~Q_{s,t}(f),
$$
which implies that
$$
\Psi_{H_{s,t}}(\mu)R^{(t)}_{s,t}(f/H)=\frac{\mu(Q_{s,t}(f))}{\mu(Q_{s,t}(H))}=\frac{\Phi_{s,t}(\mu)(f)}{\Phi_{s,t}(\mu)(H)}.
$$

Combining this with (\ref{link-s-t-Q-intro}) and the estimate (\ref{Lip-Psi-h-2}) we have
\begin{eqnarray*}
\vertiii{ \Phi_{s,t}(\mu)-\Phi_{s,t}(\eta)}_{V}
&=&
\vertiii{\Psi_{1/H}(\Psi_{H_{s,t}}(\mu)R^{(t)}_{s,t})-\Psi_{1/H}(\Psi_{H_{s,t}}(\eta)R^{(t)}_{s,t})}_{V}\\
& \leq&{\color{black}\Phi_{s,t}(\mu)(H)}
\left(1+\Phi_{s,t}(\eta)(V)\right)\times 
\vertiii{\left(\Psi_{H_{s,t}}(\mu)-\Psi_{H_{s,t}}(\eta)\right)R^{(t)}_{s,t}}_{V/H}.
\end{eqnarray*}

Applying (\ref{stab-time-varying-h}) to $u=t$ we find that
$$
\vertiii{ \Phi_{s,t}(\mu)-\Phi_{s,t}(\eta)}_{V}
 \leq  \kappa_H(\mu)
\left(1+ \kappa_V(\eta)\right)~a~ e^{-b (t-s)}
\vertiii{ \Psi_{H_{s,t}}(\mu)-\Psi_{H_{s,t}}(\eta)}_{V/H_{s,t}}.
$$
On the other hand, applying (\ref{Lip-Psi-h-2-back}) we check that
\begin{eqnarray*}
\vertiii{ \Psi_{H_{s,t}}(\mu)-\Psi_{H_{s,t}}(\eta)}_{V/H_{s,t}}
&\leq &  \frac{1}{\mu(H_{s,t})}\left(1+\frac{\eta(V)}{\eta(H_{s,t})}\right)~\vertiii{\mu-\eta}_{V}.
\end{eqnarray*}
This ends the proof of the theorem.
\end{proof}

\subsection{Proof of Theorem~\ref{theo-equivalence}}\label{theo-equivalence-proof}

\begin{proof}
Assume that  $Q_t$ is a stable $V$-positive semigroup. In this context, there exists an eigen-triple  $(\rho,\eta_{\infty},h)\in (\RR\times\Pa_V(E)\times \Ba_{0,V}(E))$ satisfying (\ref{ref-dominated-cv}).
Choosing $(\overline{\eta}_0,H)$ in (\ref{def-over-Q-H}) and (\ref{def-R-intro}), we  readily check that
$$
R^{(t)}_{s,s+u}(f)=P^h_{u}(f)\quad \mbox{\rm and}\quad
Q_{t}(V)/V=e^{\rho t}~P^h_{t}(V^h)/V^h\quad \mbox{\rm with}\quad V^h:=V/h\in \Ba_{\infty}(E).
$$
In this situation, the Doob h-transform,  $P^h_t$  is a $V^h$-positive semigroup of Markov operators from  $\Ba_{V^h}(E)$ to itself;  $P^h_t$ maps $\Ba_{V^h}(E)$ into $\Ba_{0,V^h}(E)$ for any $t>0$ and we have
$$
P^h_{\tau}(V^h)/V^h\leq e^{-\rho\tau}\Theta_{\tau}
$$ 
In addition, condition (\ref{loc-dob}) applied to $H=h$ takes the form
\begin{equation}\label{ref-lem-t}
\sup_{V(x)\vee V(y)\leq r}\left\Vert \delta_xP^h_{\tau}-\delta_yP^h_{\tau}\right\Vert_{\tiny tv}\leq 1-\alpha(r).
\end{equation}
Observe that
$$
V^h(x)\leq r_h\Longrightarrow V(x)\leq r=r_h\sup_{V\leq r_h}h~
$$
This implies that
\begin{equation}\label{loc-V-h-Ph}
\sup_{V^h(x)\vee V^h(y)\leq r}\left\Vert \delta_xP^h_{\tau}-\delta_yP^h_{\tau}\right\Vert_{\tiny tv}\leq 1-\alpha_h(r)\quad \mbox{\rm with}\quad
\alpha_h(r):=\alpha(r\sup_{V\leq r}h).
\end{equation}
We conclude that $P^h_t$ is a stable $V^h$-positive semigroup.

In the reverse angle, choosing $(\overline{\eta}_0,H)=(\eta_{\infty},h)\in (\Pa_V(E)\times\Ba_{0,V}(E))$ in (\ref{def-over-Q-H}) and (\ref{def-R-intro}), we  readily check that
$$
\overline{\eta}_0Q_{t,t+\tau}=e^{\rho \tau}=\lambda^-=\lambda^-(\eta_{\infty})
\quad \mbox{\rm and}\quad
H_{s,t}= h
$$
Note that in this case, the quantities $(\varsigma_{r}(H),\vertiii{H}_V)$ defined in (\ref{ref-H0-2}) become
$$
 \varsigma_{r}(H)=\inf_{V\leq r}~h>0\quad \mbox{\rm and}\quad \vertiii{H}_V=\Vert h\Vert_V<\infty.
$$
Now assume that $P^h_t$ satisfies (\ref{loc-V-h-Ph}) for some function $\alpha_h(r)$. Using the fact that
$$
V(x)\leq r\Longrightarrow V^h(x)\leq r/\inf_{V^h\leq r}h
$$
we check condition (\ref{loc-dob}) applied to $H=h$. We conclude that
$Q_t$ is a stable $V$-positive semigroup as soon as  $P^h_t$ is a stable $V^h$-positive semigroup. This ends the proof of the theorem.
\end{proof}

\subsection{Proof of Theorem~\ref{stab-h-transform-time}}\label{theo-homo-sec}
This section is mainly concerned with the proof of Theorem~\ref{stab-h-transform-time}
on the stability of the time homogeneous models discussed in (\ref{rho-h}). In what follow, we set $V^h:=V/h$.
\begin{lem}\label{ref-lem-Lyapunov-2}
For any $\epsilon>0$ we have the Foster-Lyapunov inequality
\begin{equation}\label{Lyap-P-h}
P^h_{\tau}\left(V^h\right)<  \epsilon~V^h~1_{K^c_{\epsilon,\tau}}+ c_{\epsilon,\tau}~ 1_{K_{\epsilon,\tau}}
\end{equation}
with the parameter $c_{\epsilon,\tau}$ and the compact set $ K_{\tau,\epsilon}$ given by
\begin{eqnarray*}
c_{\epsilon,\tau}&:=&e^{-\rho \tau}~\Vert Q_{\tau}(V)/V\Vert~\sup_{K_{\epsilon,\tau}}V^h
\quad\mbox{and}\quad
 K_{\tau,\epsilon}:=\{Q_{\tau}(V)/V\geq \epsilon ~e^{\rho\tau}\}.
\end{eqnarray*}
\end{lem}
\begin{proof}
For any $\epsilon>0$ we have
$$
{P^h_{\tau}(V^h)}/{V^h}=e^{-\rho\tau}{Q_{\tau}(V)}/{V}< \epsilon 1_{K_{\epsilon}^c}+
e^{-\rho \tau}\Vert Q_{\tau}(V)/V\Vert 1_{K_{\epsilon,\tau}}.
$$
This readily yields the estimate (\ref{Lyap-P-h}).
\end{proof}

\begin{prop}\label{prop-unif-eta-h-V}
For any $\mu\in \Pa_{V/h}(E)$ we have
$$
0<\kappa_{\tau}(\mu):=
\inf_{t\in [0,\infty[_{\tau}}\mu P^h_t(1/h)\leq
\kappa^h_{\tau,V^h}\left(\mu\right):=
\sup_{t\in [0,\infty[_{\tau}}\mu P^h_t(V^h)<\infty.
$$
In addition, for any $\eta\in\Pa_V(E)$ we have the estimates (\ref{ref-to-prop-unif-eta-h-V}).
\end{prop}
\begin{proof}
Following word-for-word the same arguments as the proof of (\ref{ref-Markov-Lyap}) the Foster-Lyapunov estimate (\ref{Lyap-P-h}) implies that
$$
\kappa^h_{\tau,V^h}\left(\mu\right)\leq  \mu(V^h)+c_{\epsilon,\tau}(1-\epsilon)^{-1} < \infty.
$$
Now, we come to the proof of the left-hand side estimate. Since $0<h\in \Ba_{0,V}(E)$ and $\Vert h\Vert_V=1$ the function  $V^h\geq 1$ has compact level sets and $h$ is bounded on compact sets. 
Consider the compact sets $\Ka^h_{\tau,V^h}(\delta)$ indexed by $\delta\in ]0,1[$ and
defined by
$$
\Ka^h_{V^h}(\mu,\delta):=
\{\delta V^h\leq \kappa^h_{\tau,V/h}\left(\mu\right)\}.
$$
Arguing as in (\ref{ref-Markov-ineq-compact}), we have
 $$
\inf_{t\in [0,\infty[_{\tau}}\mu P^h_t\left(\Ka^h_{V/h}(\mu,\delta)\right)\geq  1-\delta\quad
\mbox{\rm and}\quad
\quad
\kappa(\mu)\geq 
 {(1-\delta)}/{\sup_{\Ka^h_{V^h}(\mu,\delta)}h}.
$$
Observe that
$$
\sup_{t\in [0,\infty[_{\tau}}\Psi_h(\eta) P^h_t(V^h)\leq  \eta(V)/\eta(h)+c_{\epsilon,\tau}(1-\epsilon)^{-1}
$$
and consider the compact sets
$$
\Ka_{V}(\eta,\delta):=
\{x : \delta V(x) \leq   \eta(V)/\eta(h)+c_{\epsilon,\tau}(1-\epsilon)^{-1}\}\subset E.
$$
Arguing as above, 
we check that
$$
\inf_{t\in [0,\infty[_{\tau}}\Psi_h(\eta) P^h_t\left(\Ka_{V}(\eta,\delta)\right)\geq  1-\delta
~~\mbox{\rm and}~~
\inf_{t\in [0,\infty[_{\tau}}\Psi_h(\eta) P^h_t(1/h)\geq 
\frac{(1-\delta)}{\sup_{\Ka_{V}(\eta,\delta)}h}.
 $$
This yields for any $t\in[0,\infty[_{\tau}$ and $\delta\in ]0,1[$ the uniform estimate
$$
\Phi_{t}(\eta)(V)=\frac{\Psi_{h}(\eta)P^h_{t}(V/h)}{\Psi_{h}(\eta)P^h_{t}(1/h)}\leq  \frac{\eta(V)/\eta(h)+c_{\epsilon,\tau}(1-\epsilon)^{-1}}{{(1-\delta)}/{\sup_{\Ka_{V}(\eta,\delta)}h}}.
$$
This implies that $\kappa_{\tau,V}(\eta)<\infty$. By lemma~\ref{lem-d-to-c}, for continuous time indices we also have
$\kappa_{V}(\eta)<\infty$.
This also ensures that the sequence of probability measures
$
\Phi_{s,t}(\eta)
$
indexed by $s\leq t$ is tight.  Choosing the compact set
\begin{equation}\label{def-compact-set-proof}
K(\delta,\eta):=
\left\{\delta~ V\leq \kappa_V(\eta)\right\}
\end{equation}
we readily check that
$$
\Phi_{s,t}(\eta)(h)\geq \left(\inf_{K(\delta,\eta)}h\right)~\Phi_{s,t}(\eta)(K(\delta,\eta))
\geq (1-\delta)~\left(\inf_{K(\delta,\eta)}h\right)>0.
$$
We conclude that $\kappa_h^-(\eta)>0$. This ends the proof of the proposition.
\end{proof}

\begin{proof}[Proof of Theorem~\ref{stab-h-transform-time}]
\textcolor{black}{The estimates  (\ref{ref-to-prop-unif-eta-h-V}) have been checked in Proposition~\ref{prop-unif-eta-h-V}.}

Observe that
$$
(\ref{Lyap-P-h})\Longrightarrow
P^h_{\tau}\left(W\right)\leq \epsilon~W +1$$
with the function $W\geq 1$ defined by
\begin{equation}\label{ineq-V-norms-equiv}
 \frac{\epsilon}{c_{\epsilon,\tau}}~\frac{V}{h}\leq 
 W:=1+\frac{\epsilon}{c_{\epsilon,\tau}}~\frac{V}{h}\leq \left(1+\frac{\epsilon}{c_{\epsilon,\tau}}\right)~\frac{V}{h},
\end{equation}
which has compact level sets. 
On the other hand, by (\ref{ref-P-h-n-2-hom}) we have
\begin{eqnarray*}
\sup_{W(x)\vee W(y)\leq r}\Vert\delta_xP^h_{\tau}-\delta_yP^h_{\tau}\Vert_{\tiny tv}\leq 1-\alpha(c_{\epsilon,\tau}r/\epsilon).
\end{eqnarray*}
The estimate (\ref{first-h-estimate}) is now a direct consequence of  (\ref{ref-V-contraction}) and the estimates (\ref{ineq-V-norms-equiv}).
 The proof of the theorem is now completed.
 \end{proof}

\appendix

\section{Proofs of Several Technical Reaults}

\subsection{Proof of Lemma~\ref{lem-d-to-c}}\label{lem-d-to-c-proof}

 Condition $\kappa_{\tau,V}(\mu)<\infty$ ensures that that the sequence of probability measures
$
\Phi_{s,t}(\mu)
$
indexed by $s\in \Ta$ and $t\in [s,\infty[_{\tau}$ is tight. 
In this situation, for any $\delta\in ]0,1[$ there exists some compact set $K$ such that $  \Phi_{s,t}(\mu)(K)\geq (1-\delta)$ for any $s\in \Ta$ and $t\in [s,\infty[_{\tau}$. Thus,
 for any $s_n:=s+n\tau$ and $\epsilon\in [0,\tau[$ we have
$$
\Phi_{s,s_n+\epsilon}(\mu)(V)=\frac{\Phi_{s,s_n}(\mu)Q_{s_n,s_n+\epsilon}(V)}{\Phi_{s,s_n}(\mu)Q_{s_n,s_n+\epsilon}(1)}\leq (\pi_{\tau}(V)/\pi^-_{\tau}(K))~\kappa_V(\mu)/(1-\delta)
$$
with the parameter $\pi_\tau^-(K_{r})$ defined in (\ref{pi-min}).
 In the same vein, using (\ref{discrete-2-continuous}) we have
$$
\Phi_{s,s_n+\epsilon}(\mu)(H)=\frac{\Phi_{s,s_n+\epsilon}(\mu)(H)}{\Phi_{s,s_n}(\mu)Q_{s_n,s_n+\epsilon}(1)}\geq \inf_KH~(1-\delta)/\pi_{\tau}>0.
$$

This shows that $\Phi_{s,t}(\mu)(V)$ is uniformly bounded and $\Phi_{s,t}(\mu)(H)$ is uniformly  positive constant with respect to the parameters $s\in \Ta$ and $t\geq s$. In addition,  the tightness of $\Phi_{s,t}(\mu)$ combined with (\ref{discrete-2-continuous-intro-K})  ensures that $\Phi_{s,t}(\mu)(Q_{t,t+\epsilon}(1))$  is uniformly positive with respect to~$s\in \Ta$ and $t\geq s$ and $\epsilon\in [0,\tau[$.

For any $\delta\in ]0,1[$ there exists some compact set $K$ such that  $\pi^-_{\tau}(K)>0$  and $\Phi_{s,t}(\mu)(K)\geq (1-\delta)$ and
$$
(\ref{discrete-2-continuous-intro-K})\Longrightarrow
\sup_{s\geq 0}\sup_{t\geq s}\sup_{\epsilon\in [0,\tau]}\frac{\Vert Q_{t,t+\epsilon}(1)\Vert}{\Phi_{s,t}(\mu) Q_{t,t+\epsilon}(1)}\leq \frac{\pi_{\tau}}{(1-\delta)\pi_\tau^-(K)} <\infty.
$$
Finally observe that for any $s\leq t$ and $\epsilon\in [0,\tau]$ we have
$$
(\ref{discrete-2-continuous})\Longrightarrow
\Phi_{s,t}(\mu)(Q_{t,t+\epsilon}(1))\leq \Phi_{s,t}(\mu)(V~Q_{t,t+\epsilon}(V)/V)\leq
\pi_{\tau}(V)~\kappa_V(\mu).
$$
This shows that  $\Phi_{s,t}(\mu)(Q_{t,t+\epsilon}(1))$  is uniformly bounded with respect to  $s\in \Ta$ and $t\geq s$. 
This ends the proof of the lemma.
\cqfd

\subsection{Proof of Lemma~\ref{lemma-Nick}}\label{app:lemma-Nick-proof}

Let $H$ be some  locally bounded  positive functions s.t. $V/H\in \Ba_{\infty}(E)$.
Recall that for any $u,t\in \Ta$, we have
 \begin{equation}\label{ref-R-t-h}
\lambda_{u, t}~H_{u, t} = Q_{u,t}(H) =: h_{u, t}\Longrightarrow
Q_{s,u}(h_{u,t})=h_{s,t}>0.
\end{equation} 
By (\ref{def-V-epsilon}) and (\ref{def-V-epsilon-2}) there exists some 
for  some $ \epsilon_1>0$ and any $0<\epsilon\leq \epsilon_{1}$ we have
\begin{equation}\label{ref-q-chi-2}
 0<\iota^-_{\epsilon}:=\inf_{t\in \Ta}\inf_{\Va_{\epsilon}^2}q_{t,t+\tau}\leq
\iota_{\epsilon}:=\sup_{t\in \Ta}\sup_{ \Va_{\epsilon}^2}q_{t,t+\tau} <\infty\quad \mbox{and}\quad
0 < \nu_{\tau}(\Va_{\epsilon}) < \infty
\end{equation} 
with the compact $\epsilon$-super-level sets $\Va_{\epsilon}$ of the function $\Theta_{\tau}$ defined in (\ref{def-V-epsilon}).
To simplify the notation, we write $\nu$ instead of $\nu_{\tau}$. In this notation,  we have
 \begin{equation}\label{def-nu-min}
\iota^-_{\epsilon}~1_{\Va_{\epsilon}}(x)~ \nu(dy)~1_{\Va_{\epsilon}}(y)
\leq 1_{\Va_{\epsilon}}(x)~Q_{t,t+\tau}(x,dy)~1_{\Va_{\epsilon}}(y)\leq ~\iota_{\epsilon}~1_{\Va_{\epsilon}}(x)~\nu(dy)~
1_{\Va_{\epsilon}}(y).
\end{equation}

The next lemma is a slight modification of \cite[Proposition 1]{whiteley}.
\begin{lem}
There exists $\lambda_0>0$ and a finite constant $c_0<\infty$ such that for any $s\in\Ta$ and $t\in [s,\infty[_{\tau}$ and $\eta\in \Pa_V(E)$  we have 
\begin{equation}\label{ref-P-h-nh}
 R^{(t)}_{s,t}\left(\frac{V}{H}\right)\leq 
e^{-\lambda_0 (t-s)}~\frac{V}{h_{s,t}}+c_0\quad \mbox{and}\quad \limsup_{t\rightarrow\infty}
\frac{e^{-\lambda_0 (t-s)}}{\eta(h_{s,t})}=0.
\end{equation}
In addition, we have the uniform estimates stated in (\ref{ref-P-h-nh-2}).
\end{lem}

\begin{proof}
We set
$$
h^V_{s,t}:={h_{s,t}}/{\Vert h_{s,t}\Vert_{V}}
\quad \mbox{\rm and}
\quad  
 \beta_{s,u}^{(t)}=
\Vert h_{s,t}\Vert_{V}/\Vert h_{u,t}\Vert_{V}.
$$
We have
$$
R^{(t)}_{s,s+\tau}\left(\frac{V}{h^V_{s+\tau,t}}\right):=\frac{1}{\beta^{(t)}_{s,s+\tau}}~\frac{V}{h^V_{s,t}}~\frac{Q_{s,s+\tau}(V)}{V}~~\mbox{\rm and}~~
Q_{s,s+\tau}(h^V_{s+\tau,t})=\beta^{(t)}_{s,s+\tau} ~h^V_{s,t}.
$$
Using (\ref{def-V-epsilon}) and (\ref{def-nu-min}) and recalling that $ Q_{t,t+\tau}(V)/V\leq \Theta_{\tau}$, for any $0\leq \epsilon<\epsilon_{1}$ 
we check that
$$
R^{(t)}_{s,s+\tau}\left(\frac{V}{h^V_{s+\tau,t}}\right)\leq \frac{\epsilon}{\beta^{(t)}_{s,s+\tau}}~\frac{V}{h^V_{s,t}}+\frac{a_{\epsilon}}{\nu(1_{\Va_{\epsilon}}h^V_{s+\tau})}
\quad\mbox{\rm
with}\quad
a_{\epsilon}:=\sup_{\Va_{\epsilon}}(V\Theta_{\tau} )/\iota_{\epsilon}^-.
$$
This implies that
$$
R^{(t)}_{s,s+n\tau}\left(\frac{V}{h^V_{s+n\tau,t}}\right)\leq
\frac{\epsilon}{\beta^{(t)}_{s+(n-1)\tau,s+n\tau}} R^{(t)}_{s,s+(n-1)\tau}\left(\frac{V}{h^V_{s+(n-1)\tau,t}}\right)+\frac{a_{\epsilon}}{\nu(1_{\Va_{\epsilon}}h^V_{s+n\tau,t})}.
$$
Applying the above to $s+n\tau=t$ we obtain the formula
$$
\begin{array}{l}
\displaystyle R^{(t)}_{s,t}\left(\frac{V}{H}\right)\\
\\
\displaystyle\leq 
\frac{\epsilon^n}{\beta^{(t)}_{s,t}}\frac{V}{h^V_{s,t}}+\sum_{1\leq k\leq n}
\frac{\epsilon^{n-k}}{\beta^{(t)}_{s+k\tau,t}}
\frac{a_{\epsilon}}{\nu(1_{\Va_{\epsilon}}h^V_{s+k\tau,t})}=
\frac{\epsilon^n}{h_{s,t}}~V+\sum_{1\leq k\leq n}
\epsilon^{n-k}
\frac{a_{\epsilon}}{\nu(1_{\Va_{\epsilon}}h_{s+k\tau,t})}.
\end{array}
$$
Observe that  for any $ 0<\epsilon\leq \epsilon_2<\epsilon_{1}$ we have $\Va_{\epsilon_2}\subset \Va_{\epsilon}$. Thus,  for any $k\leq n$ and $t=s+n\tau$ we have the lower bound estimate
\begin{eqnarray*}
1_{\Va_{\epsilon}}(x)h_{s+k\tau,s+n\tau}(x)&\geq &1_{\Va_{\epsilon_2}}(x)Q_{s+k\tau,s+(k+1)\tau}1_{\Va_{\epsilon_2}}\dots  Q_{s+(n-1)\tau,s+n\tau}(1_{\Va_{\epsilon_2}}H)(x)\\
&\geq &1_{\Va_{\epsilon_2}}(x)\left(\inf_{\Va_{\epsilon_2}}H\right)(\iota_{\epsilon_2}^-\nu(\Va_{\epsilon_2}))^{n-k}.
\end{eqnarray*}

Choosing $0<\epsilon<\epsilon_3:=\epsilon_2\wedge (\iota_{\epsilon_2}^-\nu(\Va_{\epsilon_2})/2)$ we conclude that
$$
R^{(t)}_{s,t}\left(\frac{V}{H}\right)\leq 
\frac{\epsilon^n}{h_{s,t}}~V+\frac{2a_{\epsilon}}{\nu(\Va_{\epsilon_2})}~\frac{1}{\inf_{\Va_{\epsilon_2}}H}.
$$
In the same vein, for any $\eta\in \Pa_V(E)$ we have
$$
\eta(h_{s,s+n\tau})\geq \eta(\Va_{\epsilon_2})~\left(\inf_{\Va_{\epsilon_2}}H\right)~(\iota_{\epsilon_2}^-~\nu(\Va_{\epsilon_2}))^{n}\geq
\eta(\Va_{\epsilon_2})~\left(\inf_{\Va_{\epsilon_2}}H\right)~(2\epsilon)^{n}.
$$
This yields the estimate
$$
\frac{\epsilon^{(t-s)/\tau}}{\eta(h_{s,t})}\leq 2^{-(t-s)/\tau}~\frac{1}{\eta(\Va_{\epsilon_2})~\inf_{\Va_{\epsilon_2}}H}.
$$
To summarise: there exists some $\epsilon_3>0$ such that for any $s\in \Ta $ s.t. $t\in [s,\infty[_{\tau}$ and for any $0<\epsilon<\epsilon_3$ we have
$$
R^{(t)}_{s,t}\left(V/H\right)\leq 
\epsilon^{(t-s)/\tau} ~V/h_{s,t}+c_{\epsilon}\quad \mbox{and}\quad \limsup_{t\rightarrow\infty}
\frac{\epsilon^{(t-s)/\tau}}{\eta(h_{s,t})}=0.
$$
This ends the proof of (\ref{ref-P-h-nh}). The last assertion follows from the fact that
$$
\left(\Psi_{h_{s,t}}\left(\eta\right) R^{(t)}_{s,t}\right)\left({V}/{H}\right)\leq 
\frac{e^{-\lambda_0 (t-s)}}{\eta(h_{s,t})}~\eta(V)+c_0\leq c_1(\eta)
$$
for some finite constant $c_1(\eta)<\infty$.
This implies that the sequence of probability measures
$$
\Psi_{h_{s,t}}\left(\eta\right) R^{(t)}_{s,t}
$$
indexed by $s\in \Ta $ s.t. $t\in [s,\infty[_{\tau}$ is tight. More precisely, choosing the compact set
$$
K_{\epsilon,\eta}:=
\left\{x\in E~:~\epsilon~ (V/H)(x)\leq c_1(\eta)\right\}
$$
we have
$$
\left(\Psi_{h_{s,t}}\left(\eta\right) R^{(t)}_{s,t}\right)(K_{\epsilon,\eta}^c)\leq \epsilon
\quad \mbox{\rm so that}\quad
\left(\Psi_{h_{s,t}}\left(\eta\right) R^{(t)}_{s,t}\right)(1/H)\geq \frac{1-\epsilon}{\sup_{K_{{\epsilon,\eta}}}H}.
$$
On the other hand, we have
$$
\Phi_{s,t}(\eta)(V)=
\Psi_{1/H}\left(\Psi_{h_{s,t}}\left(\eta\right) R^{(t)}_{s,t}\right)(V)\leq
c_1(\eta)~ \frac{\sup_{K_{{\epsilon,\eta}}}H}{1-\epsilon}.
$$
We conclude that $\kappa_{\tau,V}(\eta)<\infty$.
This also shows that the sequence of probability measures
$
\Phi_{s,t}(\eta)
$
indexed by $s\in \Ta$ and $t\in [s,\infty[_{\tau}$ is tight.  Choosing the compact set
\begin{equation}\label{def-compact-set-proof}
K_{\epsilon,\eta}:=
\left\{x\in E~:~\epsilon~ V(x)\leq \kappa_V(\eta)\right\}
\end{equation}
we readily check that
$$
\Phi_{s,t}(\eta)(H)\geq \left(\inf_{K_{\epsilon,\eta}}H\right)~\Phi_{s,t}(\eta)(K_{\epsilon,\eta})
\geq (1-\epsilon)~\left(\inf_{K_{\epsilon,\eta}}H\right)>0.
$$
We conclude that $\kappa_{\tau,H}^-(\eta)>0$.
 To take the final step,
  observe that for any $0<\epsilon<\epsilon_{0}$ we have
 \begin{equation}\label{def-nu-min-ref-1}
(\ref{def-nu-min})\Longrightarrow
 \inf_{t\in \Ta}\inf_{\Va_{\epsilon}}Q_{t,t+\tau}(1)\geq  \nu(\Va_{\epsilon})>0.
\end{equation}
Thus, for any $t\in [s,\infty[_{\infty}$ and $0<\epsilon<\epsilon_{0}$ we have
$$
\Phi_{s,t}(\eta)(Q_{t,t+\tau}(1))\geq 
\Phi_{s,t}(\eta)\left(1_{\Va_{\epsilon}}Q_{t,t+\tau}(1)\right)\geq \Phi_{s,t}(\eta)(\Va_{\epsilon})~\nu(\Va_{\epsilon}).
$$
Similarly, we have
$$
\Phi_{s,t}(\eta)(Q_{t,t+\tau}(1))\leq \Phi_{s,t}(\eta)(V~Q_{t,t+\tau}(V)/V)\leq
\Vert\Theta_{\tau}\Vert~\kappa_V(\eta).
$$
By lemma~\ref{lem-d-to-c}, this ends the proof of the estimates in (\ref{ref-P-h-nh-2}) for continuous or discrete time indices. 
 The proof of the lemma is completed.
\end{proof}

The next result is a variation of \cite[lemma 10]{whiteley}.
\begin{lem}\label{lem-10-Nick}
We have the estimates (\ref{ref-H0-2}). In addition, for any $0<\delta<1$ there exists some $0<\epsilon\leq \epsilon_1$  such that the following uniform  compact-approximation  estimate holds
\begin{equation}\label{ref-max-H}
\sup_{t\in \Ta}\vertiii{\overline{Q}_{t,t+\tau}-1_{\Va_{\epsilon}}~\overline{Q}_{t,t+\tau}}_V<\delta
\quad\mbox{and}\quad
\sup_{s\in \Ta}
\sup_{t\geq s}\vertiii{\overline{Q}_{s,t}}_V<\infty.
\end{equation}
\end{lem}

\begin{proof}
We start by proving the estimates given in~\eqref{ref-max-H}. We use the same notation as in the proof of Lemma~\ref{lem-proof-P-H-t-hom} and we set $\Va_{\epsilon}:=\{\Theta_{\tau}\geq \epsilon\}$ . For any $n\geq 1$ such that $\epsilon_{n}^-:=\epsilon_{n}/\lambda^-<1$ 
$$
\begin{array}{l}
\displaystyle
\vertiii{1_{\Va_{\epsilon_{n}}^c}~\overline{Q}_{s,s+\tau}}_V\\
\\
\displaystyle=\vertiii{\overline{Q}_{s,s+\tau}-1_{\Va_{\epsilon_{n}}}~\overline{Q}_{s,s+\tau}}_V=\frac{1}{\lambda_{s,s+\tau}}~\Vert ({Q_{s,s+\tau}(V)}/{V})~1_{\Va^c_{\epsilon_{n}}}\Vert< \epsilon_{n}^-.
\end{array}$$
This yields the left-hand side estimate in (\ref{ref-max-H}).

For the right-hand side, for any $i\in \{0,1\}^n$ and  $1\leq k\leq n$, we set
$$
l(i):=\inf\left\{1\leq k< n:(i_k,i_{k+1})=(1,1)\right\}
$$
and
$$
\{0,1\}^n_k=\{i\in \{0,1\}^n:l(i)=k\}
$$ 
 with the convention that
 $$
  \{0,1\}^n_n=\{i\in \{0,1\}^n~:~\forall 1\leq k< n~:~(i_k,i_{k+1})\not=(1,1)\}.
$$
Further, set
$$
\Va_{\epsilon}(0)=\Va_{\epsilon}^c\quad\mbox{\rm and}\quad\Va_{\epsilon}(1)=\Va_{\epsilon}.
$$
Then we may decompose $\overline{Q}_{s_0,s_n}$ as follows:
\begin{align*}
\overline{Q}_{s_0,s_n}
&=\sum_{i\in \{0,1\}^n}1_{\Va_{\epsilon}(i_1)}\overline{Q}_{s_0,s_1}\ldots 1_{\Va_{\epsilon}(i_n)}\overline{Q}_{s_0,s_n}\\
&= \sum_{1\leq k\leq n}\sum_{i\in \{0,1\}^n_k}1_{\Va_{\epsilon}(i_1)}\overline{Q}_{s_0,s_1}\ldots 1_{\Va_{\epsilon}(i_n)}\overline{Q}_{s_{n-1},s_n}.
\end{align*}

Our aim is to obtain suitable bounds for the summands in the above decomposition. To this end, set for any $0<\epsilon\leq \epsilon_n$
$$
\nu_{\epsilon}(dx):=\frac{\nu(dx)1_{\Va_{\epsilon}}(x)}{\nu(\Va_{\epsilon})}.
$$
Using this and the notation introduced in~\eqref{def-nu-min}, we have
$$
\begin{array}{l}
\displaystyle 1_{\Va_{\epsilon}}(x)~\overline{Q}_{s,s+\tau}\left(1_{\Va_{\epsilon}}~\overline{Q}_{s+\tau,t}(V)\right)(x)\\
\\
\displaystyle\leq \iota_{\epsilon}~\frac{\nu\left(1_{\Va_{\epsilon}}~Q_{s+\tau,t}(V)\right)}{\overline{\eta}_s\left(Q_{s,s+\tau}Q_{s+\tau,t}(1)\right)}\leq \iota_{\epsilon}~\frac{\nu\left(1_{\Va_{\epsilon}}~Q_{s+\tau,t}(V)\right)}{\overline{\eta}_s\left(1_{\Va_{\epsilon}}~Q_{s,s+\tau}1_{\Va_{\epsilon}}~ Q_{s+\tau,t}(1)\right)}\\
\\
\displaystyle\leq \frac{\iota_{\epsilon}}{\overline{\eta}_s(\Va_{\epsilon})}~\frac{\nu\left(1_{\Va_{\epsilon}}~Q_{s+\tau,t}(V)\right)}{\nu\left(1_{\Va_{\epsilon}}~Q_{s+\tau,t}(1)\right)}=\frac{\iota_{\epsilon}/\iota_{\epsilon}^-}{\overline{\eta}_s(\Va_{\epsilon})}~\Phi_{s+\tau,t}\left(\nu_{\epsilon}\right)(V)\leq \kappa_V(\nu_{\epsilon})~\frac{\iota_{\epsilon}/\iota_{\epsilon}^-}{\overline{\eta}_s(\Va_{\epsilon})}.
\end{array}
$$
This yields the uniform estimate
$$
 1_{\Va_{\epsilon_{n}}}(x)\overline{Q}_{s,s+\tau}\left(1_{\Va_{\epsilon_{n}}}\overline{Q}_{s+\tau,t}(V)\right)(x)
 \leq \iota^{\prime}_{\epsilon_n}:=\left(\frac{\iota_{\epsilon_{n}}}{\iota_{\epsilon_n}^-}\right)\kappa_V(\nu_{\epsilon_{n}})\left(1+ \frac{\overline{\eta}(V)}{n}\right).
$$

Next fix $1 \le k < n$.
Then one has
\begin{align}
\vertiii{1_{\Va_{\epsilon}(i_1)}\overline{Q}_{s_0,s_1}\ldots 1_{\Va_{\epsilon}(i_{n})}\overline{Q}_{s_{n-1},s_n}}_V
&\leq  \iota^{\prime}_{\epsilon_{n}}\vertiii{1_{\Va_{\epsilon}(i_1)}\overline{Q}_{s_0,s_1}\ldots 1_{\Va_{\epsilon}(i_{k})}\overline{Q}_{s_{k-1},s_k}}_V\notag\\
&\leq\iota^{\prime}_{\epsilon_{n}}\prod_{1\leq m \le k}\vertiii{1_{\Va_{\epsilon}(i_m)}\overline{Q}_{s_{m-1},s_m}}_V\\
&\leq ~\iota^{\prime}_{\epsilon_{n}}\prod_{1\leq m \le k}\Vert\Theta_{\tau}\Vert^{1_{i_m=1}}\left(\epsilon_{n}^-\right)^{1_{i_m=0}}\notag\\
&= \iota^{\prime}_{\epsilon_{n}}\Vert\Theta_{\tau}\Vert^{\sum_{m=1}^{k}1_{i_m=1}}\left(\epsilon_{n}^-\right)^{\sum_{m=1}^k1_{i_m=0}}.
\label{eq-k}
\end{align}

Similarly, for the case $k = n$, we have
\begin{equation}
\vertiii{1_{\Va_{\epsilon}(i_1)}\overline{Q}_{s_0,s_1}\ldots 1_{\Va_{\epsilon}(i_{n})}\overline{Q}_{s_{n-1},s_n}}_V
\leq  \Vert\Theta_{\tau}\Vert^{\sum_{m=1}^{n}1_{i_m=1}}~\left(\epsilon_{n}^-\right)^{\sum_{m=1}^n1_{i_m=0}}.
\label{eq-n}
\end{equation}
Now note that for any $i\in \{0,1\}^n$ we have
$$
\sum_{1\leq k\leq n}1_{i_k=1}\leq \frac{n+1}{2}+\frac{1}{2}\sum_{1\leq k<n}1_{(i_k,i_{k+1})=(1,1)}.
$$
This follows from the fact that
\begin{eqnarray*}
n-1&\geq &\sum_{1\leq k<n}\left(1_{(i_k,i_{k+1})=(1,0)}+1_{(i_k,i_{k+1})=(0,1)}\right)\\
&=&\left(2\sum_{1\leq k\leq n}1_{i_k=1}-(1_{i_n=1}+1_{i_1=1})\right)-\sum_{1\leq k<n}1_{(i_k,i_{k+1})=(1,1)}\\
&\geq&2\left(\sum_{1\leq k\leq n}1_{i_k=1}-1\right)-\sum_{1\leq k<n}1_{(i_k,i_{k+1})=(1,1)}.
\end{eqnarray*}
Equivalently, we have
$$
\sum_{1\leq k\leq n}1_{i_k=0}\geq \frac{n-1}{2}-\frac{1}{2}\sum_{1\leq k<n}1_{(i_k,i_{k+1})=(1,1)}.
$$
Further observe that for any  $i\in \{0,1\}^n_k$ with $1\le k \le n$ we have
$$
\sum_{m=1}^k
1_{(i_m,i_{m+1})=(1,1)}=0
\Longrightarrow
\sum_{m=1}^k 1_{i_m=1}\leq \frac{k + 1}{2}~~
\mbox{\rm and}~~
\sum_{m=1}^k 1_{i_m=0}\geq \frac{k-1}{2}.
$$
From~\eqref{eq-k} and~\eqref{eq-n}, respectively, we conclude that for any  $i\in \{0,1\}^n_k$ with $1\le k<n$ we have
$$
\begin{array}{l}
\displaystyle
\vertiii{1_{\Va_{\epsilon}(i_1)}\overline{Q}_{s_0,s_1}\ldots 1_{\Va_{\epsilon}(i_{n})}\overline{Q}_{s_{n-1},s_n}}_V\leq \iota^{\prime}_{\epsilon_{n}}~(1\vee\Vert\Theta_{\tau}\Vert)^{(k + 1)/2}~\left(\epsilon_{n}^-\right)^{(k-1)/2}.
\end{array}$$
and for $k = n$, we have
\begin{equation*}
\vertiii{1_{\Va_{\epsilon}(i_1)}\overline{Q}_{s_0,s_1}\ldots 1_{\Va_{\epsilon}(i_{n})}\overline{Q}_{s_{n-1},s_n}}_V
\leq (1\vee\Vert\Theta_{\tau}\Vert)^{(n+1)/2}~\left(\epsilon_{n}^-\right)^{(n-1)/2}.
\end{equation*}
We end the proof of the right-hand side estimate in (\ref{ref-max-H}) by choosing $n\geq 1$ such that $$
\epsilon_{n}^-<1\wedge (1/\Vert\Theta_{\tau}\Vert).
$$
\textcolor{black}{
Now for the estimates in~\eqref{ref-H0-2}, observe that
$$
H\leq V\Longrightarrow
\sup_{s\in \Ta}
\sup_{t\in [s,\infty[_{\tau}}\Vert \overline{Q}_{s,t}(H)\Vert_V\leq \sup_{s\in \Ta}
\sup_{t\in [s,\infty[_{\tau}}\vertiii{\overline{Q}_{s,t}}_V<\infty.
$$
By remark~\ref{rmk-discrete-2-continuous}, this yields right-hand side estimate in (\ref{ref-H0-2})}.

The proof of the left-hand side estimate in (\ref{ref-H0-2}) follows the same lines of arguments as those given for~\eqref{ref-max-H}, thus it is only sketched. Using the same notation as above, we have
$$
\overline{\eta}_{s_0}Q_{s_0,s_n}(1)=\sum_{1\leq k\leq n}\sum_{i\in \{0,1\}^n_k}\overline{\eta}_{s_0}1_{\Va_{\epsilon}(i_1)}Q_{s_0,s_1}\ldots 1_{\Va_{\epsilon}(i_n)}Q_{s_{n-1},s_n}(1).
$$
For any
$
i \in \{0,1\}^n_k$ for some $1\le k<n
$, we also have
$$
 \begin{array}{l}
\displaystyle\overline{\eta}_{s_0}1_{\Va_{\epsilon}(i_1)}Q_{s_0,s_1}\ldots 1_{\Va_{\epsilon}(i_n)}Q_{s_{n-1},s_n}(1)\\
\\
\displaystyle\leq \iota_{\epsilon}~\overline{\eta}_{s_0}(V1_{\Va_{\epsilon}(i_1)}Q^V_{s_0,s_1}\ldots 1_{\Va_{\epsilon}(i_k)}Q^V_{s_{k-1},s_k}(1_{\Va_{\epsilon}}))~\nu(1_{\Va_{\epsilon}}Q_{s_{k+1},s_n}(1))\\
\\
\displaystyle\leq \iota_{\epsilon}~ \overline{\eta}(V)~(1\vee\Vert\Theta_{\tau}\Vert)^{(k + 1)/2}~\epsilon^{(k-1)/2}~\nu(1_{\Va_{\epsilon}}Q_{s_{k+1},s_n}(1)),
 \end{array}
 $$
 with $Q^V_{s,t}(f):=Q_{s,t}(f V)/V.$
For any $\epsilon\leq \epsilon_1$ we have
$$
 \begin{array}{l}
\displaystyle
1_{\Va_{\epsilon}}(x) Q_{s_0,s_n}(H)(x)\\
\\
\displaystyle\geq1_{\Va_{\epsilon}}(x)~1_{\Va_{\epsilon_1}}(x) \left(Q_{s_0,s_1}1_{\Va_{\epsilon_1}} Q_{s_1,s_n}(H)\right)(x)\geq \iota^-_{\epsilon}~1_{\Va_{\epsilon}}(x)~\nu\left(1_{\Va_{\epsilon_1}}Q_{s_1,s_n}(H)\right).
 \end{array}$$
In addition, we have
$$
 \begin{array}{l}
\displaystyle\nu\left(1_{\Va_{\epsilon_1}}Q_{s_1,s_n}(H)\right)
\\
\\
\geq \displaystyle\nu \left(1_{\Va_{\epsilon_1}}Q_{s_1,s_2}1_{\Va_{\epsilon}}\ldots Q_{s_{k-1},s_{k}}1_{\Va_{\epsilon}}Q_{s_{k},s_{k+1}}1_{\Va_{\epsilon}}Q_{s_{k+1},s_n}(H)\right)\\
\\
\displaystyle\geq \iota_{\epsilon}^-~ \nu \left(1_{\Va_{\epsilon_1}}Q_{s_1,s_2}1_{\Va_{\epsilon}}\ldots Q_{s_{k-1},s_{k}}1_{\Va_{\epsilon}}\right)
\nu(1_{\Va_{\epsilon}}Q_{s_{k+1},s_n}(H)).
 \end{array}$$
This yields the estimate
$$
1_{\Va_{\epsilon}}(x) ~Q_{s_0,s_n}(H)\geq 1_{\Va_{\epsilon}}(x) ~(\iota_{\epsilon}^-)^2~
(\iota_{\epsilon_1}^-\nu\left(\Va_{\epsilon_1}\right))^{k}~\nu(1_{\Va_{\epsilon}}Q_{s_{k+1},s_n}(H)).
$$
from which we check that
$$
 \begin{array}{l}
\displaystyle1_{\Va_{\epsilon}}(x) ~\left(\overline{\eta}_{s_0}1_{\Va_{\epsilon}(i_1)}Q_{s_0,s_1}\ldots 1_{\Va_{\epsilon}(i_n)}Q_{s_{n-1},s_n}(1)\right)/Q_{s_0,s_n}(H)(x)\\
\\
\displaystyle\leq  ({\iota_{\epsilon}}/{\epsilon})~ ~( \overline{\eta}(V)/\kappa^-_H(\nu_{\epsilon}))~\left(\epsilon~{(1\vee\Vert\Theta_{\tau}\Vert)}/{(\iota^-_{\epsilon_1}\nu\left(\Va_{\epsilon_1}\right))^2}\right)^{k/2}.~ \end{array}$$
The end of the proof of the left-hand side~assertion in (\ref{ref-H0-2}) now follows word-for-word the same lines of arguments as the proof of the right-hand side estimate in (\ref{ref-H0-2}), thus it is skipped. The proof of the lemma is now completed.
\end{proof}

\subsection{Proof of Lemma~\ref{lem-BG-V}}\label{app:lem-BG-V}

\begin{proof}
Observe that for any $f$ such that $\Vert f\Vert_{V}\leq 1$, we have
$$
(\Psi_{1/h}(\mu_1)-\Psi_{1/h}(\mu_2))(f)=\frac{1}{\mu_1(1/h)}~(\mu_1-\mu_2)(g)
$$
with the function
$$
g:=(1/h)(f-\Psi_{1/h}(\mu_2)(f))\in  \Ba_{V/h}(E).
$$
Note that
$$
\frac{g}{V/h}=\frac{f-\Psi_{1/h}(\mu_2)(f)}{V}=\frac{f}{V}-\frac{\Psi_{1/h}(\mu_2)(V{f}/{V})}{V}\Longrightarrow
\Vert g\Vert_{V/h}\leq 1+\frac{\mu_2(V/h)}{\mu_2(1/h)}.
$$
This ends the proof of the first assertion.
Now, for any $\Vert f\Vert_{V/h}\leq 1$  we have
$$
(\Psi_{h}(\mu_1)-\Psi_{h}(\mu_2))(f)=\frac{1}{\mu_1(h)}~(\mu_1-\mu_2)(g)
$$
with the function
$$
\begin{array}{l}
\displaystyle g:=h(f-\Psi_{h}(\mu_2)(f))\\
\\
\displaystyle\Longrightarrow 
\frac{\vert g\vert }{V}\leq 1+
\frac{h}{V}~\frac{\mu_2(V)}{\mu_2(h)}\leq 1+\frac{\mu_2(V)}{\mu_2(h)}\Longrightarrow\Vert g\Vert_V\leq 1+\frac{\mu_2(V)}{\mu_2(h)}.
\end{array}
$$
This ends the proof of the lemma.
\end{proof}

\subsection{Proof of Lemma~\ref{lem-beta-12}}\label{app:lem-beta-12}

\begin{proof}
We set $(s,t)=(1,2)$ and
$P=P_{1,2}$, and
 $V^{\rho}_i:=1/2+\rho V_i$, with $i\in \{1,2\}$ and $\rho \in ]0,1[$. We also consider the function
 $\Delta_{\rho}$ on $E_1^2$ defined for any $(x,y)\in E_1^2$ by
\begin{eqnarray*}
\Delta_{\rho}(x,y)&:=&\frac{\Vert \delta_xP-\delta_yP\Vert_{V^{\rho}_{2}}}{\Vert \delta_x-\delta_y\Vert_{V^{\rho}_{1}}}\\
&\leq& \frac{\Vert \delta_xP-\delta_yP\Vert_{\tiny tv}}{1+\rho(V_1(x)+V_1(y))}
+\frac{\rho (P(V_2)(x)+P(V_2)(y))}{1+\rho(V_1(x)+V_1(y))}.
\end{eqnarray*}
Using the left-hand side estimate in (\ref{PW-12}), we have
\begin{eqnarray*}
{P(V_2)(x)+P(V_2)(y)} & \leq & \epsilon(V_1(x)+V_1(y))+{2} \\
& = & (V_1(x)+V_1(y))\left(\epsilon + \frac{2}{(V_1(x)+V_1(y))} \right).
\end{eqnarray*}
When $V_1(x)+V_1(y)\geq r {>r_0\vee r_{\epsilon}}$ this yields the estimate
\begin{eqnarray*}
\Delta_{\rho}(x,y)
&\leq& \frac{1}{1+\rho(V_1(x)+V_1(y))}
+\frac{\rho (V_1(x)+V_1(y))}{1+\rho(V_1(x)+V_1(y))}~\left(\epsilon+\frac{2}{r}\right)\\
&=& 1-\left(1-\frac{1}{1+\rho((V_1(x)+V_1(y))}\right)\left(1-\left(\epsilon+\frac{2}{r}\right)\right)\leq 1-d^1_{\rho},
\end{eqnarray*}
with
\[
 {d^1_{\rho}}:=\left(1-\frac{1}{1+ {\rho r}}\right)\left(1-\left(\epsilon+\frac{2}{r}\right)\right).
\]
Recalling that $V_i\geq 1$ we have
$$
P(V_2)/V_1\leq 1+\epsilon.
$$
This implies that for any $(x,y)$ such that $ V_1(x)+V_1(y)\leq r$ we have
$$
\frac{\rho V_1(x) (P(V_2)(x)/V_1(x))+\rho V_1(y)~(P(V_2)(y)/V_1(y))}{1+\rho(V_1(x)+V_1(y))}\leq(1+\epsilon)~\frac{\rho r}{1+2\rho},
$$
 This yields the estimate
\begin{eqnarray*}
\Delta_{\rho}(x,y)
&\leq&1-d^2_{\rho}:= \frac{1-\alpha(r)}{1+2\rho} +\frac{\rho r}{1+2\rho}(1+\epsilon).
\end{eqnarray*}
 {Choosing $$
\rho=\rho(r):=\frac{1}{1+\epsilon}~\frac{\alpha(r)}{2r}$$  we have
$$
d^1_{\rho(r)}=\frac{\alpha(r)/2}{(1+\epsilon)/2+\alpha(r)/4}\left(\frac{1-\epsilon}{2}-\frac{1}{r}\right)\geq \frac{2}{5}~(1-\epsilon)~\left(1-\frac{r_{\epsilon}}{r}\right)~\frac{\alpha(r)}{2},
$$
and
$$
1-d^2_{\rho(r)}=\frac{1-\alpha(r)+\rho r(1+\epsilon)}{1+2\rho}=\frac{1-\alpha(r)/2}{1+2\rho}\leq 1-\frac{\alpha(r)}{2}.
$$
This yields the estimate
$$
(1-d^1_{\rho(r)})\vee (1-d^2_{\rho(r)})\leq  1-(1-\epsilon)~\left(1-\frac{r_{\epsilon}}{r}\right)~\frac{\alpha(r)}{5}.
$$}
This ends the proof of the lemma.
\end{proof}

\subsubsection{Proof of lemma~\ref{lem-eventually-compact}}\label{lem-eventually-compact-proof}
 \begin{proof}
 For any  $t\in \Ta$, we set $R_t^{K}(f):=1_{K}Q_t(f)$ so that
$$
Q_t=R_t^{K}+R_t^{K^c}
\quad
\mbox{\rm and}\quad
Q_{t+\tau}-Q^{K}_{\tau}Q_{t}=Q_{\tau}R^{K^c}_t+R^{K^c}_{\tau}R_t^{K}.
$$
This decomposition implies that
\begin{eqnarray*}
\vertiii{Q_{t+\tau}-Q^{K}_{\tau}Q_{t}}_V & \leq & \left(\vertiii{Q_{\tau}}_V+\vertiii{R^{K}_t}_V\right)~\vertiii{R^{K^c}_{\tau}}_V \\
& \leq & \left(\vertiii{Q_{\tau}}_V+\vertiii{Q_t}_V\right)~~\vertiii{R^{K^c}_{\tau}}_V.
\end{eqnarray*}
Also note that
$$
\vertiii{R^{K^c}_{\tau}}_V=\Vert 1_{K^c}~{Q_{\tau}(V)}/{V}\Vert.
$$
Thus, choosing a sequence of compact sets $K_n$ such that  
$$
1_{K_n^c}~Q_{\tau}(V)\leq  \frac{1}{n}~ V
$$
implies that
$$
\vertiii{Q_{t+\tau}-Q^{K_n}_{\tau}Q_{t}}_V
\leq \frac{1}{n}~ \left(\vertiii{Q_{\tau}}_V+\vertiii{Q_t}_V\right).
$$
Since the product operator $Q^{K_n}_{\tau}Q_{t}$ is compact, $Q_{t+\tau}$ is the limit in norm of compact operators, hence it is compact. This ends the proof of the lemma.
\end{proof}

\subsubsection{Proof of the compactness of (\ref{ref-continuous-Q}) }\label{ref-continuous-Q-proof}

For any collection of functions $(f_n)_{n\geq 0}\in\Ba_V(E)^{\NN}$ in the unit ball $\{f \in \Ba_V(E) : \Vert f\Vert_V\leq 1\}$  we have
$$
\sup_{n\geq 0}\Vert Q^K_{\tau}(f_n)\Vert_V\leq \Vert Q^K_{\tau}(V)/V\Vert.
$$
In addition, for any $(x,y)\in K^2$ we have
$$
\vert Q^K_{\tau}(f_n)(x)-Q^K_{\tau}(f_n)(y)\vert\leq \int~\vert q_{\tau}(x,z)-q_{\tau}(y,z)\vert~1_K(z)\nu_\tau(dz).
$$
Since $q_{\tau}$ is continuous on the compact set $(K\times K)$, it is uniformly continuous. Thus for any $\epsilon>0$, there is some $\delta>0$ such that $d(x,y)\leq \delta$ implies that for any $n\geq 0$ we have 
 $$
 \vert Q^K_{\tau}(f_n)(x)-Q^K_{\tau}(f_n)(y)\vert\leq\epsilon~\nu_\tau(K).
 $$
By the Arzela-Ascoli theorem, $Q^K_{\tau}(f_n)$ converges uniformly to a continuous limit extended by  $0$ outside the set $K$. Recalling that $V\geq 1$ is also converges on $\Ba_V(E)$. 
 $Q^K_{\tau}$ is an irreducible and compact operator on  $\Ba_V(E)$. This ends the proof of compactness of (\ref{ref-continuous-Q}).

\subsubsection*{Proof of (\ref{compact-countable-V})}\label{compact-countable-V-proof}
For any $\epsilon>0$ the set 
$
K_{\epsilon}:=\{Q_{\tau}(V)/V\geq \epsilon\}
$
is finite, and following the proof of the above lemma, we have
$$
\vertiii{Q_{t+\tau}-Q_{\tau}^{K_{\epsilon}}Q_{t}}_V\leq \epsilon~ \left(\vertiii{Q_{\tau}}_V+\vertiii{Q_t}_V\right).
$$
Since for any finite set $K$, the operator
$$
Q^{K}_{\tau}(f)(x)=1_{K}(x)~\sum_{y\in K}Q_{\tau}(x,y)~f(y)
$$
is bounded and has finite range, it is compact and we thus conclude that
$Q_{t}$ is compact for any $t\geq \tau$ as soon as $Q_{\tau}(V)/V\in \Ba_0(E)$. 

\textcolor{black}{The set $E$ can be written as the union $E=\cup_{n\geq 0}K_n$ of an increasing sequence of compact sets $K_n$.}
Whenever $Q_{t}$ is compact on $\Ba_V(E)$, up to a subsequence extraction the functions $Q_t(V1_{K^c_n})\in \Ba_V(E)$ forms a Cauchy sequence. Thus, for any $\epsilon>0$ there exists some $n_{\epsilon}\geq 1$ such that
for any $n_{\epsilon}\leq m\leq n$ we have
$$
\Vert Q_t(V 1_{K^c_n})-Q_t(V1_{K^c_m})\Vert_V=\Vert Q_{t}(V1_{K_n-K_m})\Vert_V\leq \epsilon
$$
and therefore
$$
\Vert Q_t(V 1_{K_m^c})/V\Vert\leq \epsilon.
$$
Now, assume that for any $\epsilon>0$ there exists some $n_{\epsilon}\geq 1$ such that
$$
\forall n\geq n_{\epsilon}\qquad\Vert Q_t(V 1_{K_n^c})/V\Vert\leq \epsilon.
$$
For any 
$\Vert f\Vert_V\leq 1$, we have
$$
\Vert Q_t(f)-Q_t(1_{K_n}f)\Vert_V\leq \Vert Q_t(V 1_{K_n^c})/V\Vert\leq \epsilon.
$$
Arguing as above, using the fact that $f\in \Ba_V(E)\mapsto 1_{K_n}f\in \Ba_V(E)$ is a finite range compact operator we conclude that $Q_t$ is compact on $\Ba_V(E)$. 
This ends the proof of (\ref{compact-countable-V}).

\subsubsection*{Acknowledgements}

Ajay Jasra was supported by KAUST baseline funding and Pierre Del Moral was partially supported by the ANR project QuAMProcs: ANR-19-CE40-0010.
We thank two referees and the editors for their comments, which have greatly improved the paper.

\end{document}